\documentclass[leqno]{amsart}

\usepackage{amssymb,latexsym}

\author[W.\thinspace{}Kim]{Wansu Kim}
\address{Wansu Kim\\%
Department of Mathematics\\%
South Kensington Campus\\%
Imperial College London\\%
London, SW7 2AZ\\%
United Kingdom}
\email{w.kim@imperial.ac.uk}

\usepackage[active]{srcltx}
\usepackage{amsmath}
\usepackage{amsfonts}

\usepackage{amsthm}
\usepackage{amscd}
\usepackage{mathrsfs}
\usepackage{graphicx}
\usepackage{psfrag}
\usepackage{calc}
\usepackage{url}
\usepackage[all]{xy}

\usepackage{egastyle}

\numberwithin{equation}{subsection}

\theoremstyle{plain}
\newtheorem{thm}[subsection]{Theorem}
\newtheorem*{thm*}{Theorem}

\newtheorem{thmsub}[equation]{Theorem}

\newtheorem{lem}[subsection]{Lemma}
\newtheorem{lemsub}[equation]{Lemma}
\newtheorem*{lem*}{Lemma}

\newtheorem{prop}[subsection]{Proposition}

\newtheorem{propsub}[equation]{Proposition}
\newtheorem*{prop*}{Proposition}

\newtheorem{cor}[subsection]{Corollary}
\newtheorem{corsub}[equation]{Corollary}
\newtheorem*{cor*}{Corollary}

\newtheorem{claimsub}[equation]{Claim}
\newtheorem*{claim*}{Claim}

\theoremstyle{definition}

\newtheorem*{defn*}{Definition}
\newtheorem{defnsub}[equation]{Definition}

\newtheorem{exasub}[equation]{Example}
\newtheorem*{exa*}{Example}

\theoremstyle{remark}
\newtheorem{rmk}[subsection]{Remark}
\newtheorem{rmksub}[equation]{Remark}
\newtheorem*{rmk*}{Remark}

\numberwithin{figure}{subsection}
\numberwithin{table}{subsection}

\newcounter{listnum}

\newcommand{\abs}[1]{\left|#1\right|}

\newcommand{\scri}{\mathscr I}

\DeclareMathOperator{\im}{im}

\DeclareMathOperator{\coker}{coker}

\newcommand{\Proj}{\mathbb{P}}

\DeclareMathOperator{\Mat}{Mat}

\DeclareMathOperator{\id}{id}
\DeclareMathOperator{\Id}{Id}

\newcommand{\e}{\mathbf{e}}

\DeclareFontEncoding{OT2}{}{} 

\newcommand{\nf}[1]{\underline{#1}}
\newcommand{\ol}[1]{\overline{#1}}
\newcommand{\wt}[1]{\widetilde{#1}}

\newcommand{\tim}{\!\cdot\!}
\newcommand{\geqs}{\geqslant}
\newcommand{\leqs}{\leqslant}

\newcommand{\et}{\text{\rm\'et}}
\DeclareMathOperator{\cyc}{cyc}
\newcommand{\cycchar}{\chi_{\cyc}}

\newcommand{\wh}[1]{\widehat{#1}}

\newcommand{\cL}{\mathcal{L}}

\newcommand{\M}{\mathcal{M}}
\newcommand{\MM}{\boldsymbol{\mathcal{M}}}

\newcommand{\disk}{\mathbf{\Delta}}

\DeclareMathOperator{\Sym}{Sym}

\newcommand{\llpar}{(\!(}
\newcommand{\rrpar}{)\!)}

\newcommand{\rep}{\mathfrak{R}}

\DeclareMathOperator{\Spec}{Spec}

\DeclareMathOperator{\Spf}{Spf}

\DeclareMathOperator{\res}{res}
\DeclareMathOperator{\Ker}{Ker}

\DeclareMathOperator{\sep}{sep}

\newcommand{\starr}{^\times}

\DeclareMathOperator{\Frac}{Frac}

\newcommand{\ra}{\rightarrow}
\newcommand{\xra}[1]{\xrightarrow{#1}}

\newcommand{\xlra}[1]{\stackrel{#1}{\longrightarrow}}
\newcommand{\hra}{\hookrightarrow}

\newcommand{\thra}{\twoheadrightarrow}

\newcommand{\la}{\leftarrow}

\newcommand{\riso}{\xrightarrow{\sim}}
\newcommand{\liso}{\stackrel{\sim}{\la}}

\newcommand{\self}[1]{{#1}\ra {#1}}
\newcommand{\fstr}[1]{\sig^*{#1}\ra {#1}}
\newcommand{\etfstr}[1]{\sig^*{#1}\riso {#1}}

\DeclareMathOperator{\cont}{cont}

\newcommand{\invlim}{\mathop{\varprojlim}\limits}

\newcommand{\set}[1]{\{#1\}}
\newcommand{\iv}{^{-1}}
\newcommand{\ivtd}[1]{\frac{1}{#1}}

\newcommand{\abcd}[4]{\left(
        \begin{smallmatrix}#1&#2\\#3&#4\end{smallmatrix}\right)}
\newcommand{\mthree}[9]{\left(
        \begin{matrix}#1&#2&#3\\#4&#5&#6\\#7&#8&#9
        \end{matrix}\right)}

\newcommand{\Eps}{\mathcal{E}}
\newcommand{\vphi}{\varphi}

\newcommand{\sig}{\sigma}
\newcommand{\Sig}{\mathfrak{S}}

\newcommand{\Q}{\mathbb{Q}}

\newcommand{\cD}{\mathcal{D}}

\newcommand{\PP}{\mathcal{P}}

\newcommand{\LT}{\mathcal{LT}}

\newcommand{\C}{\mathbb{C}}

\newcommand{\Sets}{(\mbox{\rm\bf Sets})}

\newcommand{\kbar}{\bar{k}}

\newcommand{\Z}{\mathbb{Z}}

\newcommand{\F}{\mathbb{F}}

\newcommand{\Qp}{\Q_p}

\newcommand{\Kbar}{\overline{K}}

\newcommand{\Zp}{\Z_p}
\newcommand{\Fp}{\F_p}
\newcommand{\Fq}{\F_q}

\DeclareMathOperator{\Ad}{Ad}

\DeclareMathOperator{\nilp}{nilp}
\DeclareMathOperator{\unip}{unip}

\newcommand{\OO}{\mathcal{O}}
\newcommand{\fo}{\mathfrak{o}}
\newcommand{\ep}{\epsilon}
\newcommand{\A}{\mathcal{A}}

\DeclareMathOperator{\ord}{ord}
\DeclareMathOperator{\ssing}{ss}
\DeclareMathOperator{\GL}{GL}

\DeclareMathOperator{\Gal}{Gal}
\newcommand{\gal}{\boldsymbol{\mathcal{G}}}
\newcommand{\GK}{\gal_K}
\newcommand{\GKinfty}{\gal_{K_\infty}}
\newcommand{\IK}{I_K}
\newcommand{\IKinfty}{I_{K_\infty}}

\newcommand{\m}{\mathfrak{m}}

\newcommand{\gM}{\mathfrak{M}}
\newcommand{\gN}{\mathfrak{N}}

\newcommand{\comp}[1]{\ \wh{}\!\!_{#1}\,}

\DeclareMathOperator{\ur}{ur}

\DeclareMathOperator{\fl}{fl}

\DeclareMathOperator{\Hom}{Hom}

\DeclareMathOperator{\End}{End}

\DeclareMathOperator{\HH}{H}
\newcommand{\coh}[1]{\HH^{#1}}

\DeclareMathOperator{\Tor}{Tor}

\DeclareMathOperator{\Fil}{Fil}

\DeclareMathOperator{\Mod}{Mod}
\DeclareMathOperator{\ModFI}{ModFI}
\newcommand{\Modh}{(\Mod/\Sig)^{\leqslant h}}
\newcommand{\ModFIh}[1]{(\ModFI/\Sig)_{#1}^{\leqslant h}}
\newcommand{\ModFIBT}[1]{(\ModFI/\Sig)_{#1}^{\leqslant1}}
\newcommand{\ModFIet}[1]{(\ModFI/\fo_\Eps)_{#1}^{\et}}
\newcommand{\Th}{\underline{T}^{\leqs h}}

\DeclareMathOperator{\cris}{cris}
\newcommand{\Acris}{A_{\cris}}

\newcommand{\Bcris}{B_{\cris}}
\DeclareMathOperator{\st}{st}
\DeclareMathOperator{\qst}{qst}
\DeclareMathOperator{\dR}{dR}

\newcommand{\Asthat}{\widehat{A}_{\st}}
\newcommand{\Bst}{B_{\st}}
\newcommand{\BdR}{B_{\dR}}

\DeclareMathOperator{\tor}{tor}
\DeclareMathOperator{\free}{free}
\newcommand{\MF}{\mathcal{MF}}

\newcommand{\waMFphih}[1]{\MF(\vphi)_{#1}^{wa,[0,h]}}
\newcommand{\waMFphiNh}[1]{\MF(\vphi,N)_{#1}^{wa,[0,h]}}

\newcommand{\waMFphiBT}[1]{\MF(\vphi)_{#1}^{wa,[0,1]}}

\newcommand{\phimodh}[1]{\nf{\Mod}_{#1}(\varphi)^{\leqs h}}
\newcommand{\phimodBT}[1]{\nf{\Mod}_{#1}(\varphi)^{\leqs 1}}

\newcommand{\etphimod}[1]{\nf{\Mod}_{#1}^{\et}(\varphi)}
\newcommand{\fretphimod}[1]{\nf{\Mod}_{#1}^{\et,\free}(\varphi)}
\newcommand{\toretphimod}[1]{\nf{\Mod}_{#1}^{\et,\tor}(\varphi)}

\DeclareMathOperator{\Rep}{Rep}

\newcommand{\prep}{\Rep_{\Zp}}

\newcommand{\art}[1]{\mathfrak{AR}_{#1}}
\newcommand{\arhat}[1]{\widehat{\mathfrak{AR}}_{#1}}
\newcommand{\aug}[1]{\mathfrak{Aug}_{#1}}

\newcommand{\bfv}{\mathbf{v}}

\newcommand{\GR}{\mathscr{GR}}

\newcommand{\GRh}{\mathscr{GR}^{\leqs h}}

\newcommand{\GRv}{\mathscr{GR}^{\bfv}}

\title[Galois~Deformation~Theory]{Galois deformation theory for norm fields \\and flat deformation rings}

\begin{document} 

\begin{abstract}
Let $K$ be a finite extension of $\mathbb{Q}_p$, and choose a uniformizer $\pi\in K$, and put $K_\infty:=K(\sqrt[p^\infty]{\pi})$. We introduce a new technique using restriction to $\Gal(\ol K/K_\infty)$ to study flat deformation rings. We show the existence of deformation rings for $\Gal(\ol K/K_\infty)$-representations ``of height $\leqslant h$'' for any positive integer $h$, and we use them to give a variant of Kisin's proof of connected component analysis of a certain flat deformation rings, which was used to prove Kisin's modularity lifting theorem for potentially Barsotti-Tate representations. Our proof does not use the classification of finite flat group schemes, so it avoids Zink's theory of windows and displays when $p=2$.

This $\Gal(\ol K/K_\infty)$-deformation theory has a good analogue in positive characteristics analogue of crystalline representations in the sense of Genestier-Lafforgue. In particular, we obtain a positive characteristic analogue of crystalline deformation rings, and can analyze their local structure. 
\end{abstract}

\maketitle
\tableofcontents
\section{Introduction}
Since the pioneering work of Wiles on the modularity of semi-stable elliptic curves over $\Q$, there has been huge progress on modularity lifting. Notably, Kisin  \cite{Kisin:ModuliGpSch, Kisin:2adicBT} (later improved by Gee \cite{Gee:MLTwt2HMF}) proved a very general modularity lifting theorem for potentially Barsotti-Tate representations, which had enormous impacts on this subject. (For the precise statement of the theorem, see the aforementioned references.)

One of the numerous noble innovations that appeared in Kisin's result is his improvement of Taylor-Wiles patching argument. The original patching argument required relevant local deformation rings to be formally smooth, which is a very strong requirement. Under Kisin's improved patching, we only need to show that he generic fiber of local deformation rings are formally smooth with correct dimension, and we need to have some control of their connected components. (See  \cite[Corollary~1.4]{Kisin:ModularityGeomGalRep} for the list of sufficient conditions on local deformation rings to prove modularity lifting.) It turns out that the most difficult part among them (and the hurdle to proving modularity lifting for more general classes of $p$-adic Galois representations)  is to ``control'' the connected components of certain $p$-adic deformation rings at places over $p$. (The relevant local deformation rings here are ``flat deformation rings''.)

The main purpose of this paper is to give another proof of the following theorem of Kisin. Let us fix some notations. Let $K$ be a finite extension of $\Qp$, $\F$ a finite field of characteristic $p$, and $\fo$ a complete discrete valuation ring with residue field $\F$. Let $\bar\rho:\GK\ra\GL_2(\F)$ be a continuous representation, and let $R^\Box$ be the framed deformation $\fo$-algebra of $\bar\rho$ (whose existence was shown by Mazur\cite{Mazur:DeformingGalReps}). Let $R^{\Box,\bfv}_{\cris}$ be the unique $p^\infty$-torsion free quotient  of $R^\Box$ whose $A$-points classifies crystalline lifts of $\bar\rho$ with ``Hodge type $(0,1)$'' for any finite $\Qp$-algebra $A$. (See \cite[\S2]{kisin:pstDefor} for the proof. The precise definition could be found in Theorem~\ref{thm:CrysStDeforRing} and \S\ref{subsec:HodgeType} for the case we need.) By \cite[Corollary~2.2.6]{kisin:fcrys} and \cite[Proposition~2.3.1]{raynaud:GpSch}, $R^{\Box,\bfv}_{\cris}$ differs only by $p^\infty$-torsion from the flat framed deformation ring of $\bar\rho$ with the inertia action on the determinant given by the $p$-adic cyclotomic character..
\begin{thm}[Kisin; Gee, Imai]\label{thm:KisinConndComp}
For finite local $\Qp$-algebras $A$ and $A'$, consider maps $\xi:R^{\Box,\bfv}_{\cris} \ra A$ and $\xi':R^{\Box,\bfv}_{\cris} \ra A'$. Let $\rho_\xi$ and $\rho_{\xi'}$ denote the lifts of $\bar\rho$ corresponding to $\xi$ and $\xi'$, respectively. 
\begin{enumerate}
\item\label{thm:KisinConndComp:pOdd}
Assume that $p>2$. Then $\xi$ and $\xi'$ are supported on the same connected component of $\Spec R^{\Box,\bfv}_{\cris}[\ivtd p]$ if and only if either both $\rho_\xi$ and $\rho_{\xi'}$ do not admit a non-zero unramified quotient (i.e. non-ordinary) or both $\rho_\xi$ and $\rho_{\xi'}$ admit a rank-$1$ unramified quotient which lift the same (mod $p$ unramified) character. In particular,  $\Spec R^{\Box,\bfv}_{\cris}[\ivtd p]$ has at most one connected component (which is geometrically connected) consisting of non-ordinary lifts, and at most one connected component (which is geometrically connected) consisting of ordinary lifts unless $\bar\rho \sim \psi_1\oplus\psi_2$ where $\psi_1$ and $\psi_2$ are distinct unramified characters, in which case there exist two connected components consisting of ordinary lifts (and each of them is geometrically connected). 
\item
\label{thm:KisinConndComp:pEven}
Assume that $p=2$. Then $\xi$ and $\xi'$ are supported on the same connected component of $\Spec R^{\Box,\bfv}_{\cris}[\ivtd p]$ if both $\rho_\xi$ and $\rho_{\xi'}$ do not admit a non-zero unramified quotient (i.e. non-ordinary).
\end{enumerate}
\end{thm}
In order to prove the $2$-adic modularity lifting theorem for potentially Barsotti-Tate representations, one needs the full statement of  Theorem~\ref{thm:KisinConndComp}\eqref{thm:KisinConndComp:pOdd} for $p=2$, not just Theorem~\ref{thm:KisinConndComp}\eqref{thm:KisinConndComp:pEven}, when $\bar\rho\sim\abcd{1}{0}{0}{1}$. But in fact, the idea of Khare-Wintenberger  \cite[3.2.5]{Khare-Wintenberger:SerreII} takes care of the ``ordinary components'' of $\Spec R^{\Box,\bfv}_{\cris}[\ivtd p]$ even when $p=2$. (See also  \cite[\S2.4]{Kisin:2adicBT}.)

Let us sketch the original idea of proof due to Kisin. When $p>2$, he constructed a kind of ``resolution'' of flat deformation rings using the classification of finite flat group schemes\footnote{See Theorem~\ref{thm:classifgpsch} for the statement. In fact, the construction of the resolved deformation space could be carried out using a slightly weaker statement.}. The statement of Theorem~\ref{thm:KisinConndComp}\eqref{thm:KisinConndComp:pOdd} concerning ordinary components follows relatively easily; see  \cite[Proposition~2.5.15]{Kisin:ModuliGpSch}, and the proof of the statements about non-ordinary components is reduced to constructing chains of rational curves in a certain subscheme of an affine grassmannian. (See \cite[\S2.5]{Kisin:ModuliGpSch}.) This ``affine grassmannian computation'' (which completes the proof of the geometric connectedness of the non-ordinary locus when $p>2$) was done in \emph{loc.~cit.} when the residue field of $K$ is $\Fp$, in \cite{Gee:MLTwt2HMF} assuming  $\bar\rho\sim\abcd{1}{0}{0}{1}$, and in \cite{Imai:ConndComp} for the general case.\footnote{Note that  \cite{Gee:MLTwt2HMF} is sufficient for application to the modularity lifting.}

As remarked earlier, the ordinary components of $2$-adic flat deformation rings can be controlled by some Kummer theory argument (which appeared in \cite[3.2.5]{Khare-Wintenberger:SerreII} or \cite[\S2.4]{Kisin:2adicBT}), so essentially it is left to show Theorem~\ref{thm:KisinConndComp}\eqref{thm:KisinConndComp:pEven}. 
The main new ingredient in generalizing this to the case $p=2$ is indeed proving the classification of \emph{connected} finite flat group schemes over a $2$-adic base  \cite[\S1]{Kisin:2adicBT},  which is based upon Zink's theory of windows and displays \cite{Zink:WindowsDisplayProgrMath195, Zink:DisplayFormalGpAsterisq278}. (Note that Theorem~\ref{thm:KisinConndComp}\eqref{thm:KisinConndComp:pEven} only concerns about non-ordinary components, so it suffices to work with ``connected finite flat group scheme models''.) The rest of the proof, including \cite{Gee:MLTwt2HMF, Imai:ConndComp} which are added later, goes through without major modification, so we obtain   Theorem~\ref{thm:KisinConndComp}\eqref{thm:KisinConndComp:pEven}. 

The purpose of this paper is to give another proof of Theorem~\ref{thm:KisinConndComp} that does not use (the classification of) finite flat group schemes, and instead relies more on ``linear-algebraic'' tools from $p$-adic Hodge Theory. 
Indeed, the only possible place where we use finite flat group schemes in our proof is Fontaine's ramification estimate of mod $p$ finite flat representations \cite{Fontaine:RamifEstimate}. As a consequence, we can remove the difficulty in proving the classification of connected finite flat group schemes (especially, over a $2$-adic base) from the proof of modularity lifting theorem for potentially Barsotti-Tate representations. We remind that the full proof of Serre's modularity conjecture \cite{Khare-Wintenberger:SerreI, Khare-Wintenberger:SerreII} uses the $2$-adic modularity lifting theorem for potentially Barsotti-Tate representations. Another motivation for removing finite flat group schemes from the proof is that it would be a sensible first step for the modularity lifting theorem with higher weight where no reasonable analogue of finite flat group schemes for torsion representations is available. (We do not claim, however, that our proof gives any indication towards this generalization.)

We point out that our technique is motivated by the author's study of positive characteristic analogue of crystalline deformation rings (using the  theory of Genestier-Lafforgue  \cite{Genestier-Lafforgue:FontaineEqChar} and Hartl \cite{hartl:period, hartl:dictionary}). We include a section (\S\ref{sec:GenLaffHartl}) to sketch this positive characteristic deformation theory.

\subsection{Structure and overview of the paper}
Let $K$ be a finite extension of $\Qp$, $K_\infty = K(\sqrt[p^\infty]{\pi})$ for a chosen uniformizer $\pi\in K$, $\GK:=\Gal(\ol K/K)$ and $\GKinfty:=\Gal(\ol K/K_\infty)$.
The main tool is the deformation theory of ``$\GKinfty$-representations of height $\leqs 1$''  -- in fact, we show  that there exists a complete noetherian ring which classifies all ``framed deformations of height $\leqs 1$''. See Theorem~\ref{thm:Representability} for the precise statement. The surprising aspect of this theorem is that $\GKinfty$ does not satisfy the cohomological finiteness condition that ensures the finiteness of the tangent space of the ``unrestricted'' deformation functor. (See \S\ref{subsec:RepbilitySetup} for more discussions.) The main content of the theorem is that the deformation condition of ``being of height $\leqs 1$'' ensures the finiteness of the tangent space. In \S\ref{sec:KisinThy} we give a summary of Kisin theory \cite{kisin:fcrys} to motivate the study of ``$\GKinfty$-deformations of height $\leqs 1$'', as well as to state theorems from \emph{loc.~cit.} that will be constantly used throughout the paper. The existence of framed deformation ring of height $\leqs1$ (Theorem~\ref{thm:Representability}) is proved in \S\ref{sec:Representability}.

Another main ingredient of the proof is to compare the $\GKinfty$-deformation ring of height $\leqs1$ and the crystalline deformation ring with Hodge-Tate weights in $[0,1]$ (or flat deformation ring, if you wish). We show that the natural map between them given by restriction of $\GK$-action to $\GKinfty$ is an isomorphism on the generic fibers (Theorem~\ref{thm:GabberDefor}). In \S\ref{sec:ApplBT} we use Gabber's result \cite[Appendix]{Kisin:ModularityGeomGalRep} to the verification of a certain statement about $\GKinfty$-stable $\Zp$-lattice in a crystalline $\GK$-representation (Proposition~\ref{prop:Claim1}), and in \S\ref{sec:IntHT} we prove this statement using strongly divisible $S$-modules. Our use of strongly divisible $S$-modules is limited to constructing a $\GK$-stable $\Zp$-lattice in a crystalline representation, and we will not use the subtle result of associating a strongly divisible $S$-module to a $\GK$-stable $\Zp$-lattice.

Theorem~\ref{thm:GabberDefor} finally reduces the proof of Theorem~\ref{thm:KisinConndComp} to the proof of the analogous (geometric) connectedness result for some $\GKinfty$-deformation ring of height $\leqs 1$ (Proposition~\ref{prop:ConndOrdNonOrd}). This is done in \S\ref{sec:GenFibers}. The main idea is to ``resolve'' a $\GKinfty$-deformation ring using ``moduli of $\Sig$-modules''. The linear-algebraic structure of the argument is very similar to that of Kisin's setting (i.e. resolution of a flat deformation ring via ``moduli of finite flat group schemes''), and most of the linear-algebraic arguments (including the affine grassmannian computation in \cite[\S2.5]{Kisin:ModuliGpSch}, \cite{Gee:MLTwt2HMF}, and \cite{Imai:ConndComp}) carry over to our setting.

The results proved  in \S\ref{sec:Representability} and \S\ref{sec:GenFibers} are inspired by the author's study of theor analogues in positive characteristic. As this positive characteristic result is of separate interest, we sketch this positive characteristic deformation theory and explain the analogy with the ($p$-adic) $\GKinfty$-deformation theory in \S\ref{sec:GenLaffHartl}.

\subsection*{Acknowledgement}
The author deeply thanks his thesis supervisor Brian Conrad for his guidance. The author especially appreciates his careful listening of my results and numerous helpful comments. The author thanks Tong Liu for his advice on the proof of Proposition~\ref{prop:Claim1} when $p=2$.

\section{Review of  Kisin theory}\label{sec:KisinThy}
Let $k$ be a finite extension of $\Fp$,\footnote{Most results in \S\ref{sec:KisinThy} holds when $k$ is perfect. But we need $k$ to be finite for the existence of (framed) deformation rings in the category of complete local noetherian rings.}  $W(k)$ its ring of Witt vectors, and $K_0:=W(k)[\ivtd p]$. Let $K$ be a finite totally ramified extension of $K_0$ and let us fix its algebraic closure $\Kbar$. We \emph{fix} a uniformizer $\pi\in K$. and choose $\pi^{(n)}\in\Kbar$ so that $(\pi^{(n+1)})^p = \pi^{(n)}$ and $\pi^{(0)}=\pi$. Put $K_\infty:= \bigcup_n K(\pi^{(n)})$, $\gal_K:=\Gal(\Kbar/K)$, and $\gal_{K_\infty}:=\Gal(\Kbar/K_\infty)$. 

Kisin \cite{kisin:fcrys} gave a new classification of $p$-adic crystalline $\GK$-representations using a certain kind of Frobenius modules which encode information of their restrictions to $\GKinfty$. (In particular, he showed that restricting the $\GK$-action to $\GKinfty$ defines a fully faithful functor from the category of $p$-adic crystalline $\GK$-representations to the category of $p$-adic $\GKinfty$-representations.) Furthermore, such Frobenius modules come equipped with a natural ``integral'' structure which corresponds to$\GKinfty$-stable $\Zp$-lattices. In this section we recall main results from \cite{kisin:fcrys} and record some consequences to $p^\infty$-torsion $\GKinfty$-representations which will be needed later. Note that this section does not contain any original result.

\subsection*{Basic definitions and conventions}
Let  consider a ring $R$ equipped with an endomorphism $\sig:\self R$. (We will often assume that $\sig$ is finite flat.) By \emph{$(\vphi,R)$-module} (often abbreviated as a \emph{$\vphi$-module}, if $R$ is understood), we mean a finitely presented $R$-module $M$ together with an $R$-linear morphism $\vphi_M:\sig^*M\ra M$, where $\sig^*$ denotes the scalar extension by $\sig$. A \emph{morphism} between to $(\vphi,R)$-modules is a $\vphi$-compatible $R$-linear map.

If $R'$ is another ring equipped with a ring endomorphism $\sig':\self{R'}$ and $f:R\ra R'$ is a ring morphism such that $f\circ\sig = \sig'\circ f$, then for any $(\vphi,R)$-module $M$ the scalar extension $M\otimes_{R,f}R'$ has a natural structure of $(\vphi,R')$-module.

\subsection{Fontaine's theory of \'etale $\vphi$-modules}
Let $\Sig:=W(k)[[u]]$ where $u$ is a formal variable. Let $\fo_\Eps$ be the $p$-adic completion of $\Sig[\ivtd u]$, and $\Eps:=\fo_\Eps[\ivtd p]$ the field of fractions. Note that $\fo_\Eps$ is a complete discrete valuation ring with uniformiser $p$ and $\fo_\Eps/(p)\cong k\llpar u \rrpar$. (Note that we should view the residue field $k\llpar u \rrpar$ as the norm field for the extension $K_\infty /K$. See \cite{wintenberger:NormFiels} for more details.)
We extend the Witt vectors Frobenius to $\Sig$, $\fo_\Eps$, and $\Eps$ by sending $u$ to $u^p$, and denote them by $\sig$. (We write $\sig_\Sig$ instead, if we need to specify that it's an endomoprhism on $\Sig$, for example.) Note that $\sig$ is finite and flat. We denote by $\sig^*(\cdot)$ the scalar extension by $\sig$.

\begin{defnsub}
An \emph{\'etale $(\vphi,\fo_\Eps)$-module} (or simply, an  \emph{\'etale $\vphi$-module}) is a finitely generated $\fo_\Eps$-module $M$ equipped with an $\fo_\Eps$-linear isomorphism $\vphi_M:\etfstr M$. We say an \'etale $\vphi$-module $M$ is \emph{free} (respectively, \emph{torsion}) if the underlying $\fo_\Eps$-module is free (respectively, $p^\infty$-torsion).

We let $\etphimod{\fo_\Eps}$ denote the category of \'etale $(\vphi,\fo_\Eps)$-modules with $\vphi$-compatible $\fo_\Eps$-linear morphisms, and let $\fretphimod{\fo_\Eps}$ and $\toretphimod{\fo_\Eps}$ respectively denote the full subcategories of free and torsion \'etale $\vphi$-modules. Let $\etphimod{\fo_\Eps}[\ivtd p]$ denote the ``isogeny category'' for  $\fretphimod{\fo_\Eps}$; i.e. the category defined by formally inverting the multiplication by $p$.
\end{defnsub}

There exist natural notions of subquotient, direct sum, $\otimes$-product, and internal hom for \'etale $\vphi$-modules. We respectively define the duals for free and torsion \'etale $\vphi$-modules using the $\fo_\Eps$-linear duals and the Pontryagin duals with the induced $\vphi$-structures.

 Let $\wh\fo_{\Eps^{\ur}}$ denote the $p$-adic completion of strict henselisation of $\fo_\Eps$. We define $\sig$ on $\wh\fo_{\Eps^{\ur}}$ to be the unique lift of the $p$th power map on the residue field which extends $\sig$ on $\fo_\Eps$, and the $\GKinfty$-action on   $\wh\fo_{\Eps^{\ur}}$ is the one that uniquely lifts the natural $\GKinfty$-action on $k\llpar u \rrpar^{\sep}$ via the norm field isomorphism\footnote{See \cite{wintenberger:NormFiels} for more details on the norm field isomorphism.} $\GKinfty\cong\Gal(k\llpar u \rrpar^{\sep}/k\llpar u \rrpar)$ and fixes $\fo_\Eps$. The existence and uniqueness of $\sig$ and the $\GKinfty$ action on $\wh\fo_{\Eps^{\ur}}$ are the consequence of the universal property of strict henselization.
 
 Let $\prep(\GKinfty)$ denote the category of finitely generated $\Zp$-modules equipped with a continuous $\GKinfty$-action. 
 Now we make the following definitions:
 \begin{subequations}
 \label{eqn:EtPhiModEqCat}
 \begin{equation}\label{eqn:EtPhiModTEps}
 \nf T_\Eps (M):=(M\otimes_{\fo_\Eps}\wh\fo_{\Eps^{\ur}})^{\vphi=1} \qquad \textrm{for }M\in\etphimod{\fo_\Eps},\\
 \end{equation}
viewed as a $\GKinfty$-module via its natural action on $\wh\fo_{\Eps^{\ur}}$, and
 \begin{equation}\label{eqn:EtPhiModDEps}
 \nf D_\Eps (T):=(T\otimes_{\fo_\Eps}\wh\fo_{\Eps^{\ur}})^{\GKinfty} \qquad \textrm{for }T\in\prep(\GKinfty),
\end{equation}
equipped with the $\vphi$-structure induced from the natural one on $\wh\fo_{\Eps^{\ur}}$.
 \end{subequations}
 
\begin{thmsub}[Fontaine {\cite[\S{A}\,1.2]{fontaine:grothfest}}]\label{thm:FontaineEtPhiMod}
The assignments $\nf T_\Eps$ and $ \nf D_\Eps$ define quasi-inverse exact equivalences of $\otimes$-categories between $\etphimod{\fo_\Eps}$ and $\prep(\GKinfty)$. An \'etale $\vphi$-module $M$ is free of $\fo_\Eps$-rank $r$ (respectively, torsion of $\fo_\Eps$-length $r$) if and only if $\nf T_\Eps(M)$ is free of $\Zp$-rank $r$ (respectively, torsion of $\Zp$-length $r$). Furthermore the functors $\nf T_\Eps$ and $ \nf D_\Eps$ commute with suitable duality when restricted to free (respectively, torsion) objects.
\end{thmsub}

\begin{rmksub}\label{rmk:FontaineEtPhiModIsog}
Since any $\Qp$-representation of $\GKinfty$ has a $\GKinfty$-stable $\Zp$-lattice, Theorem~\ref{thm:FontaineEtPhiMod} implies that the functors $\nf T_\Eps$ and $ \nf D_\Eps$ induce quasi-inverse equivalence of categories between the category of $\Qp$-representations of $\GKinfty$ and the ``isogeny category'' of  \'etale $\vphi$-modules (i.e., the category of $(\vphi,\Eps)$-modules which are the form $M[\ivtd p]$ for some \'etale $\vphi$-module $M$ free over $\fo_\Eps$). 
\end{rmksub}

\subsection{Kisin's ``integral $p$-adic Hodge theory''}\label{subsec:KisinIntPAdicHT} We first introduce a new class of semi-linear algebra objects which classify crystalline $\GK$-representations whose ``integral structure'' classifies $\GKinfty$-stable $\Zp$-lattices.

Let $\PP(u)\in W(k)[u]$ be the Eisenstein polynomial with $\PP(\pi)=0$ normalized\footnote{The usual notation for the Eisenstein polynomial is $E(u)$ but we chose to use a different letter to minimize confusion with a finite extension $E$ of $\Qp$. See Example~\ref{exa:repfinht}\eqref{exa:repfinht:cycchar} for our choice of the normalization $\PP(0)=p$.}  so that the constant term is $p$. We view $\PP(u)$ as an element of $\Sig:=W(k)[[u]]$.

\begin{defnsub}
For a non-negative integer $h$, a \emph{$(\vphi,\Sig)$-module of $\PP$-height $\leqs h$} (or simply, a  \emph{$\vphi$-module of height $\leqs h$} if $\Sig$ and $\PP(u)$ are understood) is a finite free $\Sig$-module $\gM$ equipped with an $\Sig$-linear morphism $\vphi_\gM:\fstr \gM$ such that $\coker (\vphi_\gM)$ is killed by $\PP(u)^h$. 

We let $\phimodh\Sig$ denote the category of $\vphi$-modules of height $\leqs h$ with $\vphi$-compatible $\Sig$-linear morphisms, and let $\phimodh\Sig[\ivtd p]$ denote the ``isogeny category'' for $\phimodh\Sig$; i.e. the category defined by formally inverting the multiplication by $p$.
\end{defnsub}

There exist natural notions of  subquotient, direct sum, and $\otimes$-product for $\vphi$-modules of height $\leqs h$. Note that $\PP(u)$ is a unit in $\fo_\Eps$, so for any $\gM\in\phimodh\Sig$ the scalar extension $\gM\otimes_\Sig \fo_\Eps$ (with $\vphi_\gM\otimes\id_{\fo_\Eps}$) is an \'etale $\vphi$-module. (In particular, it follows that $\vphi_\gM$ is \emph{injective}.) We define a functor $\nf T^{\leqs h}_\Sig:\phimodh\Sig \ra \prep(\GKinfty)$ as follows:
\begin{equation}\label{eqn:TSigLattice}
\nf T^{\leqs h}_\Sig(\gM):=\nf T_\Eps(\gM\otimes_\Sig \fo_\Eps)(h), \textrm{ for any }\gM\in\phimodh\Sig,
\end{equation}
where $\nf T_\Eps$ is as defined in \eqref{eqn:EtPhiModTEps} and $T(h)$ for some $T\in\prep(\GKinfty)$ denotes the ``Tate twist''; i.e., twisting the $\GKinfty$-action on $T$ by $\cycchar^h|_{\GKinfty}$. 

\begin{thmsub}[Kisin]\label{thm:main}\hfill
\begin{enumerate}
\item
\label{thm:main:fullfth}The functor $\nf T^{\leqs h}_\Sig:\phimodh\Sig \ra \prep(\GKinfty)$ is fully faithful.
\item
\label{thm:main:lattices} If $V = \nf T^{\leqs h}_\Sig(\gM)[\ivtd p]$ for some $\gM\in\phimodh\Sig$, then for any $\GKinfty$-stable $\Zp$-lattice $T'\subset V$ there exists a $\vphi$-module $\gM'$ of height $\leqs h$ such that $T'\cong\nf T^{\leqs h}_\Sig(\gM')$.
\end{enumerate}
\end{thmsub}
\begin{proof}
The full faithfulness of $\nf T^{\leqs h}_\Sig$ is a restatement of \cite[Prop~2.1.12]{kisin:fcrys} using Theorem~\ref{thm:FontaineEtPhiMod}. For the second claim of the statement, one can repeat the proof of \cite[Lemma~2.1.15]{kisin:fcrys}, keeping track of the $\PP$-height bound. (See \cite[Prop~5.2.9]{Kim:Thesis} for more details.)
\end{proof}

Let us discuss how $\vphi$-modules of height $\leqs h$ are related to $p$-adic Hodge theory. Let $\waMFphih{K}$ and $\waMFphiNh{K}$ respectively denote the categories of weakly admissible filtered $\vphi$-modules and $(\vphi,N)$-modules with the associated grading concentrated in degrees $[0,h]$, and let $\Rep_{\cris}^{[0,h]}(\GK)$ and $\Rep_{\st}^{[0,h]}(\GK)$ respectively denote the categories of crystalline and semistable $p$-adic $\GK$-representations with Hodge-Tate weights\footnote{For us, the $p$-adic cyclotomic character has Hodge-Tate weight $1$.} in $[0,h]$. (See \cite{fontaine:Asterisque223ExpIII} for the definitions.)  We view $\waMFphih{K}$ as a full subcategory of $\waMFphiNh{K}$ by ``setting $N=0$''. By the fundamental theorem of Colmez-Fontaine \cite{colmez-fontaine}, we have an equivalence of categories $\nf V_{\st}(h):\waMFphiNh K \riso \Rep_{\st}^{[0,h]}(\GK);\ D\mapsto \nf V_{\st}(D)(h)$ which restricts to an equivalence of categories $\nf V_{\cris}(h):\waMFphih K \riso \Rep_{\cris}^{[0,h]}(\GK);\ D\mapsto \nf V_{\cris}(D)(h)$, where $\nf V_{\st}$ and $\nf V_{\cris}$ are covariant functors defined by Fontaine \cite{fontaine:Asterisque223ExpIII}.
\begin{thmsub}[Kisin]\label{thm:KisinIntPAdicHT}\hfill
\begin{enumerate}
\item 
\label{thm:KisinIntPAdicHT:st}There exists a diagram of functors
\begin{equation}\label{eqn:KisinIntPAdicHT}
\xymatrix{
\waMFphih K \ar@{^{(}->}[r] \ar[d]^{\cong}_{\nf V_{\cris}(h)} & \waMFphiNh K \ar[r] \ar[d]^{\cong}_{\nf V_{\st}(h)} & \phimodh\Sig[1/p] \ar@{^{(}->}[d]_{\nf T_\Sig^{\leqs h}}\\
\Rep_{\cris}^{[0,h]}(\GK) \ar@{^{(}->}[r] & \Rep_{\st}^{[0,h]}(\GK) \ar[r]^{\res} & \Rep_{\Qp}(\GKinfty)
}\end{equation}
which commutes up to equivalences, where $\res$ is defined by restricting the $\GK$-action to $\GKinfty$. 
\item 
\label{thm:KisinIntPAdicHT:crys}The composite $\waMFphih K \ra  \phimodh\Sig[\ivtd p]$ of arrows in the top row of diagram \eqref{eqn:KisinIntPAdicHT} is fully faithful and its essential image contains all the rank-$1$ objects. 
\item 
\label{thm:KisinIntPAdicHT:BT} If $h=1$ then the functor $\waMFphiBT K \ra  \phimodBT\Sig[\ivtd p]$ as in \eqref{thm:KisinIntPAdicHT:crys} is an equivalence of categories.
\end{enumerate}
\end{thmsub}
Note the following marvelous consequence of the theorem: any $p$-adic crystalline $\GK$-representation is uniquely determined by its restriction to $\GKinfty$ up to isomorphism.
\begin{defnsub}\label{def:hthLattice}
A finite free $\Zp$-module $T$ with continuous $\GKinfty$-action (which will be called a ``$\Zp$-lattice $\GKinfty$-representation'' from now on) is  \emph{of height $\leqs h$} if there exists $\gM\in\phimodh\Sig$ such that $T\cong \nf T^{\leqs h}_\Sig(\gM)$; or equivalently, the \'etale $\vphi$-module $\nf D_\Eps(T(-h))$ admits a (necessarily unique) $\vphi$-stable $\Sig$-lattice $\gM \in \phimodh\Sig$.

A $p$-adic $\GKinfty$-representation $V$ is  \emph{of height $\leqs h$} if there exists a $\GKinfty$-stable $\Zp$-lattice which is of height $\leqs h$ (or equivalently by Theorem~\ref{thm:main}\eqref{thm:main:lattices}, if any  $\GKinfty$-stable $\Zp$-lattice which is of height $\leqs h$).
\end{defnsub}

\begin{exasub}\label{exa:repfinht}\hfill
\begin{enumerate}
\item 
\label{exa:repfinht:sst}Theorem~\ref{thm:KisinIntPAdicHT}\eqref{thm:KisinIntPAdicHT:st} asserts that the $\GKinfty$-restriction of $V\in\Rep^{[0,h]}_{\st}(\GK)$ is of height $\leqs h$, and Theorem~\ref{thm:main}\eqref{thm:main:lattices} implies that  any $\GKinfty$-stable $\Zp$-lattice in $V\in\Rep^{[0,h]}_{\st}(\GK)$ is of height $\leqs h$. Furthermore, Theorem~\ref{thm:KisinIntPAdicHT}\eqref{thm:KisinIntPAdicHT:BT} asserts that any $p$-adic $\GKinfty$-representation of height $\leqs 1$ extends uniquely up to isormorphism to a crystalline $\GK$-representation with Hodge-Tate weights in $[0,1]$.
\item 
\label{exa:repfinht:unram}A $\Zp$-lattice $\GKinfty$-representation $T$ is unramified if and only if $T$ is of height $\leqs 0$ (i.e. $T$ comes from an ``\'etale $(\vphi,\Sig)$-module''). This can be seen either from Theorem~\ref{thm:KisinIntPAdicHT}\eqref{thm:KisinIntPAdicHT:crys} and the natural isomorphism $\GKinfty/ \IKinfty \riso \GK/\IK$ induced by the natural inclusion, or alternatively by \cite[Prop~5.2.10]{Kim:Thesis}.
\item 
\label{exa:repfinht:cycchar}For an integer $r\in[0,h]$, let $\Sig(r)$ denote the rank-$1$ $\vphi$-module equipped with $\vphi_{\Sig(r)}:\sig^*\e \mapsto \PP(u)^r\e$ where $\e\in\Sig(r)$ is a $\Sig$-basis.  Then the natural $\GKinfty$-action on  $\nf T^{\leqs h}_\Sig(\Sig(r))$ is given by $\cycchar^{h-r}|_{\GKinfty}$ where $\chi_{\cyc}:\GK \ra \Zp\starr$ is the $p$-adic cyclotomic character. (Here the normalization $\PP(0)=p$ is used.) This can be proved by applying to this setup the construction of the arrows in the upper row of the diagram \eqref{eqn:KisinIntPAdicHT}, which is done in \cite[\S1.2]{kisin:fcrys}. The proof is done in \cite[Lemma~2.3.4]{Kisin:2adicBT}. 
\end{enumerate}\end{exasub}

\begin{rmksub}
From Theorems~\ref{thm:main} and \ref{thm:KisinIntPAdicHT}\eqref{thm:KisinIntPAdicHT:st} we obtain the classification of $\GKinfty$-stable $\Zp$-lattices in a semi-stable $\GK$-representation via $\vphi$-modules of finite height. Tong Liu \cite{Liu:SST-Lattice} found an additional structure on $\vphi$-modules of finite height which can be defined if and only if the corresponding $\GKinfty$-stable $\Zp$-lattice is $\GK$-stable. But T.~Liu's theory involves a ring whose structure is not well-understood.
\end{rmksub}

\subsection{Torsion theory}
We say that a $p^\infty$-torsion $\GK$-representation $T$ is \emph{(torsion) crystalline with Hodge-Tate weights in $[0,h]$} if $T\cong T_0/T_1$ for some $\GK$-stable lattices $T_0,T_1$ in a crystalline $p$-adic $\GK$-representation $V$ with Hodge-Tate weights in $[0,h]$. We similarly define (torsion) semi-stable $\GK$-representations.

Except for some limited cases\footnote{i.e., when we can use finite flat group schemes or Fontaine-Laffaille theory} it seems to be very difficult to study torsion crystalline or semi-stable $\GK$-representations, but Kisin's theory outlined in \S\ref{subsec:KisinIntPAdicHT} provides a linear-algebraic tool to study their $\GKinfty$-restrictions.

\begin{defnsub}\label{def:TorKisMod}
For a non-negative integer $h$, a  \emph{torsion $(\vphi,\Sig)$-module of $\PP$-height $\leqs h$} (or simply, a  \emph{torsion $\vphi$-module of height $\leqs h$} if $\Sig$ and $\PP(u)$ are understood) is a finitely generated $\Sig$-module $\gM$ equipped with an $\Sig$-linear morphism $\vphi_\gM:\fstr \gM$ with the following properties:
\begin{enumerate}
\item
Some power of $p$ annihilates $\gM$.
\item\label{def:TorKisMod:ProjDim}
There is no non-zero $u$-torsion in $\gM$; or equivalently by a theorem of Auslander-Buchsbaum, $\gM$ is of projective dimension $\leqs 1$ as a $\Sig$-module.
\item
The cokernel of $\vphi_\gM$ is killed by $\PP(u)^h$. 
\end{enumerate}

We let $(\Mod/\Sig)^{\leqs h}$ denote the category of $\vphi$-modules of height $\leqs h$ with $\vphi$-compatible $\Sig$-linear morphisms, and let $\ModFIh{}$ denote the full subcategory of objects $\gM\in(\Mod/\Sig)^{\leqs h}$ which is isomorphic to $\bigoplus_i (\Sig/p^i)^{n_i}$ as a $\Sig$-module.
\end{defnsub}

There exist natural notions of  subquotient, direct sum, and $\otimes$-product for $\vphi$-modules of height $\leqs h$. Since the scalar extension $\gM\otimes_\Sig \fo_\Eps$ (with $\vphi_\gM\otimes\id_{\fo_\Eps}$) is a $p^\infty$-torsion \'etale $\vphi$-module\footnote{Since $\gM$ has no non-zero $u$-torsion, the natural map $\gM \ra\gM\otimes_\Sig\fo_\Eps$ is injective. By diagram chasing, it follows that $\vphi_\gM$ is \emph{injective}.}, we define a functor $\nf T^{\leqs h}_\Sig:\phimodh\Sig \ra \prep(\GKinfty)$ in a similar manner to \eqref{eqn:TSigLattice}, as follows:
\begin{equation}\label{eqn:TSigTors}
\nf T^{\leqs h}_\Sig(\gM):=\nf T_\Eps(\gM\otimes_\Sig \fo_\Eps)(h), \textrm{ for any }\gM\in(\Mod/\Sig)^{\leqs h}.
\end{equation}
Note that this functor $\nf T^{\leqs h}_\Sig$ is \emph{not} in general fully faithful.\footnote{The functor $\nf T^{\leqs h}_\Sig$ is fully faithful only when $eh<p-1$. See \cite[Prop 9.3.3]{Kim:Thesis} for the proof.}

The following lemma, which is proved in \cite[Lemma 2.3.4]{kisin:fcrys}, motivates Definition \ref{def:TorKisMod}, especially the condition \eqref{def:TorKisMod:ProjDim}.
\begin{lemsub}\label{lem:TorKisMod}
A $\Sig$-module $\gM$ equipped with a $\Sig$-linear map $\vphi_\gM:\sig^*\gM\ra\gM$ is a torsion $\vphi$-mdoule of height $\leqs h$ if and only if $\gM\cong \coker[f:\gM_1 \ra \gM_0]$ for some $\gM_0,\gM_1\in\phimodh\Sig$ and a $\vphi$-compatible map $f$ such that $f\otimes\Qp$ is an isomorphism. Furthermore, we have a natural $\GKinfty$-isomorphism $\nf T^{\leqs h}_\Sig(\gM)\cong \coker \nf T^{\leqs h}_\Sig(f)$.  
\end{lemsub}

\begin{defnsub}\label{def:hthTors}
A finite torsion $\Zp$-module $T$ with continuous $\GKinfty$-action (which will be called a ``torsion $\GKinfty$-representation'' from now on) is  \emph{of height $\leqs h$} if the following equivalent conditions satisfy:
\begin{enumerate}
\item \label{def:hthTors:model}
For some $\gM\in(\Mod/\Sig)^{\leqs h}$, there is a $\GKinfty$-isomorphism $T\cong \nf T^{\leqs h}_\Sig(\gM)$. (We say that such $\gM$ is a \emph{$\Sig$-module model of height $\leqs h$} for $T$.)
\item  \label{def:hthTors:Siglattice} 
For some $\gM\in(\Mod/\Sig)^{\leqs h}$, we have an isomorphism $\gM\otimes_\Sig\fo_\Eps \cong \nf D^{\leqs h}_\Eps(T(-h))$ as \'etale $\vphi$-modules. (We say that such $\gM$ is a \emph{$\Sig$-submodule of height $\leqs h$} in the \'etale $\vphi$-module $\nf D^{\leqs h}_\Eps(T(-h))$.)
\item
There is a $\GKinfty$-isomorphism $T\cong T_0/T_1$ where $T_0$ and $T_1$ are $\Zp$-lattices $\GKinfty$-representation of height $\leqs h$. 
\end{enumerate}
\end{defnsub}
The claimed equivalence easily follows Lemma~\ref{lem:TorKisMod} and Theorem~\ref{thm:FontaineEtPhiMod}. 
By Example~\ref{exa:repfinht}\eqref{exa:repfinht:sst}, the $\GKinfty$-restriction of a torsion semi-stable $\GK$-representation with Hodge-Tate weights in $[0,h]$ is of height $\leqs h$. Note that $\gM$ as  in \eqref{def:hthTors:model} and \eqref{def:hthTors:Siglattice}, is not in general unique\footnote{If $eh<p-1$ where $e$ is the absolute ramification index of $K$, then such $\gM$ is unique up to isomorphism.}, nor does  there exists a natural choice. 

The following theorem,which we call \emph{the Breuil-Kisin classification of finite flat group schemes}, will not be used later in this paper, but it provides a motivation to study $\Modh$ (at least when $h=1$). 
\begin{thmsub}[Classification of finite flat group schemes]\label{thm:classifgpsch}
If $p>2$ then there exists an exact equivalence of categories $\nf G$ from $(\Mod/\Sig)^{\leqs1}$ to the category of finite flat group schemes of $p$-power order over $\fo_K$ such that $\nf G(\gM)(\bar K)|_{\GKinfty}\cong \nf T^{\leqs 1}_\Sig(\gM)$ for any $\gM\in(\Mod/\Sig)^{\leqs1}$. 
\end{thmsub}
See \cite[\S2.3]{kisin:fcrys} for the proof. When $p=2$ one has a similar classification for connected finite flat group schemes; see \cite[\S1]{Kisin:2adicBT} for the statement and the proof.

\subsection{$\Sig$-submodules of height $\leqs h$}
As remarked previously, there is no \emph{a priori} canonical choice of $\Sig$-module models for a $p^\infty$-torsion $\GKinfty$-representation. This becomes a problem when we try to use torsion $\Sig$-modules to study  deformations of $\GKinfty$-representation. 
The results in this section is needed to handle difficulties arising from this ``non-canonicity''.

\parag\label{par:ModFI}
Let $A$ be a $p$-adically separated and complete topological ring\footnote{For us topological rings are always linearly topologized. Later we need to consider coefficient rings that are not finite $\Zp$-algebras such as $A=\F[t]$, especially for analyzing the connected components of the generic fiber of a deformation ring.}, (for example, finite $\Zp$-algebras or any ring $A$ with $p^N\cdot A = 0$ for some $N$). Set $\Sig_A:=\Sig\wh\otimes_{\Zp}A:=\varprojlim_\alpha \Sig\wh\otimes_{\Zp}A/I_\alpha$ where $\set{I_\alpha}$ is a basis of open ideals in $A$. We define a ring endomorphism $\sig:\self{\Sig_A}$ (and call it \emph{Frobenius endomorphism})  by $A$-linearly extending the Frobenius endomorphism $\sig_\Sig$. We also  put $\fo_{\Eps,A}:=\fo_\Eps\wh\otimes_{\Zp}A :=\varprojlim_\alpha\fo_\Eps\wh\otimes_{\Zp}A/I_\alpha$ and similarly define $\sig:\self{\fo_{\Eps,A}}$.

Let $\ModFIh A$ be the category of finite free $\Sig_A$-modules $\gM_A$ equipped with a $\Sig_A$-linear map $\vphi_{\gM_A}:\sig^*(\gM_A)\ra\gM_A$ such that $\PP(u)^h$ annihilates $\coker(\vphi_{\gM_A})$. If $A$ is finite artinean $\Zp$-algebra, then $\gM_A\in\ModFIh A$ is precisely a torsion $(\vphi,\Sig)$-module of height $\leqs h$ equipped with a $\vphi$-compatible $A$-action such that $\gM_A$ is finite free over $\Sig_A$. 

Similarly, let $\ModFIet A$ be the category of finite free $\fo_{\Eps,A}$-modules $M_A$ equipped with a $\fo_{\Eps,A}$-linear isomorphism $\vphi_{M_A}:\sig^*(M_A)\riso M_A$. If $A$ is finite artinean $\Zp$-algebra, then one can check that $\nf T_\Eps$ and $\nf D_\Eps$ as in Theorem~\ref{thm:FontaineEtPhiMod} induce quasi-inverse equivalences of categories between $\ModFIet A$ and the category of $\GKinfty$-representations over $A$.\footnote{The relevant freeness follows from length consideration.}

\begin{lemsub}\label{lem:FlatCokerVphi}
Let $A$ be as in \S\ref {par:ModFI} and let $\gM_A \in \ModFIh A$.
\begin{enumerate}
\item 
\label{lem:FlatCokerVphi:Inj}
The map $\vphi_{\gM_A}$ is injective.
\item 
\label{lem:FlatCokerVphi:Flat}
The $\Sig_A/((\PP(u)^h)$-module $\coker(\vphi_{\gM_A})$, viewed as an $A$-module via scalar restriction, is flat over $A$.
 \end{enumerate}
\end{lemsub}
\begin{proof}
One can easily reduce to the case when $A$ is discrete. 
The injectivity of  $\vphi_{\gM_A}$ follows from a simple diagram chasing using the fact that the $\fo_{\Eps,A}$-linear extension of $\vphi_{\gM_A}$ is an isomorphism on $\gM_A\otimes_{\Sig_A}\fo_{\Eps,A}$, and that the natural map $\gM_A \ra \gM_A\otimes_{\Sig_A}\fo_{\Eps,A}$ is injective.

To show \eqref{lem:FlatCokerVphi:Flat}, observe that the exact sequence
\begin{displaymath}
 0\ra \sig^*\gM_A \xra{\vphi_{\gM_A}} \gM_A \ra \coker(\vphi_{\gM_A}) \ra 0
\end{displaymath}
stays short exact after applying $A/I\otimes_A (\cdot)$ for any ideal $I\subset A$. Hence, \eqref{lem:FlatCokerVphi:Flat} follows from standard facts about flatness (e.g. by \cite[Ch.I \S2.5 Prop 4]{Bourbaki:CommAlg}, or by an argument using $\Tor_1^A$.) 
\end{proof}

\parag\label{par:CarDual}
Let us  define duality on $\Modh$ which is an analogue of Cartier duality. For $\gM\in\Modh$ we define the \emph{dual of height $h$} to be $\gM^\vee:=\Hom_\Sig(\gM, \Sig[\ivtd p]/\Sig)$ equipped with a $\Sig$-linear map $\vphi_{\gM^\vee}:\sig^*\gM^\vee \ra \gM^\vee$ defined as follows:
\begin{equation}
\big(\vphi_{\gM^\vee}(l)\big)(m)= l\big(\vphi\iv(\PP(u)^hm)\big),
\end{equation} 
where  $m\in\gM$ and $l\in\sig^*(\gM^\vee) \cong \Hom_\Sig(\sig^*\gM, \Sig[\ivtd p]/\Sig)$. This makes $\gM^\vee$ as an object in  $\Modh$.

We similarly define $\gM_A^\vee$ for $\gM_A\in\ModFIh A$ using $\Sig_A$-linear duality instead of Pontryagin duality. When $A$ is finite, these two definitions of $\gM_A^\vee$ are compatible.

\begin{propsub}\label{prop:RamakrishnaCrit}
Any subquotients and direct sums of torsion $\GKinfty$-representations of height $\leqs h$ is of height $\leqs h$.
\end{propsub}
\begin{proof}
The assertion about direct sums is obvious. Now consider a short exact sequence $0\ra M' \ra M \ra M'' \ra 0$ of $p^\infty$-torsion \'etale $\vphi$-modules and  assume that  there is a $\vphi$-sbable $\Sig$-submodule $\gM\in\Modh$ in $M$ such that $\gM\otimes_\Sig\fo_\Eps = M$. Let $\gM''$ be the image of $\gM$ by  $M \thra M''$ and $\gM'$ the kernel of the natural map $\gM \ra \gM''$. One can check that $\gM'$ and $ \gM''$ are objects  in $\Modh$ such that $\gM'\otimes_\Sig\fo_\Eps = M'$ and  $\gM''\otimes_\Sig\fo_\Eps = M''$. Now the proposition follows from the exactness of $\nf D_\Eps$ and $\nf T_\Eps$ (Theorem~\ref{thm:FontaineEtPhiMod}). 
\end{proof}

\parag
Let $M$ be a $p^\infty$-torsion \'etale $\vphi$-module, and $\gM,\gM'\subset M$ be $\Sig$-submodules of height $\leqs h$ (as in Definition~\ref{def:hthTors}\eqref{def:hthTors:Siglattice}). One can check that the $(\vphi,\Sig)$-submodules $\gM+\gM'$ and $\gM\cap\gM'$ of $M$ are also objects in $\Modh$, hence are $\Sig$-submodules of height $\leqs h$ in $M$. Therefore, inclusion defines a partial ordering for the set of $\Sig$-submodules of height $\leqs h$ in $M$. The following lemma further shows that this set is bounded.

\begin{lemsub}\label{lem:MaxMinProlongation}
Let $M$ be a $p^\infty$-torsion \'etale $\vphi$-module which admits a $\Sig$-submodule of height $\leqs h$ (defined in Definition~\ref{def:hthTors}(\ref{def:hthTors:Siglattice})). Then there exist the maximal and the minimal $\Sig$-submodules of height $\leqs h$ (denoted by $\gM^+$ and $\gM^-$, respectively) with respect to the partial ordering given by inclusion. In particular, there are only finitely many $\Sig$-submodules of height $\leqs h$.
\end{lemsub}
\begin{proof}
Let $M$ be as in the statement. In order to prove the lemma, it suffices to show the existence of the maximal $\Sig$-submodule  $\gM^+\subset M$ of height $\leqs h$; the existence of the minimal $\Sig$-submodule $\gM^-\subset M$ of height $\leqs h$ follows via the duality defined in \S\ref{par:CarDual}, and the finite assertion follows since $\gM^+/\gM^-$ is an artinean module (being killed by some powers of $p$ and $u$).  

In order to show the existence of $\gM^+\subset M$, it is enough to handle the case when $p\tim M=0$  by a  d\'evissage argument using Proposition~\ref{prop:RamakrishnaCrit}. Now, assume that there exists $\gM\in\Modh$ with $\gM[\ivtd u]=M$ and $p\cdot\gM=0$. 
Consider the following algebras
\begin{equation}\label{eqn:FaltingStrMod1}
A_M:=\frac{\Sym_{\fo_\Eps/(p)} (M)}{\langle m^p-\vphi(\sig^*m):m\in M\rangle},\textrm{ and}\quad \A_\gM:= \frac{\Sym_{\Sig/(p)}(\gM) }{ \langle m^p-\vphi(\sig^*m):m\in\gM\rangle}.
\end{equation}
Clearly, $\A_\gM$ is finite flat over $\Sig/(p)$ with $\A_\gM[\ivtd u] \cong A_M$ since $\gM$ is finite flat over $\Sig/(p)$ with $\gM[\ivtd u]=M$. Note also that $A_M$ is an \'etale algebra over $\fo_\Eps/(p)$ since $M$ is an \'etale $\vphi$-module. If $\gM'\subset M$ is another $\Sig$-submodule of height $\leqs h$ containing $\gM$, then $\A_{\gM'}$ is finite over $\A_\gM$ and we have $\A_{\gM'}[\ivtd u]= A_M$. But the integral closure of $\A_\gM$ in $A_M$ is finite over $\A_\gM$ since $A_M$ is \'etale\footnote{Since $A_M$ is \'etale, the ``generic trace pairing'' $A_M\otimes_{\fo_\Eps/(p)} A_M \ra \fo_\Eps/(p)$ is perfect. The integral closure of $\A_\gM$ is therefore contained in the $\fo_K$-linear dual of $\A_\gM$ embedded in $A_M$ via the ``generic trace pairing'', and this is a finitely generated $\A_\gM$-module.}, so the set of $\Sig$-submodules of height $\leqs h$ in $M$ is bounded above. \end{proof}

Note that any $\vphi$-compatible map $f:M\ra M'$ of $p^\infty$-torsion \'etale $\vphi$-modules restricts to a map $\gM^+\ra{\gM'}^+$ between the maximal $\Sig$-submodules of height $\leqs h$. Using this, we obtain the following useful corollary:

\begin{corsub}\label{cor:RaynaudScheme}
Let $\F$ be a finite extension of $\Fp$, and $\bar\rho$ a $\GKinfty$-representation over $\F$ which is of height $\leqs h$ as a torsion $\GKinfty$-representation (in the sense of Definition~\ref{def:hthTors}). Then there exists $\gM_\F \in\ModFIh\F$ such that $\Th_\Sig(\gM_\F)\cong\bar\rho$.
\end{corsub}
\begin{proof}
Put $M:=\nf D^{\leqs h}_\Eps(\bar\rho(-h))$ and let $\gM_\F:=\gM^+\subset M$ be the maximal $\Sig$-submodule of height $\leqs h$. Then the $\vphi$-compatible $\F$-action on $M$ (induced by the scalar multiplication on $\bar\rho$) induces a $\vphi$-compatible $\F$-action on $\gM_\F$, which makes $\gM_\F$ a projective $\Sig_\F$-module. (Note that $\Sig_\F$ is a product of copies of a discrete valuation ring.) To show $\gM_\F\in\ModFIh\F$ it is left to show that $\gM_\F$ is free over $\Sig_\F$, but this follows because the endomorphism $\sig:\self{\Sig_\F}$ transitively permutes the orthogonal idempotents of $\Sig_\F$.
\end{proof}

\begin{rmksub}\label{rmk:FailureIsom4T}
Corollary~\ref{cor:RaynaudScheme} does not fully generalize to a $\GKinfty$-representation $\rho_A$ of height $\leqs h$ over a finite artinean $\Zp$-algebra $A$. 
Assume that $he\geqs p$ where $e$ is the absolute ramification index of $K$, and consider $\F[\ep]$ where $\ep^2=0$. Let $M$ be a rank-$1$ free $\fo_{\Eps,\F[\ep]}$-module equipped with $\vphi_M(\sig^*\e)=(\PP(u)^h+\frac{1}{u}\ep)\e$ for an $\fo_{\Eps,\F[\ep]}$-basis $\e\in M$. Let $\gM$ be a $\Sig_\F$-span of $\{\e,\ivtd u\ep\e\}$ in $M$. Then $\gM\subset M$ is a $\Sig$-submodule of height $\leqs h$ (using that $he\geqs p$), but one can check that there cannot exist a $\Sig$-submodule of height $\leqs h$ which is rank-$1$ free over $\Sig_{\F[\ep]}$.\footnote{One way to see this is by directly computing the ``$\vphi$-matrix'' for any $\fo_{\Eps,\F[\ep]}$-basis $\e'\in M$, and show that it cannot divide $\PP(u)^h$.} Note also that $\gM$ above is the maximal $\Sig$-submodule of height $\leqs h$ and has a $\vphi$-compatible $\F[\ep]$-action induced from $M$, but $\gM$ is not projective over $\Sig_{\F[\ep]}$. This is where the proof of Corollary~\ref{cor:RaynaudScheme} fails.
\end{rmksub}

\begin{propsub}\label{prop:LimitThm}
A $\Zp$-lattice $\GKinfty$-representation $T$ is of height $\leqs h$ in the sense of Definition~\ref{def:hthLattice} if and only if $T/p^nT$ is of height $\leqs h$ for each $n$ in the sense of Definition \ref{def:hthTors}.
\end{propsub}
\begin{proof}[Idea of Proof]
This proposition is proved in \cite[Prop~9.2.6]{Kim:Thesis} (\emph{cf.} \cite[Thm~2.4.1, \S4.4]{Liu:Fontaine}) and here we only indicate the main idea. The proposition follows from the same argument as the proof of Raynaud's ``limit theorem'' for Barsotti-Tate representations \cite[Prop~2.3.1]{raynaud:GpSch}, but working with torsion $\Sig$-modules of height $\leqs h$ instead of finite flat group schemes. This is possible, provided that we have suitable analogues of generic fiber, prolongations (with a notion of partial ordering), schematic closure and Cartier duality. The analogue of Cartier dual $\gM^\vee$ for $\gM\in(\Mod/\Sig)^{\leqs h}$ is introduced in \S\ref{par:CarDual}. 
The analogue of the generic fiber for torsion $\vphi$-modules of height $\leqs h$ is the $p^\infty$-torsion \'etale $\vphi$-module obtained by extending scalars to $\fo_\Eps$, and the analogue of prolongation is a $\Sig$-submodule of height $\leqs h$ in a torsion \'etale $\vphi$-module in the sense of Definition~\ref{def:hthTors}\eqref{def:hthTors:Siglattice}  with partial ordering given by inclusion, which is bounded by Lemma ~\ref{lem:MaxMinProlongation}. 
For $\gM\in(\Mod/\Sig)^{\leqs h}$ and a surjective map $f: \gM\otimes_\Sig\fo_\Eps \thra M'$ of \'etale $\vphi$-modules, $f(\gM)$ serves as the analogue of ``schematic closure of $M'$ in $\gM$''. Note that $f(\gM)$ is a torsion $\vphi$-module of height $\leqs h$.   
\end{proof}


We end this section by stating Kisin's theorem on the existence of (potentially) crystalline and semistable deformation rings. Let $\F$ be a finite field of characteristic $p$, $\fo$ a $p$-adic discrete valuation ring with residue field $\F$, and $\bar\rho:\GK\ra\GL_d(\F)$ a  representation. Let $R^\Box$ be a universal framed deformation ring of $\rho$ over $\fo$, which exists by Mazur's theorem \cite{Mazur:DeformingGalReps}. Let $\rho^\Box$ be the universal framed deformation of $\bar\rho$. 
\begin{thmsub}[Kisin]\label{thm:CrysStDeforRing}
There exists unique $\fo$-flat quotients $R_{\cris}^{\Box,[0,h]}$ and  $R_{\st}^{\Box,[0,h]}$ of $R^\Box$ such that for any finite $\Qp$-algebra $A$ a morphism $\xi:R^\Box \ra A$ factors through $R^{\Box,[0,h]}_{\cris}$ (respectively, $R^{\Box,[0,h]}_{\st}$) if and only if $\rho^\Box\otimes_{R,\xi}A$ is crystalline (respectively, semi-stable) with Hodge-Tate weights in $[0,h]$ as a $\GK$-representation over $\Qp$ in the sense of Fontaine \cite{fontaine:Asterisque223ExpIII}. A similar statement holds for (unframed) universal deformation rings if they exist.
\end{thmsub}
Note that the mapping properties in the corollary above only characterize the generic fibers\footnote{Note that $R^\Box[\ivtd p]$ is a Jacobson ring so the properties uniquely determines $R^{\Box,[0,h]}_{\cris}[\ivtd p]$ and $R^{\Box,[0,h]}_{\st}[\ivtd p]$ if they exist.} $R^{\Box,[0,h]}_{\cris}[\ivtd p]$ and $R^{\Box,[0,h]}_{\st}[\ivtd p]$ as quotients of $R^\Box[\ivtd p]$, so the uniqueness assertion follows from requiring that the quotients are $\fo$-flat. These quotients  $R_{\cris}^{\Box,[0,h]}$ and  $R_{\st}^{\Box,[0,h]}$  were first constructed by Kisin \cite[\S2]{kisin:pstDefor}. Later, T.~Liu  \cite{Liu:Fontaine} gave more ``conceptual'' construction of these quotients by proving the ``limit theorem'' (i.e., the natural analogue of Proposition~\ref{prop:LimitThm}) for crystalline and semi-stable $\GK$-representations with \emph{bounded}  Hodge-Tate weights. 

Using \cite[Corollary~2.2.6]{kisin:fcrys} and Raynaud's ``limit theorem'' \cite[Proposition~2.3.1]{raynaud:GpSch}, it follows that $R^{\Box,[0,1]}_{\cris}\cong R^\Box_{\fl}/(p^\infty\textrm{-torsion})$ where $R^\Box_{\fl}$ is a framed flat deformation ring. We do not use this result, as we solely work with $R^{\Box,[0,1]}_{\cris}$.

\section{$\GKinfty$-deformation rings of height $\leqs h$}\label{sec:Representability}
\subsection{}\label{subsec:RepbilitySetup}
Let $\F$ be a finite field of characteristic $p$, and $\bar\rho_\infty:\GKinfty\ra\GL_d(\F)$ a  representation. Let $\fo$ be a $p$-adic discrete valuation ring with residue field $\F$. Let $\art{\fo}$ be the category of artin local $\fo$-algebras $A$ whose residue field is $\F$, and let $\arhat{\fo}$ be the category of complete local noetherian $\fo$-algebras with residue field  $\F$.

Let $D_\infty,D_\infty^\Box:\arhat\fo \ra \Sets$ be the deformation functor and framed deformation functor for $\bar\rho_\infty$. For the definition, see the standard references such as \cite{Mazur:Deformation-FLT, Mazur:DeformingGalReps, Gouvea:GalDefor}. Contrary to local and global deformation functors we usually consider, these functors \emph{cannot} be represented by complete local noetherian rings since the tangent spaces $D_\infty (\F[\ep])$ and $D_\infty^\Box(\F[\ep])$ are infinite dimensional $\F$-vector spaces. See \S\ref{par:MazurRamakrishna} for more details.

We say that a deformation $\rho_{\infty,A}$ over $A\in\art\fo$ is \emph{of height $\leqs h$} if it is a torsion $\GKinfty$-representation of height $\leqs h$ as a torsion $\Zp[\GKinfty]$-module; or equivalently, if there exists $\gM\in\Modh$ and an isomorphism $\Th_\Sig(\gM)\cong\rho_{\infty,A}$ as $\Zp[\GKinfty]$-modules. For $A\in\arhat\fo$, we say that $\rho_{\infty,A}$ is \emph{of height $\leqs h$} if $\rho_{\infty,A}\otimes A/\m_A^n$ is a deformation of height $\leqs h$ for each $n$. When $A\in\art\fo$, both definitions are compatible by Proposition~\ref{prop:RamakrishnaCrit}. When $A$ is finite flat over $\Zp$, a deformation $\rho_{\infty,A}$ over $A$ is of height $\leqs h$ if and only if $\rho_{\infty,A}$ is of height $\leqs h$ as a $\Zp$-lattice $\GKinfty$-representation (in the sense of Definition~\ref{def:hthLattice}), by Proposition~\ref{prop:LimitThm}.

Let $D_\infty^{\leqs h}\subset D_\infty$ and $D_\infty^{\Box,\leqs h}\subset D_\infty^\Box$ respectively denote subfunctors of deformations and framed deformations of height $\leqs h$. The main goal of this section is to prove the following theorem:
\begin{thm}\label{thm:Representability}
The functor $D^{\leqs h}_\infty$ always has a hull. If $\End_{\GKinfty}(\bar\rho_\infty) \cong \F$ then $D^{\leqs h}_\infty$ is representable (by $R_\infty^{\leqs h}\in\arhat\fo$). The functor $D^{\Box,\leqs h}_\infty$ is representable (by $R_\infty^{\Box,\leqs h}\in\arhat\fo$) with no assumption on $\bar\rho_\infty$.
Furthermore, the natural inclusions $D^{\leqs h}_\infty \hra D_\infty$ and $D^{\Box,\leqs h}_\infty\hra D^\Box_\infty$ of functors are relatively representable by surjective maps in $\arhat\fo$.
\end{thm}
We call $R_\infty^{\Box,\leqs h}$ the \emph{universal framed deformation ring of height $\leqs h$} and $R_\infty^{\leqs h}$ the  \emph{universal deformation ring of height $\leqs h$} if it exists. 
We give a proof later in this section (\S~\ref{par:MazurRamakrishna} and onward), and simply remark that the hard part of the proof is to show that the tangent space of $D^{\leqs h}_\infty$ is a finite-dimensional $\F$-vector space. 

We record a few consequences of Theorem~\ref{thm:Representability}. Firstly, we obtain the following result which was proved by Kisin \cite[Prop~1.6.4(2)]{kisin:pstDefor} via different method:
\begin{cor}\label{cor:hthQt}
Let $\rho_{\infty,R}$ be a $\GKinfty$-representation over $R\in\arhat\fo$. Then there exists a quotient $R^{\leqs h}$ of $R$ such that for any finite $W(k)[\ivtd p]$-algebra $A$, a map $\xi:R\ra A$ factors through $R^{\leqs h}$ if and only if  $\rho_{\infty,R}\otimes_{R,\xi}A$ is of height $\leqs h$ as $\GKinfty$-representation over $\Qp$ (in the sense of Definition ~\ref{def:hthLattice}). 
\end{cor}
Note that the above mapping property uniquely determines $R^{\leqs h}[\ivtd p]$ because $R[\ivtd p]$ is a Jacobson ring. For the proof of the corollary, we set $R^{\leqs h}$ to be the universal quotient of $R$ of height $\leqs h$ (as in Theorem~\ref{thm:Representability}), and this satisfies the desired mapping property by Proposition~\ref{prop:LimitThm}.

Let $\bar\rho:\GK\ra\GL_d(\F)$ be a representation and set $\bar\rho_\infty:=\bar\rho|_{\GKinfty}$. Let $R_{\st}^{\Box,[0,h]}$ and $R_{\cris}^{\Box,[0,h]}$ respectively denote semi-stable and crystalline framed deformation ring of $\bar\rho$, and $R_\infty^{\Box,\leqs h}$ is the universal framed deformation ring  of $\bar\rho_\infty$ with height $\leqs h$. 
\begin{cor}\label{cor:Restr}
Restricting to $\GKinfty$ induces the following morphisms:
\begin{equation*}
\res_{\st}^{\leqs h}:R_\infty^{\Box,\leqs h} \ra R_{\st}^{\Box,[0,h]}\textrm{, and }\res_{\cris}^{\leqs h}:R_\infty^{\Box,\leqs h} \ra R_{\cris}^{\Box,[0,h]}.
\end{equation*}
The same holds for deformation rings (without framing) if $\End_{\GKinfty}(\bar\rho_\infty)\cong\F$. 
\end{cor}
\begin{proof}
Let $R^\Box$ be the universal framed $\GK$-deformation ring of $\bar\rho$, and $R^{\Box,\leqs h}$  the universal quotient of $R^\Box$ with height $\leqs h$  (obtained by applying Corollary~\ref{cor:hthQt} to $R=R^\Box$). Kisin \cite[\S2]{kisin:pstDefor} constructed $R_{\st}^{\Box,[0,h]}$ and $R_{\cris}^{\Box,[0,h]}$ as quotients of $R^{\Box,\leqs h}$, so the corollary follows.

Alternatively, if one uses T.~Liu's construction of $R_{\st}^{\Box,[0,h]}$ and $R_{\cris}^{\Box,[0,h]}$ as in \cite[Theorem~1.0.2]{Liu:Fontaine}, then the corollary follows from the observation that  the $\GKinfty$-restriction of torsion semi-stable $\GK$-representation with Hodge-Tate weights in $[0,h]$ is of height $\leqs h$ (by Example~\ref{exa:repfinht}\eqref{exa:repfinht:sst} and Definition~\ref{def:hthTors}).
\end{proof}
\begin{rmk}\label{rmk:ResNonIsom}
Kisin's full faithfulness result (Theorem~\ref{thm:main}\eqref{thm:main:fullfth}) implies that $\res_{\cris}^{\leqs h}\otimes_{\Zp}\Qp$ induces a surjection on completions at any maximal ideal of $R_\infty^{\Box,\leqs h}\otimes_{\Zp}\Qp$. We will show that $\res_{\cris}^{\leqs 1}\otimes_{\Zp}\Qp$ is an isomorphism when $p>2$ (and a weaker statement when $p=2$). On the other hand, $\res_{\cris}^{\leqs h}$ is \emph{not} in general an isomorphism; the dimensions of the source and the target are not same at a maximal ideal of $R_{\cris}^{\Box,[0,h]}[\ivtd p]$ which corresponds to a lift which has two Hodge-Tate weights whose difference is at least $2$. See \cite[Theorem~3.3.8]{kisin:pstDefor} and \cite[Corollary~11.3.11]{Kim:Thesis} for the dimension formulas.
\end{rmk}
%

\subsection{Resum\'e of Mazur's and Ramakrishna's theory}\label{par:MazurRamakrishna}
Given a functor $D:\art\fo \ra \Sets$, Schlessinger found three conditions (H1)-(H3) which are equivalent for $D$ to have a hull. He also showed that $D$ is pro-representable if and only if $D$ satisfies an additional condition (H4). For the statement and a proof, see \cite[Thm 2.11]{Schlessinger:FunctArtRing}.

Mazur \cite[\S1.2]{Mazur:DeformingGalReps} showed that for a profinite group $\Gamma$ and a continuous $\F$-linear $\Gamma$-representation $\bar\rho$, the deformation functor for $\bar\rho$ 
always satisfies (H1)-(H2), and satisfies (H4) if $\End_{\Gamma}(\bar\rho) \cong \F$ (e.g. if $\bar\rho$ is absolutely irreducible). Furthermore, Mazur showed that the \emph{framed} deformation functor for $\bar\rho$ 
always satisfies (H1), (H2) and (H4) with no assumption on $\bar\rho$.

On the other hand,  framed or unframed deformation functors for arbitrary $\Gamma$ do not automatically satisfy (H3) (i.e., the tangent space is a finite-dimensional $\F$-vector space). The condition (H3) is satisfied when $\Gamma$ is $p$-finite in the sense of \cite[\S1.1]{Mazur:DeformingGalReps}), and this is the case when $\Gamma$ is either an absolute Galois group for a finite extension of $\Qp$, or a certain quotient of the absolute Galois group of any finite extension of $\Q$.
Unfortunately, $\GKinfty$ does not satisfy the $p$-finiteness, and in fact (H3) fails even when $\bar\rho_\infty$ is $1$-dimensional.
To see this, consider the  cohomological interpretation of  the tangent space; i.e., $D_{\infty}(\F[\ep])\cong\coh 1 (K_\infty,\Ad(\bar\rho_\infty))$, where $\Ad(\bar\rho_\infty):=(\bar\rho_\infty)^*\otimes_\F\bar\rho_\infty$ and $D_\infty$ is the deformation functor for $\bar\rho_\infty$. If $\bar\rho_\infty$ is $1$-dimensional, then  $\Ad(\bar\rho_\infty)$ is the trivial $1$-dimensional $\GKinfty$-representation, so $\coh 1 (K_\infty,\Ad(\bar\rho_\infty))\cong \Hom_{\cont}(\GKinfty,\F)$. We can see this is infinite from the norm field isomorphism $\GKinfty \cong \Gal(k\llpar u \rrpar^{\sep}/k\llpar u \rrpar)$ and the existence of infinitely many Artin-Schreier cyclic $p$-extensions of $k\llpar u \rrpar$.
For a finite dimensional $\bar\rho_\infty$,  one sees that the deformation and framed deformation functors $D_{\infty}$ and  $D^\Box_{\infty}$ never satisfies (H3) from deforming the determinant
\footnote{For any $\F[\ep]$-deformation $\det(\bar\rho_\infty)+\ep\tim c$ of $\det(\bar\rho_\infty)$ (where $c:\gal_K\ra F$ is a cocycle), the deformation $\bar\rho_\infty + \ep\tim\tilde c$ with $\tilde c := \mthree{c}{0}{\cdots}{0}{0}{}{\vdots}{}{\ddots}$ has determinant $\det(\bar\rho_\infty)+\ep\tim c $.}, and in particular these `unrestricted' deformation functors are never represented by a complete local noetherian ring.

Now, let us look at the subfunctors $D^{\leqs h}_{\infty}\subset D_{\infty} $ and $D^{\Box,\leqs h}_{\infty}\subset D^\Box_{\infty}$ which consist of deformations of height $\leqs h$ (as defined in \S\ref{subsec:RepbilitySetup}).
We have seen, in Proposition \ref{prop:RamakrishnaCrit}, that these subfunctors are closed under subobjects, quotients, and direct sums. Under this setup, Ramakrishna proved that if the ambient functor satisfies (H$i$) for some $i=$1, 2, 3 or 4, then so does the subfunctor. (For more details, see the proof of \cite[Theorem~1.1]{ramakrishna:FlatDeforRing}, or  \S25 and \S23 of \cite{Mazur:Deformation-FLT}.)

Applying this to our setup, we obtain the following results.
\begin{enumerate}
\item The functor $D^{\leqs h}_{\infty}$ always satisfies (H1)-(H2), and satisfies (H4) if we have $\End_{\GKinfty}(\bar\rho_\infty) \cong \F$.
\item The functor $D^{\Box,\leqs h}_{\infty}$ always satisfies (H1), (H2), and (H4) with no assumptions on $\bar\rho_\infty$.
\end{enumerate}
Therefore, for the representability assertion of  Theorem \ref{thm:Representability} it remains to check (H3)\footnote{Even though (H3) is false for $D_{\infty}$ and $D^{\Box}_{\infty}$, one can hope that suitable subfunctors of them satisfy (H3).} for $D^{\leqs h}_{\infty}$ and $D^{\Box,\leqs h}_{\infty}$. Before doing this, let us digress to show the relative representability of the subfunctor  $D^{\leqs h}_{\infty}\subset D_{\infty} $, which ``essentially'' follows from the discussion above on Ramakrishna's theory.

\begin{prop}\label{prop:RelRepbility}
The subfunctor $D^{\leqs h}_\infty \subset D_\infty$ 
is relatively representable by surjective maps in $\arhat\fo$.
In other words, for any given deformation 
$\rho_A$ over $A\in\arhat\fo$, there exists a universal quotient $A^{\leqs h}$ of $A$ over which the deformation 
is of height $\leqs h$.
\end{prop}
\begin{proof}
Consider a functor $\mathfrak{h}_A:\art\fo \ra\Sets$ defined by $\mathfrak{h}_A(B):=\Hom_\fo(B,A)$ for $B\in\art\fo$, and a subfunctor $\mathfrak{h}_A^{\leqs h}\subset \mathfrak{h}_A$ defined as below:
\begin{equation*}
\mathfrak{h}_A^{\leqs h}(B):=\left\{f:B\ra A\textrm{ such that } \rho_A\otimes_{A,f}B\textrm{ is of height }\leqs h\right\},
\end{equation*}
where $B\in\art\fo$. Since $\mathfrak{h}_A$ is prorepresentable and the subfunctor $\mathfrak{h}_A^{\leqs h}$ is closed under subquotients and direct sums, it follows that $\mathfrak{h}_A^{\leqs h}$ is prorepresentable, say by a quotient $A^{\leqs h}$ of $A$. It is clear that $A^{\leqs h}$ satisfies the desired properties. (\emph{cf.} the proof of \cite[Theorem 1.1]{ramakrishna:FlatDeforRing}.)
\end{proof}

Now let us verify (H3) for $D^{\leqs h}_{\infty}$ and $D^{\Box,\leqs h}_{\infty}$, thus prove the  representability assertion of  Theorem \ref{thm:Representability}. 
\begin{prop}\label{prop:RepresentabilityH3}
The tangent spaces $D^{\leqs h}_{\infty}(\F[\ep])$ and $D^{\Box,\leqs h}_{\infty}(\F[\ep])$ are finite-dimensional $\F$-vector spaces.
\end{prop}
\begin{proof}
Since $D^{\Box,\leqs h}_{\infty}(\F[\ep])$ is a torsor of $\wh{\nf\GL}(n)(\F[\ep])/(1+\ep\Ad(\bar\rho_\infty)^{\GKinfty})$ over $D^{\leqs h}_{\infty}(\F[\ep])$, it is enough to show that the set  $D^{\leqs h}_{\infty}(\F[\ep])$ is finite. 
We proceed in the following steps.

\parag[Setup]
\label{Repbility1}
We use the notation introduced in \S\ref{par:ModFI}.
Let $\ol M:=\nf D_\Eps (\bar\rho_\infty(-h))$ and consider $(M,\iota)$, where $M\in\ModFIet{\F[\ep]}$ and  $\iota:\ol M \riso M \otimes_{\F[\ep]}\F$ is a $\vphi$-compatible $\fo_{\Eps,\F}$-linear isomorphism. Two such lifts $(M,\iota)$ and $(M',\iota')$ are \emph{equivalent} if there exists an isomorphism $f:M\riso M'$ such that $(f\bmod{\ep})\circ\iota=\iota'$. Theorem~\ref{thm:FontaineEtPhiMod} implies that $\nf T_\Eps(\cdot)(h)$ and $\nf D_\Eps(\cdot(-h))$ induce inverse bijections between $D_\infty(\F[\ep])$ and the set of equivalent classes of $(M,\iota)$.

Now assume that there is a $\vphi$-stable $\Sig_{\F}$-lattice $\gM\subset M$ of height $\leqs h$, viewing $M$ as a $\fo_{\Eps,\F}$-module. By Corollary~\ref{cor:RaynaudScheme}, the set of equivalence classes of such $(M,\iota)$ exactly corresponds to $D^{\leqs h}_\infty(\F[\ep])$ via the bijections in the previous paragraph. So Proposition~\ref{prop:RepresentabilityH3} is equivalent to the following claim:
\begin{claimsub}\label{clm:DeforEtPhiMod}
There exists only finitely many equivalence classes of $(M,\iota)$ where $M$ admits a $\vphi$-stable $\Sig_\F$-lattice which is of height $\leqs h$.
\end{claimsub}
\parag[Strategy and Outline]
\label{Repbility2}
One possible approach to prove Claim~\ref{clm:DeforEtPhiMod} is to fix a $\fo_{\Eps,\F}$-basis for $\ol M$ and a lift to an $\fo_{\Eps,\F[\ep]}$-basis for each deformation $M$ once and for all, and identify $M$ with the ``$\vphi$-matrix'' with respect to the fixed basis and interpret the equivalence relations in terms of the ``$\vphi$-matrix.'' Then the problem turns into showing the finiteness of equivalence classes of matrices with some constraints -- namely, having some ``integral structure''; more precisely, having a $\vphi$-stable  $\Sig_\F$-lattice with height $\leqs h$ (but not necessarily a $\Sig_{\F[\ep]}$-lattice). So the fixed basis has to ``reflect'' the integral structure.

This approach faces the following obstacles. First,  the deformations $M$ we consider do not necessarily allow any $\vphi$-stable $\Sig_{\F[\ep]}$-lattice with height $\leqs h$ as we have seen at Remark \ref{rmk:FailureIsom4T}. In other words, we cannot expect, in general, to find a $\fo_{\Eps,\F[\ep]}$-basis $\{e_i\}$ for $M$ in such a way that $\{e_i,\ep e_i\}$ generates a $\Sig_{\F}$-lattice of height $\leqs h$. In \S\ref{Repbility3}--\S\ref{Repbility5} we show that a slightly weaker statement is true. Roughly speaking, we show that there is an $\fo_{\Eps,\F[\ep]}$-basis $\set{\e_i}$ for $M$ so that there exists a $\vphi$-stable $\Sig_{\F}$-lattice with height $\leqs h$ with a $\Sig_\F$-basis only involving ``uniformly'' $u$-adically bounded denominators as coefficients relative to the $\fo_{\Eps,\F}$-basis $\set{\e_i,\ep\tim\e_i}$ of $M$.

Second, we may have more than one $\vphi$-stable $\Sig_\F$-lattice with height $\leqs h$ for $\ol M$ or for $M$, especially when  $he$ is large. In particular, a fixed $\Sig_\F$-lattice for $\ol M$ may not be nicely related to any $\vphi$-stable $\Sig_\F$-lattice with height $\leqs h$ for some lift $M\in\ModFIet{\F[\ep]}$.
We get around this issue by varying the basis for $\ol M$ among finitely many choices. This step is carried out in \S\ref{Repbility6}. In fact, we only need finitely many choices of bases because there are only finitely many $\Sig_\F$-lattices of height $\leqs h$ for a fixed $\ol M$, thanks to Lemma \ref{lem:MaxMinProlongation}.

Once we get around these technical problems, we show the finiteness by a $\sig$-conjugacy computation of matrices. This is the key technical step and crucially uses the assumption that the $\F[\ep]$-deformations we consider (or rather, the corresponding \'etale $\vphi$-module $M$) admits a $\vphi$-stable $\Sig_\F$-lattice with height $\leqs h$ (in $M$). See Claim \ref{Repbility7} for more details.

\parag
 \label{Repbility3}
Let $M$ correspond to some $\F[\ep]$-deformation of height $\leqs h$.
Even though there may not exist any $\vphi$-stable $\Sig_{\F[\ep]}$-lattice with height $\leqs h$ for $M$, we can find a $\vphi$-stable $\Sig_\F$-lattice $\gM$ with height $\leqs h$ such that $\gM$ is stable under multiplication by $\ep$.\footnote{This means that $\gM$ is a $\vphi$-module over $\Sig_{\F[\ep]}$ and is projective over $\Sig_\F$, but $\gM$ does not have to be a projective $\Sig_{\F[\ep]}$-module. Hence, such $\gM$ may not be an object in  $\ModFIh{\F[\ep]}$. This actually occurs: $\gM\cong \Sig_\F\tim\e\oplus\Sig_\F\tim(\frac{1}{u}\ep\e)$ discussed in Remark \ref{rmk:FailureIsom4T} is such an example.}
In fact, the maximal $\Sig$-submodule $\gM^+\subset M$ among the ones with height $\leqs h$ does the job, as remarked above Corollary~\ref{cor:RaynaudScheme}. 
\parag
 \label{Repbility4}
For a $\Sig_{\F}$-lattice $\gM\subset M$ of height $\leqs h$ which is stable under the $\ep$-multiplication, we can find a $\Sig_\F$-basis which can be ``nicely'' written in terms of some $\fo_{\Eps,\F[\ep]}$-basis of $M$, as follows. Let $\ol \gM$ be the image of $\gM\ra \ol M$ induced by the natural projection $M\ra \ol M$, which  is a $\vphi$-stable $\Sig_\F$-lattice with height $\leqs h$ in $\ol M$. Now, consider the following diagram:
\begin{displaymath}
\xymatrix{
0 \ar[r] &
\gN \ar[r] \ar@{^{(}->}[d] &
\gM \ar[r] \ar@{^{(}->}[d]&
\ol \gM \ar[r] \ar@{^{(}->}[d]&
0\\
0 \ar[r] &
\ep\tim M \ar[r] &
M \ar[r] &
M_{\F} \ar[r] &
0,
}\end{displaymath}
where $\gN:=\Ker[\gM\thra \ol \gM]$ is a $\vphi$-stable $\Sig_\F$-lattice with height $\leqs h$ in $M$. We choose a $\Sig_\F$-basis $\{e_1,\cdots,e_n\}$ of $\ol \gM$. Viewing them as a $\fo_{\Eps,\F}$-basis of $\ol M$, we lift $\{e_i\}$ to an $\fo_{\Eps,\F[\ep]}$-basis of $M$ (again denoted $\{e_i\}$). By assumption from the previous step, we have $\bigoplus_{i=1}^n\Sig_\F\tim (\ep e_i) \subset \gN$, where both are $\Sig_\F$-lattices of height $\leqs h$ for $\ep\tim M$. It follows that $(\ivtd{u^{r_i}}\ep) e_i$ form a $\Sig_\F$-basis of $\gN$ for some non-negative integers $r_i$. Therefore, $\{e_i,(\ivtd{u^{r_i}}\ep)e_i\}$ is a $\Sig_\F$-basis of $\gM$. 
\parag 
\label{Repbility5}
In this step, we find an upper bound for the non-negative integers $r_i$ only depending on $\ol \gM$ and the choice of $\Sig_\F$-basis of $\ol \gM$.
Since $\gN$ is a $\vphi$-stable submodule, it contains
\begin{equation}\label{eqn:Repbility:BddDenom}
\vphi_{M}\left(\sig^*\left(\frac{1}{u^{r_i}}\ep e_i\right)\right)=\left(\frac{1}{u^{pr_i}}\ep\right)\tim \vphi_{\ol \gM}(\sig^*e_i) = \frac{1}{u^{pr_i}}\ep \tim \sum_{j=1}^n \alpha_{ij}e_j,
\end{equation}
where $\alpha_{ij}\in\Sig_\F$ satisfy $\vphi_{\ol \gM}(\sig^*e_i) = \sum_{j=1}^n \alpha_{ij}e_j$. Note that we obtain the first identity because $\vphi_{M}(\sig^*e_i)$ lifts $\vphi_{\ol \gM}(\sig^*e_i)$ and the $\ep$-multiple ambiguity in the lift disappears when we multiply against $\ep$. Since any element of $\gN$ is a $\Sig_\F$-linear combination of $(\ivtd{u^{r_i}}\ep) e_i$, we obtain inequalities $\ord_u(\alpha_{ij}) - pr_i \geq -r_j$ for all $i,j$ from the above equation \eqref{eqn:Repbility:BddDenom}. Let $r:=\max_j\set{r_j}$ and we obtain $pr_i \leq r + \min_j\set{\ord_u(\alpha_{ij})}$ for all $i$. (Note that the right side of the inequality is always finite.) Now, by taking the maximum among all $i$, we obtain 
\begin{displaymath}
r \leq \frac{1}{p-1} \max_i\big\{\min_j\set{\ord_u(\alpha_{ij})} \big\}<\infty
\end{displaymath}
This shows that the non-negative integers $r_i$ has an upper bound which only depends on the matrices entries for $\vphi_{\ol \gM}$ with respect to the $\Sig_\F$ basis of $\ol \gM$.
\parag[Recapitulation]
\label{Repbility6}
Let $\set{\ol \gM^{(a)}}$ denote the set of all the $\Sig_\F$-lattices  of height $\leqs h$ in $\ol M$. This is a finite set by Lemma~\ref{lem:MaxMinProlongation}. For each $\ol\gM^{(a)}$, we fix a $\Sig_\F$-basis $\{\e_i^{(a)}\}$ and let $\alpha^{(a)}=(\alpha_{ij}^{(a)}) \in \Mat_n(\Sig_\F)$ be the ``$\vphi$-matrix'' with respect to $\{\e_i^{(a)}\}$; i.e.,
$\vphi_{\ol \gM^{(a)}}(\sig^*\e_i^{(a)}) = \sum_{i=1}^n \alpha_{ij}^{(a)}\e_j^{(a)}$. We also view $\{\e_i^{(a)}\}$ as a $\fo_{\Eps,\F}$-basis for $\ol M$ and $(\alpha_{ij}^{(a)})$ is the matrix for $\vphi_{\ol M}$ with respect to $\{\e_i^{(a)}\}$. Note that $(\alpha_{ij}^{(a)})$ is invertible over $\fo_{\Eps,\F}$ since $\ol M=\ol \gM^{(a)}[\ivtd u]$ is an \'etale $\vphi$-module. We pick an integer $r^{(a)} \geq  \frac{1}{p-1} \max_i\big\{\min_j\set{\ord_u(\alpha_{ij})} \big\}$, for each index $a$.

For any $M$ which corresponds to a deformation of height $\leqs h$, we may find a $\Sig_\F$-lattice $\gM\subset M$ of height $\leqs h$  which is stable under $\ep$-multiplication. (See \S\ref{Repbility3}.) The image of $\gM$ inside $\ol M$ is equal to some $\ol \gM^{(a)}$.  Lift the chosen basis $\{\e_i^{(a)}\}$ to an $\fo_{\Eps,\F[\ep]}$-basis for $M$. Then $\gM$ admits a $\Sig_\F$-basis of form $\{\e_i^{(a)},(\ivtd{u^{r_i}}\ep) \e_i^{(a)}\}$ for some integers $r_i \leq r^{(a)}$ (\S\ref{Repbility4}--\S\ref{Repbility5}).

Let us consider the matrix representation of $\vphi_{M}$ with respect to the basis $\{\e_i^{(a)}\}$.
We have $\vphi_{M}(\e_i^{(a)}) = \sum_{i}(\alpha_{ij}^{(a)}+\ep\beta^{(a)}_{ij})\e_j^{(a)}$ for some $\beta^{(a)}=(\beta^{(a)}_{ij})\in \Mat_n(\fo_{\Eps,\F})$ because $\vphi_{M}$ lifts $\vphi_{\ol M}$. Furthermore we have that $\beta\in \ivtd{u^{r^{(a)}}}\tim\Mat_n(\Sig_\F)$  since $\gM\subset M$ is $\vphi$-stable.
We say two such matrices $\beta$ and $\beta'$ are \emph{equivalent} if there exists a matrix $X\in\Mat_n(\fo_{\Eps,\F})$ such that $\beta' = \beta + (\alpha^{(a)}\tim\sig(X)-X\tim\alpha^{(a)})$. This equation is obtained from the following:
\begin{displaymath}
(\alpha^{(a)}+\ep\beta') = (\Id_n+\ep X)\iv \tim(\alpha^{(a)}+\ep\beta)\tim\sig(\Id_n+\ep X),
\end{displaymath}
which defines the equivalence of two \'etale $\vphi$-modules whose $\vphi$-structures are given by $(\alpha^{(a)}+\ep\beta)$ and $(\alpha^{(a)}+\ep\beta')$, respectively.

Now, the theorem is reduced to the verification of the following claim: \emph{for each $a$, there exist only finitely many equivalence classes of matrices $\beta\in \ivtd{u^{r^{(a)}}}\tim\Mat_n(\Sig_\F)$.}
Indeed, by varying both $a$ and the equivalence classes of $\beta$, we cover all the possible lifts $M$ of ``height $\leqs h$'' up to equivalence, hence the theorem is proved.

From now on, we fix $a$ and suppress the superscript $(\cdot)^{(a)}$ everywhere. For example, $\ol\gM:=\ol\gM^{(a)}$, $r:=r^{(a)}$, and $\alpha:=\alpha^{(a)}$. Proving the following claim is the last step of the proof.
\begin{claimsub}\label{Repbility7}
For any $X\in u^{c}\Mat_n(\Sig_\F)$ with $c>2he$, the matrices $\beta$ and $\beta + X$ are equivalent.\footnote{The inequality $c>2he$ is used to ensure $p(c-he) > c$. Therefore, if $p\ne 2$ then $c=2he$ also works.}
\end{claimsub}
This claim provides a surjective map from $\left(\ivtd{u^{r}}\tim\Mat_n(\Sig_\F)\right) / \left(u^{c}\tim\Mat_n(\Sig_\F)\right)$ onto the set of equivalence classes of $\beta$'s, and the former is a finite set\footnote{We crucially used the fact that we can bound the denominator.}, thus we conclude the proof of Proposition~\ref{prop:RepresentabilityH3}.

We prove the claim by ``successive approximation.'' Let $\gamma = u^{he}\tim\alpha\iv$. Since $\ol \gM$ is of height $\leqs h$ and $\PP(u)$ has image in $\Sig_\F\cong (k\otimes_{\Fp}\F)[[u]]$ with $u$-order $e$, we know that $\gamma \in \Mat_n(\Sig_\F)$. We set $Y^{(1)} := \ivtd{u^{he}}\tim (X\gamma)$, which is in $u^{c-he}\Mat_n(\Sig_F)$ by the assumption on $X$. Then $\beta+X$ is equivalent to
\begin{equation*}
(\beta+X) + (\alpha\tim\sig(Y^{(1)}) - Y^{(1)}\alpha) = \beta + \alpha\tim\sig(Y^{(1)}) =: \beta + X^{(1)}
\end{equation*}
with $X^{(1)} \in u^{c^{(1)}}\tim\Mat_n(\Sig_\F)$, where $c^{(1)}:= p(c-he) > c$.
Now for any positive integer $i$, we recursively define the following
\begin{displaymath}
Y^{(i)} := \ivtd{u^{he}}\tim (X^{(i-1)}\gamma),\qquad X^{(i)} := \alpha\tim\sig(Y^{(i)}), \qquad c^{(i)}:= p(c^{(i-1)}-he).
\end{displaymath}
One can check that $c^{(i)}>c{(i-1)}(>2he)$, $X^{(i)} \in u^{c^{(i)}}\tim\Mat_n(\Sig_\F)$, and $Y^{(i)} \in u^{c^{(i-1)}-he}\Mat_n(\Sig_\F)$.
Also, $\beta+X$  is equivalent to
\begin{equation*}
(\beta+X) + \left(\alpha\tim\sig(Y^{(1)}+ \cdots +Y^{(i)}) - (Y^{(1)}+\cdots+Y^{(i)})\alpha\right) = \beta + X^{(i)}.
\end{equation*}

Since $c^{(i)}\to\infty$ as $i\to\infty$, it follows that  the infinite sum $Y:=\sum_{i=1}^{\infty}Y^{(i)}$ converges and $X^{(i)} \to 0$ as $i\to\infty$. Therefore we see that $\beta+X$  is equivalent to
\begin{eqnarray*}
(\beta+X) + \left(\alpha\tim \sig(Y) -y\tim\alpha\right) &=& (\beta+X) + \left(\alpha\tim\sig\big(\sum_{i=1}^{\infty}Y^{(i)}\big) - \big(\sum_{i=1}^{\infty}Y^{(i)}\big)\tim\alpha\right) \\
&=& \lim_{i\to\infty}(\beta + X^{(i)}) = \beta,
\end{eqnarray*}
so we are done.
\end{proof}

\section{Generic fibers of $\GKinfty$-deformation rings of height $\leqs 1$}\label{sec:GenFibers}
Let $\bar\rho_\infty$ be a $\GKinfty$-representation over $\F$ and let $R^{\Box,\leqs 1}_\infty$ denote the framed deformation ring of height $\leqs 1$. In this section we study the generic fiber $R^{\Box,\leqs 1}_\infty[\ivtd p]$, such as formal smoothness, dimension, and connected components. The main idea is to ``resolve'' $R^{\Box,\leqs 1}_\infty$ via some suitable closed subscheme of affine grassmannian which parametrize certain ``nice'' $\Sig$-module models of height $\leqs 1$ for framed deformations of $\bar\rho_\infty$. This technique is inspired by Kisin's construction of moduli of finite flat group schemes \cite[\S2]{Kisin:ModuliGpSch} (and also \cite[\S1]{kisin:pstDefor}).  In fact, most of linear algebra steps in aforementioned Kisin's works carry over word by word to our situation due to the similarities of linear algebraic structures involved. Note, however, that we ``resolve'' a framed $\GKinfty$-deformation ring $R^{\Box,\leqs 1}_\infty$ instead of a framed flat deformation ring.

With some extra work, many results in this section generalize to (framed) $\GKinfty$-deformation rings of height $\leqs h$ for any $h$, possibly except the result on non-ordinary components. See \cite[\S11]{Kim:Thesis} for the statements and proofs. All the results apply to (unframed) $\GKinfty$-deformation rings if they exist.

\subsection{Definition: moduli of $\Sig$-modules of height $\leqs h$}\label{subsec:DefnOfResol}
Let $h$ be a positive integer, and we will later specialize to the case when $h=1$. Consider a deformation $\rho_R
$ of $\bar\rho_\infty$ over $R\in\arhat\fo$ which is of height $\leqs h$ (i.e. $\rho_R\otimes_R R/\m_R^n$ is of height $\leqs h$ for each $n$). The main examples to keep in mind are universal framed or unframed deformation of height $\leqs h$. 

We use the notations from \S\ref{par:ModFI}.
Put $M_R:=\varprojlim M_n$ where $M_n\in\ModFIet{R/\m_R^n}$ is such that $\nf T_\Eps(M_n)(h)\cong \rho_R\otimes_R R/\m_R^n$ for each $n$. For any $R$-algebra $A$, we view $M_R\otimes_R A$ as an \'etale $\vphi$-module by $A$-linearly extending $\vphi_{M_R}$. 

For a complete local noetherian ring $R$, let $\aug R$ be the category of pairs $(A,I)$ where $A$ is an $R$-algebra and $I\subset A$  is an ideal with $I^N=0$ for some $N$ which contains $\m_R\tim A$. Note that an artin local $R$-algebra $A$ can be viewed as an element in $\aug\fo$ by setting $I:=\m_A$. A morphism $(A,I)\ra(B,J)$ in $\aug R$ is an $R$-morphism $A\ra B$ which takes $I$ into $J$. 
We define a functor $D^{\leqs h}_{\Sig,\rho_R}:\aug R \ra \Sets$ by putting $D^{\leqs h}_{\Sig,\rho_R}(A,I)$ the set of $\vphi$-stable $\Sig_A$-lattices in $M_R\otimes_RA$ which are of height $\leqs h$. In \cite{kisin:pstDefor} this functor is denoted by $L^{\leqs h}_{\rho_R}$. 

\begin{prop}\label{prop:ResolDefor}
There exists a projective $R$-scheme $\GRh_{\rho_R}$ and a $\Sig\otimes_{\Zp}\OO_{\GRh_{\rho_R}}$-lattice $\nf\gM^{\leqs h}_{\rho_R} \subset M_R \otimes_{R} \OO_{\GRh_{\rho_R}}$ of height $\leqs h$ which represents $D^{\leqs h}_{\Sig,\rho_R}$ in the following sense: there exists a natural isomorphism $\GRh_{\rho_R} \riso D^{\leqs h}_{\Sig,\rho_R}$ such that for any $(A,I)\in\aug R$ it sends an $A$-point $\eta\in\GRh_{\rho_R}(A)$ to $\eta^*(\nf\gM^{\leqs h}_{\rho_R}) \in D^{\leqs h}_{\Sig,\rho_R}(A,I)$. (We call  $\nf\gM^{\leqs h}_{\rho_R}$ a \emph{universal $\Sig$-lattice of height $\leqs h$} for $\rho_R$.)

Moreover, for any map $R\ra R'$ of complete local noetherian rings, there exists a unique isomorphism $\GRh_{\rho_R} \otimes_R R' \riso \GRh_{\rho_R\otimes_RR'}$, which pulls back $\nf \gM^{\leqs h}_{\rho_R\otimes_RR'}$ to $\nf\gM^{\leqs h}_{\rho_R}\otimes_R R'$ inside of $(M_R \otimes_R \OO_{\GRh_{\rho_R}})\otimes_RR'$.
\end{prop}
This proposition is proved in \cite[Proposition~1.3; Corollary~1.5.1; Corollary~1.7]{kisin:pstDefor} possibly except the base change assertion which is straightforward, so we omit the proof.

\subsection{Generic fibers and characteristic $0$ deformations}
We now give a ``moduli interpretation'' of the generic fiber\footnote{or rather, its completion at a closed point} of a $\GKinfty$-deformation ring of height $\leqs h$. 
We begin with the following lemma:
\begin{lemsub}\label{lem:GalLattice}
Let $A$ be a finite $\Qp$-algebra, and $V_A$ a finite free $A$-module with continuous $\GKinfty$-action. Then, for some (big enough) finite $\Zp$-subalgebra $A^\circ \subset A$ with $A^\circ[\ivtd p]=A$, there exists a $\GKinfty$-stable $A^\circ$-lattice $T_{A^\circ}\subset V_A$.  Furthermore, for any choice of $\GKinfty$-stable $\fo_E$-lattice $T_{\fo_E}$ in $V_A\otimes_AE$, one can arrange $T_{A^\circ}$ so that its image under the natural projection $V_A \thra V_A\otimes_AE$ is $T_{\fo_E}$.
\end{lemsub}
\begin{proof}
We may assume $A$ is a finite local $\Qp$-algebra with residue field $E$. Let $A^+$ denote the preimage of $\fo_E$ under the natural projection $A\thra E$. Choose a $\GKinfty$-stable $\fo_E$-lattice $T_{\fo_E}$ in $V_A\otimes_AE$, and let $V_{A^+}$ denote the preimage of $T_{\fo_E}$ under the natural projection $V_A\thra V_A\otimes_AE$. Now, fix an $A^+$-basis $\set{\e_i}$ of $V_{A^+}$. Since $\GKinfty$ is compact and $A^+$ is a rising union of finite $\Zp$-subalgebras of $A$, one can find such an $A^\circ \subset A^+$ such that all the matrix entries for $\rho_A(g)$ for $g\in\GKinfty$ are valued $A^\circ$. Let  $T_{A^\circ}\subset V_{A^+}$ be the $A^\circ$-span of   $\set{\e_i}$ .
\end{proof}
\parag
For a finite $\Qp$-algebra $A$ we say an $A$-representation $\rho_A$ of $\GKinfty$ is \emph{of height $\leqs h$} if it is so as a $\Qp$-representation of $\GKinfty$ (in the sense of Definition~\ref{def:hthLattice}). For a finite extension $E$ of $\Qp$, we study the following deformation problem of a $E$-representation $\rho_E$ of $\GKinfty$: let $\art E$ be the category of artin local $E$-algebra with residue field $E$, and let $D^{\leqs h}_{\rho_E},D^{\Box,\leqs h}_{\rho_E}:\art E \ra \Sets$ respectively denote the subfunctors of the deformations and framed deformations of $\rho_E$ with height $\leqs h$.  

Choose a $\GKinfty$-stable $\fo_E$-lattice in $\rho_E$, and view $\rho_E$ as a ``deformation'' of the mod $p$ representation $\bar\rho_\infty$ of $\GKinfty$ obtained by reducing the $\fo_E$-lattice modulo $\m_E$.
By the ``limit theorem'' (Proposition~\ref{prop:LimitThm}) and Lemma~\ref{lem:GalLattice}, we obtain the following corollary:

\begin{corsub}\label{cor:GenFiberDefR}
Consider an $\fo$-map $\eta: R^{\Box,\leqs h}_\infty\ra E$ and let $\rho_\eta:= \rho^{\Box,\leqs h}\otimes_{R^{\Box,\leqs h}_\infty,\eta}E$ be the corresponding $\GKinfty$-representation over $E$.
Let $(R^{\Box,\leqs h}_\infty)\comp\eta$ be the completion of $R^{\Box,\leqs h}_\infty\otimes_\fo E$ with respect to the maximal ideal generated by $\ker(\eta)$, and $\wh\rho^{\Box,\leqs h}_\eta$ the  pullback to $(R^{\Box,\leqs h}_\infty)\comp\eta$ of the universal framed deformation of height $\leqs h$. Then $(R^{\Box,\leqs h}_\infty)\comp\eta$ together with  $\wh\rho^{\Box,\leqs h}_\eta$ represents the functor $D^{\Box,\leqs h}_{\rho_\eta}$. A similar statement holds for unframed deformation rings if they exist.
\end{corsub}
\begin{proof}
For a complete local noetherian ring $R$ with finite residue field, any ring morphism $\xi:R\ra A$ into a finite $\Qp$-algebra $A$ should factor through a finite $\Zp$-subalgebra of $A$. 
This, combined with Proposition~\ref{prop:LimitThm}, shows that the pullback $\rho_\xi:=\rho^{\Box,\leqs h}\otimes_{R^{\Box,\leqs h}_\infty,\xi}A$ under any map $\xi: R^{\Box,\leqs h}_\infty\ra A$ is of height $\leqs h$. If $\xi$ lifts $\eta$, then $\rho_\xi$ is a lift of $\rho_\eta$.

For any framed $A$-deformation $\rho_A$ of $\rho_\eta$, Lemma~\ref{lem:GalLattice} gives us a map $\xi^\circ:R^{\Box,\leqs h}_\infty\ra A^\circ$ for some finite $\fo$-subalgebra $A^\circ\subset A$. By composing it with the natural inclusion $A^\circ \hra A$, we obtain a map $\xi:R^{\Box,\leqs h}_\infty\ra A$ such that $\rho_A\cong \rho^{\Box,\leqs h}\otimes_{R^{\Box,\leqs h}_\infty,\xi}A$. This map $\xi$ lifts $\eta$ since $\rho_A$ lifts $\rho_\eta$, so $\xi$ factors through  $(R^{\Box,\leqs h}_\infty)\comp\eta$. Note that the map $\xi$ is independent of the choice of $A^\circ$. 
\end{proof}

\subsection{Generic Fiber of $\GRh_{\rho_R}$}\hfill
%
\parag
\label{par:GenFiberGRh}
For a finite $\Qp$-algebra $A$, we define $\ModFIh A$ to be the category of $\vphi$-modules $\gM_A$ such that for some finite $\Zp$-subalgebra $A^\circ\subset A$ there exists $\gM_{A^\circ}\in\ModFIh{A^\circ}$ with $\gM_{A^\circ}[\ivtd p] \cong \gM_A$. For such $\gM_A$, we define $\nf V^{\leqs h}_\Sig(\gM_A):=\nf T^{\leqs h}_\Sig(\gM_{A^\circ})\otimes_{A^\circ}A$, which is naturally an $A$-representation of height $\leqs h$. Note that $\nf V^{\leqs h}_\Sig(\gM_A)$ is independent of the choice of $\gM_{A^\circ}$.

Let $A$ be a local $R$-algebra with residue field $E$ which is finite over $\Qp$. (This makes $E$ an $R$-algebra.) Now, let us interpret $A$-points of $\GRh_{\rho_R}$. Let $A^+$ denote the preimage of $\fo_E$ under the natural projection $A\thra E$. By the valuative criterion 
any $R$-map $\Spec A \ra \GRh_{\rho_R}$ factors through $\Spec A^+$, so in turn, it factors through $\Spec A^\circ$ where $A^\circ$ is an $R$-subalgebra  in $A^+$ which is finite over $\fo$ and such that $A^\circ[\ivtd p]= A$. This  implies that $\GRh_{\rho_R}(A)$ is naturally isomorphic to the set of $\vphi$-stable $\Sig_A$-lattices $\gM_A$ of $M_R\otimes_R A$ such that $\gM_A\in \ModFIh A$.
%



\begin{propsub}\label{prop:GenIsom}
The structure morphism $\GRh_{\rho_R}\otimes_{\Zp}\Qp \ra \Spec(R\otimes_{\Zp}\Qp)$ is an isomorphism.
\end{propsub}
The  proposition is a direct consequence of  \cite[Proposition~1.6.4]{kisin:pstDefor}; note that both the source and the target of the map are Jacobson.
\begin{propsub}\label{prop:FormalSm}
Assume that the map $\Spf R \ra D^{\leqs h}_{\infty}$ of functors on $\art\fo$ defined by a deformation $\rho_R$ is formally smooth. Then $R[\ivtd p]$ is formally smooth over $\Frac(\fo)$.
\end{propsub}
Note that the assumption on $R$ is satisfied by universal framed or unframed $\GKinfty$-deformation rings of height $\leqs h$. 
\begin{proof}
Note that a noetherian Jacobson $\Frac(\fo)$-scheme $X$ (e.g., $X=\Spec R[\ivtd p]$ for some complete local noetherian $\fo$-algebra $R$) is formally smooth over $\Frac(\fo)$ if and only if its completion at each maximal ideal is geometrically regular, by \cite[0$_{\textrm{IV}}$, Th\'eor\`eme (20.5.8), Corollaires (22.6.5), (22.6.6)]{EGA}. 

Using Proposition~\ref{prop:GenIsom} it is enough to show the following claim: for any finite local $\Frac(\fo)$-algebra $A$ and $\ol A:= A/I$ where $I$ is a square-zero ideal, any $\Frac(\fo)$-map $\bar\xi:\Spec \ol A \ra \GRh_{\rho_R}$ lifts to a $\Frac(\fo)$-map $\xi:\Spec A\ra\GRh_{\rho_R}$. Note that by post-composing the structure morphism $\GRh_{\rho_R}$ over $R$, the map $\bar\xi$ induces a map $R\ra\ol A$, which we also denote by $\bar\xi$. By the discussion in \S\ref{par:GenFiberGRh}, an equivalent claim is that for any $\gM_{\ol A}\in\ModFIh{\ol A}$ with an isomorphism $\bar\iota:\gM_{\ol A}\otimes_{\Sig_{\ol A}}\Eps_{\ol A} \cong M_R\otimes_{R,\bar\xi} \ol A$, one can find $\gM_A\in\ModFIh A$,  $\xi:R\ra A$, and $\iota:\gM_A\otimes_{\Sig_A}\Eps_A \cong M_R\otimes_{R,\xi} A$, which lift $\gM_{\ol A}$, $\bar\xi$, and $\bar\iota$, respectively. 


By Lemma~\ref{lem:FormalSm} there exists $\gM_A\in\ModFIh A$ such that $\gM_A \otimes_A \ol A\cong \gM_{\ol A}$.  Note that $\nf V^{\leqs h}_\Sig(\gM_A)$, defined in \S\ref{par:GenFiberGRh}, is an $A$-deformation of of height $\leqs h$. The lifts $\xi$ and $\iota$ are obtained by formal smoothness of $\Spf R \ra D^{\leqs h}_{\infty}$. 
\end{proof}
\begin{lemsub}\label{lem:FormalSm}
Let $A\thra\ol A$ be a surjective map of finite $\Qp$-algebras with square-zero kernel $I$. For any $\gM_{\ol A}\in\ModFIh{\bar A}$, there exists a $\gM_A\in\ModFIh A$ such that $\gM_A\otimes_A \ol A \cong \gM_{\ol A}$.
\end{lemsub}
\begin{proof}
Choose $\gM_{\ol{A^\circ}}\in \ModFIh{\ol{A^\circ}}$ with $\gM_{\ol{A^\circ}}\otimes_{\ol{A^\circ}} \ol A\cong \gM_{\ol A}$ for some finite $\Zp$-subalgebra $\ol{A^\circ}\subset \ol A$ (which exists by definition), and a finite $\Zp$-subalgebra $A^\circ \subset A$ so that $A^\circ[\ivtd p]=A$ and $A^\circ$ surjects onto  $\ol{A^\circ}$. To prove the lemma, it is enough to lift $\gM_{\ol{A^\circ}}$ to some  $\gM_{A^\circ}\in \ModFIh{A^\circ}$.

Put $\omega_{\ol{A^o}}:=\coker(\vphi_{\gM_{\ol{A^o}}})$, which  is finite free over $\ol{A^o}$ by Lemma~\ref{lem:FlatCokerVphi}\eqref{lem:FlatCokerVphi:Flat}. Let $\omega_{A^o}$ be a finite free $A^o$-module that lifts $\omega_{\ol{A^o}}$, and $\gM_{A^o}$  a finite free $\Sig_{A^o}$-module that lifts $\gM_{\ol{A^o}}$. We can choose a $\Sig_{A^o}/\PP(u)^h$-linear surjection $\gM_{A^o}/\PP(u)^h\gM_{A^o} \thra \omega_{A^o}$ which lifts the natural projection $\gM_{\ol{A^o}}/\PP(u)^h\gM_{\ol{A^o}} \thra \omega_{\ol{A^o}}$. Since $\vphi_{\gM_{\ol{A^o}}}$ is injective by Lemma~\ref{lem:FlatCokerVphi}\eqref{lem:FlatCokerVphi:Inj}, we obtain the following diagram with exact rows:
\begin{equation}\label{eqn:LiftSigMod}
\xymatrix{
0 \ar[r] &
\gN_{A^o} \ar[r] \ar[d] &
\gM_{A^o} \ar[r] \ar@{->>}[d] &
\omega_{A^o} \ar[r] \ar@{->>}[d] &
0\\
0 \ar[r] &
\sig^*\gM_{\ol{A^o}} \ar[r]_-{\vphi_{\gM_{\ol{A^o}}}} &
\gM_{\ol{A^o}} \ar[r] &
\omega_{\ol{A^o}} \ar[r] &
0,
}\end{equation}
where $\gN_{A^o}$ is the kernel of the natural surjection $\gM_{A^o} \thra \gM_{A^o}/\PP(u)^h\gM_{A^o} \thra \omega_{A^o}$.  Since $\omega_{A^o}$ is flat over $A^o$, the top row stays short exact after applying $(\cdot)\otimes_{A^o} \ol{A^o}$. This shows that $\gN_{A^o}\otimes_{A^o} \ol{A^o} \riso \sig^*\gM_{\ol{A^o}}$ 
so there exists a surjective map $r:\sig^*\gM_{A^o} \thra \gN_{A^o}$ which factors the natural projection $\sig^*\gM_{A^o} \thra \sig^*\gM_{\ol{A^o}}$. Now, put $\vphi_{\gM_{A^o}}:\sig^*\gM_{A^o} \xra r \gN_{A^o} \hra \gM_{A^o}$. Clearly $\vphi_{\gM_{A^o}}$ lifts  $\vphi_{\gM_{\ol{A^o}}}$, and $\coker(\vphi_{\gM_{A^o}})$ is isomorphic to $\omega_{A^o}$ so it is annihilated by $\PP(u)^h$. 
%
\end{proof}

\subsection{Hodge type and local structure}\label{subsec:HodgeType}
From now on, we assume that $h=1$ and $\bar\rho_\infty$ is $2$-dimensional. Consider a lift $\rho_R$ of $\bar\rho_\infty$ over $R\in\arhat\fo$ as a $\GKinfty$-representation. 
We first need the following definition.
\begin{defnsub}\label{def:Lagrangian}
For a $\Zp$-algebra $A$, let $\cD_{A,K}$ be a rank-$2$ free $\fo_K\otimes_{\Zp}A$-module. We say an $\fo_K\otimes_{\Zp}A$-submodule $\cL \subset \cD_{A,K}$ is a \emph{Lagrangian} if $\cL$ is a direct factor as an $A$-submodule and it is its own annihilator with respect to the standard $\fo_K\otimes_{\Zp}A$-linear symplectic pairing $\langle\cdot,\cdot\rangle$ on $\cD_{A,K}$ (which is unique up to $(\fo_K\otimes_{\Zp}A)\starr$-multiple).
\end{defnsub}
Note that this notion does not depend on the choice of $\langle\cdot,\cdot\rangle$. If $A$ is flat over $\Zp$, then a Lagrangian is necessarily $\fo_K\otimes_{\Zp}A$-free of rank $1$, but this is not necessarily the case when $A$ has non-zero $p^\infty$-torsion and $K$ is ramified over $\Qp$. 

Define a subfunctor $D^{\bfv}_{\Sig,\rho_R} \subset D^{\leqs1}_{\Sig,\rho_R}$ as follows: for any $(A,I)\in\aug R$ the subset $D^{\bfv}_{\Sig,\rho_R}(A,I)$ consists of all $\gM_A$'s with the property that $\im(\vphi_{\gM_A})/\PP(u)\gM_A \subset \gM_A/\PP(u)\gM_A$ is a Lagrangian in the sense of Definition~\ref{def:Lagrangian}. Note that $\gM_A/\PP(u)\gM_A$ is free of rank $2$ over $\Sig_A/(\PP(u))\cong \fo_K\otimes_{\Zp}A$.
 
\begin{lemsub}\label{lem:IntHodgeType}
The subfunctor $D^{\bfv}_{\Sig,\rho_R} \subset D^{\leqs1}_{\Sig,\rho_R}$ can be represented by a closed subscheme  $\GR^{\bfv}_{\rho_R} \subset \GR^{\leqs 1}_{\rho_R}$.\end{lemsub}

\begin{proof}
We construct $\GR^{\bfv}_{\rho_R} $ as follows.
Put $\Sig_{\GR^{\leqs 1}_{\rho_R}}:=\Sig\otimes_{\Zp}\OO_{\GR^{\leqs 1}_{\rho_R}}$, and let $\nf\gM^{\leqs 1}_{\rho_R}$ denote the universal $\Sig_{\GR^{\leqs 1}_{\rho_R}}$-lattice as in Proposition~\ref{prop:ResolDefor}. We let $\vphi_{\rho_R}^{\leqs1}$ denote the universal $\vphi$-structure on $\nf\gM^{\leqs 1}_{\rho_R}$. 
Lemma~\ref{lem:FlatCokerVphi}\eqref{lem:FlatCokerVphi:Flat} implies that $\im(\vphi_{\rho_R}^{\leqs1})/\PP(u)\nf\gM^{\leqs1}_{\rho_R}\subset \nf\gM^{\leqs1}_{\rho_R}/\PP(u)\nf\gM^{\leqs1}_{\rho_R}$ is a  sub-vector bundle over $\GR^{\leqs 1}_{\rho_R}$ (i.e. its quotient is again a vector bundle over $\GR^{\leqs 1}_{\rho_R}$).

Now, choose a $\Sig_{\GR^{\leqs 1}_{\rho_R}}/(\PP(u))$-basis for $\nf\gM^{\leqs1}_{\rho_R}/\PP(u)\nf\gM^{\leqs1}_{\rho_R}$ and let $\langle \cdot,\cdot\rangle$ denote the standard symplectic pairing on $\nf\gM^{\leqs1}_{\rho_R}/\PP(u)\nf\gM^{\leqs1}_{\rho_R}$ with respect to the fixed basis. Choose an open (affine) covering $\set{U_\alpha}$ of $\GR^{\leqs 1}_{\rho_R}$  which trivializes $\im(\vphi_{\rho_R}^{\leqs1})/\PP(u)\nf\gM^{\leqs1}_{\rho_R}$, and choose an $\OO_{U_\alpha}$-basis $\set{\e_{1,\alpha},\cdots,\e_{r_\alpha,\alpha}}$ of  $(\im(\vphi_{\rho_R}^{\leqs1})/\PP(u)\nf\gM^{\leqs1}_{\rho_R})|_{U_\alpha}$. Now, let  $\scri$ be a coherent ideal on $\GR^{\leqs 1}_{\rho_R}$, where $\scri|_{U_\alpha}$ is generated by $\set{\langle \e_{i,\alpha},\e_{j,\alpha}\rangle}_{i,j=1,\cdots,r_\alpha}$, viewing $\langle\cdot,\cdot\rangle$ as an $\OO_{\GR^{\leqs 1}_{\rho_R}}$-bilinear pairing. Clearly the closed subscheme $\GR_{\rho_R}^{\bfv}$ cut out by $\scri$ represents the functor $D^{\bfv}_{\Sig,M_\F,\rho_R}$.
\end{proof}
 
 \begin{rmksub}
With a little more work, one can show that $\Spec R^{\bfv}[\ivtd p] \subset \Spec R[\ivtd p]$ is an equi-dimensional union of connected components, and when $R=R^{\Box,\leqs 1}_\infty$, then the dimension of $R^\bfv[\ivtd p] = R^{\Box,\bfv}_\infty[\ivtd p]$ is $4+[K:\Qp]$. We do not use this result later.
\end{rmksub}
 
\begin{lemsub}\label{lem:HodgeTypeQt}
Let $\Spec R^{\bfv}$ denote the scheme-theoretic image of $\GRv_{\rho_R}$ in $\Spec R$. For any finite $\Qp$-algebra $A$, a $\Zp$-morphism $\xi:R\ra A$ factors through $R^{\bfv}$ if and only if $\det(\rho_R\otimes_{R,\xi}A)|_{I_{K_\infty}} = \chi_{\cyc}|_{I_{K_\infty}}$.
\end{lemsub}
\begin{proof}
For any $(A,I)\in\aug R$ it can easily be seen that $\gM_A\in D^{\leqs1}_{\Sig,\rho_R}(A,I)$ is in $D^{\bfv}_{\Sig,\rho_R}  (A,I)$ if and only if the induced $\vphi$ on $\bigwedge^2_{\Sig_A}(\gM_A)$ takes any basis $\e$ to $\alpha\PP(u)\e$ for some $\alpha\in\Sig_A\starr$. Now the lemma is a consequence of Proposition~\ref{prop:GenIsom} and Example~\ref{exa:repfinht}\eqref{exa:repfinht:cycchar}.
 \end{proof}
 
\begin{rmksub}\label{rmk:Rv}
Let $\rho_R$ be the universal framed $\GKinfty$-deformation of height $\leqs1$ over $R:= R^{\Box,\leqs1}_\infty$. Consider the subfunctor $D^{\Box,\bfv}_\infty\subset D^{\Box,\leqs 1}_\infty$ defined by requiring that the inertia action on the determinant is given by $\chi_{\cyc}|_{I_{K_\infty}}$. One can easily see that $D^{\Box,\bfv}_\infty$ is represented by a quotient $\wt R^{\Box,\bfv}_\infty$ of $R^{\Box,\leqs 1}_\infty$. If we let $R^{\Box,\bfv}_\infty$ denote the quotient constructed in Lemma~\ref{lem:IntHodgeType}, then Lemma~\ref{lem:HodgeTypeQt} implies that $\wt R^{\Box,\bfv}[\ivtd p] =  R^{\Box,\bfv}_\infty[\ivtd p]$. (And a similar discussion holds for  $R^{\leqs1}_\infty$ if it exists.)
\end{rmksub}
 
 \parag
 We now show that locally for the  smooth topology $\GR^{\bfv}_{\rho_R}$ looks like the ``Deligne-Pappas model model''  \cite{Deligne-Pappas} if the morphism $\Spf R \ra D^{\leqs1}_\infty$ induced by $\rho_R$ is formally smooth. Since local structure of  ``Deligne-Pappas model model''  is well understood, we thus get local structure of $\GR^{\bfv}_{\rho_R}$. 

We fix $\gM_{\F}\in D^{\bfv}_{\Sig,\rho_R}(\F)$. We define a functor  $D^{\bfv}_{\gM_{\F},\rho_R}$ on $\aug R$ by assigning to $(A,I)\in\aug R$  the set of isomorphism classes of $(\gM_A,\iota_A)$, where $\gM_A\in D^{\bfv}_{\Sig,\rho_R}(A,I)$ and $\iota_A:\gM_A\otimes_A A/I \riso \gM_\F\otimes_\F A/I$ with obvious notion of isomorphisms. Clearly $D^{\bfv}_{\gM_\F,\rho_R}$ can be prorepresented by the completed local ring of $\GRv_{\rho_R}$ at the closed point corresponding to $\gM_\F$. 

We fix an isomorphism $\beta_\F:(\fo_K\otimes_{\Zp}\F)^{\oplus 2} \riso \gM_\F/\PP(u)\gM_\F$ over $\Sig_\F/\PP(u)\cong \fo_K\otimes_{\Zp}\F$. We define a functor $\wt D^{\bfv}_{\gM_\F,\rho_R}$ on $\aug R$ by assigning to $(A,I)\in\aug R$  the set  of isomorphism classes of pairs $(\gM_A,\beta_A)$ where $\gM_A\in D^{\bfv}_{\gM_\F,\rho_R}(A)$ and $\beta_A:(\fo_K\otimes_{\Zp}A)^{\oplus 2} \riso \gM_A/\PP(u)\gM_A$ is an $\fo_K\otimes_{\Zp}A$-linear isomorphism which lifts $\beta_\F$. By forgetting this isomorphism, we obtain a formally smooth morphism of functors $\wt D^{\bfv}_{\gM_\F,\rho_R}\ra D^{\bfv}_{\gM_\F,\rho_R}$. 

Now, we define another functor $\nf M_\bfv$ on $\aug R$ by assigning  to $(A,I)\in\aug R$ the set of  Lagrangians in $(\fo_K\otimes_{\Zp}A)^{\oplus 2}$ in the sense of Lemma~\ref{def:Lagrangian}. This functor can be represented by (the $p$-adic completion of) a closed subscheme of a grassmannian over $R$. We also let $\nf M_\bfv$ denote the representing projective scheme over $\Spec R$.  Clearly $\nf M_\bfv\otimes_{\Zp}\Qp$ is formally smooth over $\Frac(\fo)$. It is shown in \cite[\S4]{Deligne-Pappas} that $\nf M_\bfv$ is $\fo$-flat and relative complete intersection, and that $\nf M_\bfv\otimes_\fo \fo/\m_\fo$ is reduced. (For the $\fo$-flatness, see \cite[IV$_\textrm{2}$, (3.4.6.1)]{EGA}.)

We have a morphism $\wt D^{\bfv}_{\gM_\F,\rho_R}\ra\nf M_\bfv$ by sending $\left(\gM_A,\beta_A\right)\in \wt D^{\bfv}_{\gM_\F,\rho_R}(A,I)$ for some $(A,I)\in\aug R$ to the kernel of $(\Sig_A/\PP(u))^{\oplus 2} \xrightarrow[\beta_A]{\sim}  \gM_A/\PP(u)\gM_A \thra\coker\vphi$. The proof of Lemma~\ref{lem:FormalSm} shows that if the morphism $\Spf R \ra D^{\leqs1}_\infty$ induced by $\rho_R$ is formally smooth then the morphism $\wt D^{\bfv}_{\gM_\F,\rho_R}\ra\nf M_\bfv$ is also formally smooth.

Now, we are ready to prove the following proposition:
\begin{propsub}
Assume that the morphism $\Spf R \ra D^{\leqs1}_\infty$ induced by $\rho_R$ is formally smooth. Then $\GR^{\bfv}_{\rho_R}\otimes_{\Zp}\Qp$ is formally smooth over $\Frac(\fo)$, $\GR^{\bfv}_{\rho_R}$ is $\fo$-flat and relative complete intersection, and $\GR^{\bfv}_{\rho_R} \otimes_\fo \fo/\m_\fo$ is reduced.
\end{propsub}
\begin{proof}
Consider the following diagrams where both arrows are formally smooth.
\begin{equation*}
\xymatrix{
 &
\wt D^{\bfv}_{\gM_\F,\rho_R} \ar[rd] \ar[ld] &
\\
D^{\bfv}_{\gM_\F,\rho_R}&
&\nf M_\bfv
}\end{equation*}
Since $\nf M_\bfv$ has all the desired properties, the same holds for $D^{\bfv}_{\gM_\F,\rho_R}$ which is represented by the completion of  $\GR^{\bfv}_{\rho_R}$  at a closed point (with residue field $\F$). Passing to a finite extension of $\F$ if necessary, we cover the completions of  $\GR^{\bfv}_{\rho_R}$  at all closed points, so we conclude the proof.
\end{proof}

\parag
Let $\GR^{\bfv}_{\rho_R,0}$ denote the fiber over the closed point of $\Spec R$ under the structure morphism $\GR^{\bfv}_{\rho_R}\ra \Spec R$. 
For a scheme $X$, we let  $H_0(X)$ denote the set of connected components  of $X$. The following corollary plays an important role later.
\begin{corsub}\label{cor:FormalIdemp2}
Assume that the morphism $\Spf R \ra D^{\leqs1}_\infty$ induced by $\rho_R$ is formally smooth. Then the natural maps below
\begin{displaymath}
H_0(\GR^{\bfv}_{\rho_R}\otimes_{\Zp}\Qp) \ra H_0(\GR^{\bfv}_{\rho_R}) \la H_0(\GR^{\bfv}_{\rho_R,0})
\end{displaymath}
are bijective.
\end{corsub}
\begin{proof}
The second bijection follows from the theorem on formal functions.
The first bijection follows from an argument similar to the proof of \cite[Corollary 2.4.10]{Kisin:ModuliGpSch} using $\fo$-flatness and the reducedness of $\GR^{\bfv}_{\rho_R} \otimes_\fo \fo/\m_\fo$.
\end{proof}

\subsection{}
We digress to study  analogues of connected, multiplicative-type, and unipotent finite flat group schemes over $\fo_K$, respectively, for $\vphi$-module of height $\leqs 1$. All the results in this section are proved in  \cite[\S1]{Kisin:ModuliGpSch}. 

\parag\label{par:DefEtPhiNilp}
 Let $\gM$ be an object either in$(\Mod/\Sig)^{\leqs1}$ or in $\ModFIBT A$ where $A$ is either $p$-adically separated and complete or finite over $\Qp$. 
We say $\gM$ is \emph{\'etale}  if $\vphi_{\gM}$ is an isomorphism. We say $\gM$ is \emph{$\vphi$-nilpotent}  if there exists $n$ such that $\vphi^n_{\gM}(\sig^n\gM)\subset\m_\Sig\tim \gM$ as $\Sig$-module. 
We say that $\gM$ is \emph{of Lubin-Tate type (of height $1$)} (respectively,  \emph{$\vphi$-unipotent}\footnote{Note that $\gM\in\ModFIBT A$ being $\vphi$-unipotent does \emph{not} imply that there exists a $\Sig_A$-basis so that the corresponding matrix representation of $\vphi_{\gM}$ is strictly upper triangular. This notion should be thought of as an analogue of unipotent finite flat group schemes/Barsotti-Tate groups.}) if $\gM^\vee$ is \'etale (respectively, $\vphi$-nilpotent), where $(\cdot)^\vee$ is dual of height $1$ as defined in \S\ref{par:CarDual}\footnote{We extend duality of height $1$ to $\ModFIBT A$ for a finite $\Qp$-algebra $A$ in an obvious way.}. When $T_A:=\nf T^{\leqs 1}_\Sig(\gM)$ for some $\gM\in\ModFIBT A$ (which makes sense because $A$ is finite over either $\Zp$ or $\Qp$), one can check that $T_A(-1)$ is unramified if $\gM_A$ is \'etale, and that $T_A$ is unramified if $\gM_A$ is of Lubin-Tate type.\footnote{This is because of the ``Tate twist'' in the definition of $\nf T^{\leqs 1}_\Sig$. An analogue of this in Dieudonn\'e theory is that that the roles of Frobenius and Verschiebung are switched in the \emph{covariant} Dieudonn\'e module.} (Here, $T_A(-1):=T_A\otimes \chi_{\cyc}\iv|_{\GKinfty}$.) 

Assume, furthermore, that $\gM\in\ModFIBT A$ is of $\Sig_A$-rank $2$. We say that $\gM$ is \emph{ordinary} if $\gM$ is an extension of a rank-$1$ Lubin-Tate type $(\vphi,\Sig_A)$-module $\gM^\LT$  by a rank-$1$ \'etale  $(\vphi,\Sig_A)$-module $\gM^{\et}$. If $A$ is a finite $\Qp$-module, then $\gM\in\ModFIBT A$ is \'etale (respectively, of Lubin-Tate type; respectively, $\vphi$-nilpotent; respectively, $\vphi$-unipotent; respectively, ordinary) if and only if there exists a finite $\Zp$-subalgebra $A^\circ \subset A$ and a $\vphi$-stable $\Sig_{A^\circ}$-lattice $\gM_{A^\circ}\subset \gM$ with $\gM_{A^\circ}\in\ModFIBT{A^\circ}$ such that  $\gM_{A^\circ}$ has the same property.

Assume that $\rho_R$ is a rank-$2$ representation of $\GKinfty$ of height $\leqs1$, as before. We define the following subfunctor $D^{\ord}_{\Sig,\rho_R}\subset D^{\leqs1}_{\Sig,\rho_R}$ (respectively, $D^{\leqs1,\vphi\nilp}_{\Sig,\rho_R} \subset D^{\leqs1}_{\Sig,\rho_R}$; respectively, $D^{\leqs1,\vphi\unip}_{\Sig,\rho_R} \subset D^{\leqs1}_{\Sig,\rho_R}$) by requiring that $\gM_A$ should be ordinary (respectively, $\vphi$-nilpotent; respectively, $\vphi$-unipotent). Clearly we have $D^{\ord}_{\Sig,\rho_R}\subset D^{\bfv}_{\Sig,\rho_R}$, and $\gM_A\in D^{\bfv}_{\Sig,\rho_R}(A,I)$ with $(A,I)\in\aug R$ is $\vphi$-nilpotent if and only if it is not ordinary.  Put $D^{\ssing}_{\Sig,\rho_R}:=D^{\leqs1,\vphi\nilp}_{\Sig,\rho_R}\cap D^{\bfv}_{\Sig,\rho_R}=D^{\leqs1,\vphi\unip}_{\Sig,\rho_R}\cap D^{\bfv}_{\Sig,\rho_R}$ (where the superscript `ss' stands for `supersingular').  

The following proposition is a direct consequence of semilinear algebraic statement \cite[Proposiiton 2.4.14]{Kisin:ModuliGpSch} applied to the universal $\Sig$-lattice of height $\leqs 1$.
\begin{propsub}\label{prop:DisjOrdSS}
The subfunctors $D^{\ord}_{\Sig,\rho_R},D^{\leqs1,\vphi\nilp}_{\Sig,\rho_R}, D^{\leqs1,\vphi\unip}_{\Sig,\rho_R}\subset D^{\leqs1}_{\Sig,\rho_R}$ are represented by $\GR^{\ord}_{\rho_R}, \GR^{\leqs1,\vphi\nilp}_{\rho_R}, \GR^{\leqs1,\vphi\unip}_{\rho_R}$, respectively, which are unions of connected components in $\GR^{\leqs1}_{\rho_R}$. Consequently, $D^{\ssing}_{\Sig,\rho_R}$ can be represented by a union $\GR^{\ssing}_{\rho_R}$ of connected components in $\GRv_{\rho_R}$, and $\GRv_{\rho_R}=\GR^{\ord}_{\rho_R}\coprod \GR^{\ssing}_{\rho_R}$.
\end{propsub}

\parag\label{par:padicOrdGKinfty}
Let $A$ be either finite flat over $\Zp$ or finite over $\Qp$. Let $\rho_A$ be a rank-$2$ representation of $\GKinfty$ over  $A$ with height $\leqs1$. Let $T_A$ denote the underlying $A$-module for $\rho_A$. We say that $\rho_A$ is \emph{ordinary} if there exists a $\GKinfty$-stable $A$-line $L_A\subset T_A$ (i.e. both $L_A$ and $T_A/L_A$ are free of rank $1$ over $A$) such that both $L_A(-1)$ and  $T_A/L_A$ are unramified. We say that $\rho_A$ is \emph{supersingular} if $\rho_A$ is not ordinary and $I_{K_\infty}$ acts on $\det(\rho_A)$ via $\chi_{\cyc}|_{I_{K_\infty}}$. When $A$ is finite over $\Qp$, one can show that $\rho_A$ is ordinary (respectively, supersingular) if and only if the unique $\gM_A\in\ModFIBT A$ that gives rise to $\rho_A$ is ordinary (respectively, supersingular).

Let $\Spec R^{\ord}$ and $\Spec R^{\ssing}$ denote the scheme-theoretic images of $\GR^{\ord}_{\rho_R}$ and $\GR^{\ssing}_{\rho_R}$ in $\Spec R$, respectively. It follows that for any finite $\Qp$-algebra $A$, a map $\xi:R\ra A$ factors through $R^{\ord}$ (respectively, $R^{\ssing}$) if and only if $\rho_R\otimes_{R,\xi}A$ is ordinary (respectively, supersingular). By Propositions~\ref{prop:GenIsom} and \ref{prop:DisjOrdSS} $\Spec R^\bfv[\ivtd p]$ is a disjoint union of $\Spec R^{\ord}[\ivtd p]$ and $\Spec R^{\ssing}[\ivtd p]$, where $\Spec R^\bfv$ is the scheme-theoretic image of $\GRv_{\rho_R}$ in $\Spec R$. 
%

We now state the following theorem, which is the main result of this section.

\begin{prop}\label{prop:ConndOrdNonOrd}
Assume that the morphism $\Spf R \ra D^{\leqs1}_\infty$ induced by $\rho_R$ is formally smooth. Then $\Spec R^{\ssing}[\ivtd p]$ is geometrically connected. Furthermore $\Spec R^{\ord}[\ivtd p]$ is geometrically connected unless $\bar\rho_\infty\sim \bar\psi_1\oplus \bar\psi_2$ 
where $\bar\psi_1$ and $\bar\psi_2$ are distinct unramified characters. In the exceptional case there are exactly two components (which are geometrically connected), and two lifts are in the same component if and only if the $\GKinfty$-stable lines on which $I_{K_\infty}$ act via $\chi_{\cyc}|_{I_{K_\infty}}$ reduce to the same character (either $\bar\psi_1$ or $\bar\psi_2$). 
\end{prop}
\begin{proof}[Sketch of the proof]
Since the statement is insensitive of $\fo$ and the setting is compatible with scalar extension of $\fo$ by its finite integral extension , it is enough to prove the connectedness assertions. By Proposition~\ref{prop:GenIsom} and Corollary~\ref{cor:FormalIdemp2}, the theorem follows from the corresponding connectedness assertions on $\GR^{\ord}_{\rho_R,0}:=\GR^{\ord}_{\rho_R}\otimes_R R/\m_R$ and  $\GR^{\ssing}_{\rho_R,0}:=\GR^{\ssing}_{\rho_R}\otimes_R R/\m_R$.

Let us first treat the ``ordinary'' case. Assume that $\GR^{\ord}_{\rho_R}$ is not empty (so automatically $\bar\rho_\infty$ is reducible). For a finite $\F$-algebra $A$, let  $T_A\cong A^{\oplus2}$ with $\GKinfty$-action via $\bar\rho_\infty\otimes_\F A$. Choose a $\GKinfty$-stable $A$-line $L_A\subset T_A$ such that $L_A(-1)$ and $T_A/L_A$ are unramified (which exists by assumption). 
Consider an exact sequence of \'etale $\vphi$-modules $(\dagger):\ 0\ra M_{L_A} \ra M_A \ra M_{T_A/L_A} \ra 0$ such that $\nf T_\Eps(\dagger)(1)$ is the chosen exact sequence $0\ra L_A \ra T_A  \ra T_A/L_A \ra 0$.

Let $\gM^{\et}\subset M_{L_A}$ be a $(\vphi,\Sig)$-submodule which is \'etale as an object in $(\Mod/\Sig)^{\leqs1}$, and $\gM^\LT\subset M_{T_A/L_A}$ a $(\vphi,\Sig)$-submodule which is of Lubin-Tate type as an object in $(\Mod/\Sig)^{\leqs1}$. Note that $\gM^{\et}\subset M_{L_A}$ and $\gM^\LT\subset M_{T_A/L_A}$ are respectively maximal and minimal among $(\vphi,\Sig)$-submodules which are objects in $(\Mod/\Sig)^{\leqs1}$.\footnote{Note that shrinking the $(\vphi,\Sig)$-submodule has the effect of ``shrinking'' the image of $\vphi$ on the $\Sig$-submodule.} 

Now we claim that there are at most one $(\vphi,\Sig)$-submodule $\gM\subset M_A$ which is an extension of $\gM^\LT$ by  $\gM^{\et}$. (Such $\gM$ is automatically of height $\leqs 1$.) Indeed, if $\gM,\gN\subset M_A$ are two such, then so are $\gM+\gN$ and $\gM\cap \gN$. therefore by Lemma~\ref{lem:MaxMinProlongation} there exist maximal and minimal $(\vphi,\Sig)$-submodules in $M_A$ which are extensions of $\gM^\LT$ by  $\gM^{\et}$: call them $\gM_{\ord}^+$ and  $\gM_{\ord}^-$, respectively. But the 5-lemma for the category of $\Sig$-modules (which is an abelian category) shows that the natural inclusion $\gM_{\ord}^- \hra\gM_{\ord}^+$ is in fact an isomorphism. 

So if either $\bar\rho_\infty$ is indecomposable or $\chi_{\cyc}|_{\GKinfty}$ is not unramified mod $p$, then $\GR^{\ord}_{\rho_R,0}\cong\Spec\F$. If $\bar\rho_\infty\sim \bar\psi_1\oplus \bar\psi_2$ where $\bar\psi_1$ and $\bar\psi_2$ are distinct unramified characters, then $\GR^{\ord}_{\rho_R,0}$ consists of  two reduced $\F$-rational points which correspond to the choice $L_\F \sim \bar\psi_i$ for $i=1,2$. In the remaining case (i.e., when $\bar\rho_\infty\sim \bar\psi\oplus \bar\psi$ for an unramified character $\bar\psi$), we have a functorial isomorphism $\GR^{\ord}_{\rho_R}(A) \cong \Proj^1(A)$ for any finite $\F$-algebra $A$, so $\GR^{\ord}_{\rho_R,0}$ is smooth projective over $\F$ with the same zeta function as $\Proj^1$. Since the zeta function determines the dimension and the genus, we obtain $ \GR^{\ord}_{\rho_R,0}\cong\Proj^1_\F$ by Chevalley-Waring. This concludes the ordinary case.

To show the connectedness of $\GR^{\ssing}_{\rho_R,0}$ we give an explicitly construction of a chain of rational curves that connects any two points. We will only give the main idea and indicate references for the detailed computations. 

Let $\F'$ be a finite extension of $\F$ and choose an $\F'$-valued point $x\in\GR^{\ssing}_{\rho_R,0}(\F')$. Let $\gM_{\F'} \in D^{\ssing}_{\Sig,\rho_R}(\F')$ be the corresponding torsion $\vphi$-module. (See \S\ref{par:DefEtPhiNilp} for the definition of $D^{\ssing}_{\Sig,\rho_R}$.) We construct a rational curve passing through $x$ as follows. First fix a $\Sig_{\F'}$-basis of $\gM_{\F'}$, and let $X$ be a matrix representation of $\vphi_{\gM_{\F'}}$ for this fixed basis. For a finite extension $\F''$ of $\F'$, consider a nilpotent matrix $N\in \Mat_2(\fo_{\Eps,\F''})$  such that $\sig(N)X(-N)\in \Mat_2(\Sig_{\F''})$. Define a $(\vphi,\Sig_{\F''[T]})$-module $\gM_{\F''[T]}^{(N)}$ to be $(\Sig_{\F''[T]})^{\oplus2}$ equipped with $\vphi_{\gM_{\F''[T]}^{(N)}}$ given by the matrix $\sig(1+TN)X(1+TN)\iv$. Then since the determinant of $\sig(1+TN)X(1+TN)\iv$ is equal to $\det(X)\in\PP(u)\tim(\Sig_{\F''})\starr$, we have $\gM_{\F''[T]}^{(N)}\in\ModFIBT{\F''[T]}$ and $\im(\vphi_{\gM_{\F''[T]}^{(N)}})/\PP(u)\gM_{\F''[T]}^{(N)}\subset \gM_{\F''[T]}^{(N)}/\PP(u)\gM_{\F''[T]}^{(N)}$ is a Lagrangian in the sense of Definition~\ref{def:Lagrangian}. Also, ``multiplication by $1+TN$'' defines an isomorphism $\gM_{\F'}\otimes_{\Sig_{\F'}}\fo_{\Eps,\F''[T]} \riso \gM_{\F''[T]}^{(N)}[\ivtd u]$. With this isomorphism, we can view $\gM_{\F''[T]}^{(N)}\in D^\bfv_{\Sig,\rho_R}(\F''[T])$, so we obtain an $\F$-morphism $\Spec\F''[T] \ra \GRv_{\rho_R,0}$ where the point ``$T=0$'' is mapped to $x$ naturally viewed as an $\F''$-point. (In particular, the image lands in $\GR^{\ssing}_{\rho_R,0}$.) Now, by specializing at the point ``$T=1$'' on this curve, we obtain an $\F''$-valued point $\gM^{(N)}_{\F''[T]}/(T-1)$ with $\vphi$ defined by the matrix $\sig(1+N)X(1+N)\iv$. 

To show that $\GR^{\ssing}_{\rho_R,0}$ is connected, we proceed as follows. First, we find an $\gM_{\F'} \in D^{\ssing}_{\Sig,\rho_R}(\F')$ for which we can write down a matrix representation $X$ of $\vphi_{\gM_{\F'}}$ for some basis. This is done in \cite[Lemma~1.2]{Imai:ConndComp} generalizing  \cite[Lemma~2.5.5]{Kisin:ModuliGpSch} when the residue field $k$ of $K$ is $\Fp$. Now, it is enough to show that any  $\gN_{\F''}\in D^{\ssing}_{\Sig,\rho_R}(\F'')$ with $\F''$ some finite extension of $\F$, has a basis so that the corresponding matrix representation of $\vphi_{\gN_{\F'}}$ is $\sig(\prod^n_i U_i)X(\prod^n_i U_i)\iv$ for $U_i = 1+N_i$ where $N\in \Mat_2(\fo_{\Eps,\mathfrak{E}})$ for some finite field $\mathfrak{E}$ containing both $\F'$ and $\F''$ such that $\sig(N_i)X(-N_i)\in \Mat_2(\Sig_{\mathfrak{E}})$. This purely linear-algebraic statement is proved first by Kisin \cite[\S2.5]{Kisin:ModuliGpSch} when the residue field $k=\Fp$, by Gee  \cite[\S2]{Gee:MLTwt2HMF}  when $\bar\rho_\infty$ is trivial, and by Imai \cite[\S2]{Imai:ConndComp} for the  general case. Note that their computations was applied to a different closed subscheme of an affine grassmannian (namely, the moduli of finite flat group schemes), but the same computation applies in our situation. While the aforementioned papers were written under the running assumption that $p>2$, it is only because the moduli of finite flat group schemes was not defined when $p=2$ and the same computation works for $p=2$.
\end{proof}

\section{Application to ``Barsotti-Tate deformation rings''}\label{sec:ApplBT}
The goal of this section is to show that the natural map $\res^{\leqs1}_{\cris}\otimes_{\Zp}\Qp$ defined in Corollary~\ref{cor:Restr} is an isomorphism if $p>2$ and it induces an isomorphism between certain unions of connected components. (See Theorem~\ref{thm:GabberDefor} for the precise statement.) 
This theorem is the bridge from Proposition~\ref{prop:ConndOrdNonOrd} (on the connected components of $R^{\Box,\bfv}_\infty[\ivtd p]$) to Theorem~\ref{thm:KisinConndComp} (on the connected components of $R^{\Box,\bfv}_{\cris}[\ivtd p]$).

\subsection{}\label{subsec:padicOrdGK}
Let $A$ be either finite flat over $\Zp$ or finite over $\Qp$, and let $T_A$ be a free $A$-module of rank $2$ equipped with continuous $\GK$-action which is crystalline with Hodge-Tate weights in $[0,1]$. Assume furthermore that  $I_K$ acts on $\det V$ via $\chi_{\cyc}|_{I_K}$. 

We say that $T_A$ as in the previous paragraph is \emph{ordinary} if there exists a $\GK$-stable $A$-line
$L_A\subset T_A$ (necessarily unique, if it exists) on which $I_K$ acts via $\chi_{\cyc}|_{I_K}$. (By our determinant condition, $T_A/L_A$ is forced to be unramified.) We say that $T_A$ as in the previous paragraph is \emph{supersingular} if it is not ordinary (or equivalently, if it does not admit a non-zero unramified quotient). When $A$ is a finite $\Qp$-algebra, one can easily show that a $\GK$-representation $V_A$ over $A$ is ordinary (respectively, supersingular) if and only if there exists a finite $\Zp$-submodule $A^\circ \subset A$ and a $\GK$-stable $A^\circ$-lattice $T_{A^\circ}\subset V_A$ such that $T_{A^\circ}$ is ordinary (respectively, supersingular). Similarly, for any finite flat $\Zp$-algebra $A^\circ$ a $\GK$-representation $T_{A^\circ}$ over $A^\circ$ is ordinary (respectively, supersingular) if $T_{A^\circ}[\ivtd p]$ is ordinary (respectively, supersingular) as a $\GK$-representation over $A^\circ[\ivtd p]$.

When $A$ is a finite $\Qp$-algebra, an equivalent definition of ordinary and supersingular representations can be made using $D_{\cris}^*(V_A)$. 

Let us define $R_{\cris}^{\Box,[0,1],f}:=R_{\cris}^{\Box,[0,1]}\otimes_{R^{\Box,\leqs 1}_\infty}R^{\Box,\leqs 1,\vphi\unip}_\infty$, $R_{\cris}^{\Box,\ord}:=R_{\cris}^{\Box,[0,1]}\otimes_{R^{\Box,\leqs 1}_\infty}R^{\Box,\ord}_\infty$, and $R_{\cris}^{\Box,\ssing}:=R_{\cris}^{\Box,[0,1]}\otimes_{R^{\Box,\leqs 1}_\infty}R^{\Box,\ssing}_\infty$. (We define $R_{\cris}^{\Box,[0,1],f}$ without requiring $\bar\rho$ to be $2$-dimensional. The superscript $(\cdot)^f$ stands for ``formal''.) 

\begin{propsub}\label{prop:OrdSSDeforRings}
Let $A$ be a finite $\Qp$-algebra. Then a map $\xi: R_{\cris}^{\Box,[0,1]} \ra A $ factors through $R_{\cris}^{\Box,[0,1],f}$ (respectively, $R_{\cris}^{\Box,\ord}$; respectively,  $R_{\cris}^{\Box,\ssing}$) if and only if the corresponding framed $A$-deformation $\rho_\xi$ admits no non-zero unramified quotient (respectively, $\rho_\xi$ is an ordinary $\GK$-representation; respectively, $\rho_\xi$ is a supersingular $\GK$-representaion.)
\end{propsub}
\begin{proof}
By definition, $\xi: R_{\cris}^{\Box,[0,1]} \ra A$ a map factors through the quotient $R_{\cris}^{\Box,[0,1],f}$ (respectively, $R_{\cris}^{\Box,\ord}$; respectively,  $R_{\cris}^{\Box,\ssing}$) if and only if $\rho_\xi|_{\GKinfty}$ does not admit non-zero unramified quotient (respectively, $\rho_\xi|_{\GKinfty}$ is ordinary; respectively, $\rho_\xi|_{\GKinfty}$  is supersingular). By Kisin's theorem (Theorem~\ref{thm:main}\eqref{thm:main:fullfth}), we have $\det\rho_\xi|_{I_K} \cong \chi_{\cyc}|_{I_K}$ if and only if $\det\rho_\xi|_{I_{K_\infty}} \cong \chi_{\cyc}|_{I_{K_\infty}}$. Now the proposition follows from Lemma~\ref{lem:OrdCrit}.
\end{proof}

\begin{lemsub}\label{lem:OrdCrit}
Let $A$ be a finite local $\Qp$-algebra, and $V_A$ a crystalline $A$-representation of $\GK$ with Hodge-Tate weights in $[0,1]$. Then the maximal unramified $A[\GKinfty]$-quotient $V^{\et}_A$ of $V_A$ is free over $A$ and is a $\GK$-equivariant quotient. (So $V_A^{\et}$ is also the maximal unramified $\GK$-quotient.)
%
\end{lemsub}
\begin{proof}
By Theorem~\ref{thm:KisinIntPAdicHT}\eqref{thm:KisinIntPAdicHT:BT} the natural projection $V_A\thra V_A^{\et}$ is $\GK$-equivariant, so it remains to show that $V_A^{\et}$ is free over $A$. 

By Proposition~\ref{prop:GenIsom} there exists a finite flat $\Zp$-subalgebra $A^\circ\subset A$ with $A^\circ[\ivtd p] = A$ and $\gM_{A^\circ}\in\ModFIBT{A^\circ}$ such that $T_{A^\circ}[\ivtd p] \cong V_A|_{\GKinfty}$ where $T_{A^\circ}:=\nf T_\Sig^{\leqs1}(\gM_{A^\circ})$. Let $T_{A^\circ}^{\et}$ denote the image of $T_{A^\circ}$ under the natural projection $V_A \thra V_A^{\et}$, and let $\gN_{A^\circ}$ denote the quotient of $\gM_{A^\circ}$ which corresponds to $T_{A^\circ}^{\et}$. Since $T_{A^\circ}^{\et}$ is a $\GKinfty$-equivariant $A^\circ$-quotient of $T_{A^\circ}$ and $T_{A^\circ}^{\et}[\ivtd p]=V_A^{\et}$, it follows that $\gN_{A^\circ}$ is a $\vphi$-equivariant $\Sig_{A^\circ}$-quotient of $\gM_{A^\circ}$ which is free over $\Sig$. Moreover, $\gN_{A^\circ}$ is free over $\Sig_{A^\circ}$ by \cite[Proposition~1.2.11]{Kisin:ModuliGpSch}, so $T_{A^\circ}^{\et}$ is free over $A^\circ$. 
\end{proof}
\begin{rmksub}
The statement of Lemma~\ref{lem:OrdCrit} still holds when $V_A$ is a  semi-stable $\GK$-representation with Hodge-Tate weights in $[0,h]$, so Proposition~\ref{prop:OrdSSDeforRings} can be generalized to this setting (with suitable the definitions of ordinary and supersingular representations of $\GK$ and $\GKinfty$). See Lemma~11.4.19 and  Proposition~11.4.18 in \cite{Kim:Thesis} for the statements and the proofs.
\end{rmksub}

Let us recall the notation. Let $\bar\rho:\GK\ra\GL_d(\F)$ for a finite extension $\F$ of $\Fp$ and put $\bar\rho_\infty:=\bar\rho|_{\GKinfty}$. Let $R_{\cris}^{\Box,[0,1]}$ denote the crystalline framed deformation ring of $\bar\rho$ with Hodge-Tate weights in $[0,1]$ (as in Kisin's theorem stated as Theorem~\ref{thm:CrysStDeforRing}), and let $R^{\Box,\leqs1}_\infty$ denote the $\GKinfty$-deformation ring of $\bar\rho_\infty$ with height $\leqs 1$ (which exists by Theorem~\ref{thm:Representability}).
Let $e$ be the absolute ramification index of $K$, and $e^*:=\frac{ep}{p-1}$. Let $I^{e^*}_K$ denote the higher ramification group with upper indexing. (For upper indexing we follow the  convention in \cite[Ch.~IV]{Serre:LocalFileds}, not the one in \cite{Fontaine:RamifEstimate}.) We prove the following theorem for the rest of this section and section \S\ref{sec:IntHT}.
\begin{thm}\label{thm:GabberDefor}
Assume that $\bar\rho(I^{e^*}_K) = \set1$. 
\begin{enumerate}
\item 
Assume that $p>2$. Then the natural map $\nf\res^{\cris}:\Spec R_{\cris}^{\Box,[0,1]}[\ivtd p]\ra \Spec R^{\Box,\leqs1}_\infty[\ivtd p]$ defined by the restriction to $\GKinfty$ is an isomorphism. 
\item Assume that $p=2$. Then the natural map $\nf\res^{\leqs1}_{\cris}:\Spec R_{\cris}^{\Box,[0,1],f}[\ivtd p]\ra \Spec R^{\Box,\leqs1,f}_\infty[\ivtd p]$ defined by the restriction to $\GKinfty$ is an isomorphism. 
\end{enumerate}
If furthermore $\End_{\GK}(\bar\rho)\cong\F$, then we also have $\End_{\GKinfty}(\bar\rho_\infty)\cong\F$ and the same statements hold for (unframed) deformation rings.
\end{thm}
As explained in Remark~\ref{rmk:ResNonIsom}, we do not expect the map $\nf\res^{\leqs h}_{\cris}:\Spec R_{\cris}^{\Box,[0,h]}[\ivtd p]\ra \Spec R^{\Box,\leqs h}_\infty[\ivtd p]$ is an isomorphism when $h>1$.
\begin{rmksub}
The assumption that $\bar\rho(I^{e^*}_K) = \set1$ is always satisfied if $\bar\rho$ comes from a finite flat group scheme over $\fo_K$ by Fontaine's ramification estimate \cite{Fontaine:RamifEstimate}. So this condition is indeed automatic if $R_{\cris}^{\Box,[0,1]}$ is not zero because any crystalline $\GK$-representation with Hodge-Tate weights in $[0,1]$ admits a $\GK$-stable lattice coming from a $p$-divisible group over $\fo_K$. See  \cite[Corollary~2.2.6]{kisin:fcrys} for the proof of this that works for any $p$. 
\end{rmksub}

\begin{rmksub}
Theorem~\ref{thm:GabberDefor} could easily be deduced using the \emph{classification of finite flat group schemes} when $p>2$ (stated in Theorem~\ref{thm:classifgpsch}) and the \emph{classification of connected finite flat group schemes} when $p=2$; more precisely, we deduce from the classification some equivalence of categories between certain classes of torsion $\GK$- and $\GKinfty$- representations as in \cite[Theorem 3.4.3]{Breuil:IntegralPAdicHodgeThy}, and from this Theorem~\ref{thm:GabberDefor} easily follows. In this paper, we present a different proof of Theorem~\ref{thm:GabberDefor} (hence, of Theorem~\ref{thm:KisinConndComp}), which avoids the classification of finite flat group schemes. 
\end{rmksub}
%
Let us discuss the proof of Theorem~\ref{thm:GabberDefor}. The statement for unframed deformation rings follows from the statement for framed deformation rings, provided that we have $\End_{\GKinfty}(\bar\rho_\infty)\cong\F$ when $\End_{\GK}(\bar\rho)\cong \F$. But this follows from Lemma~\ref{lem:ModpFinFlFullFaith} using that $\bar\rho(I^{e^*}_K) = \set1$. So we may focus on proving the statement for framed deformation rings.

By avoiding the classification of finite flat group schemes, we lose our grip on the torsion theory to some extent. So instead of trying to study the artinian points of the deformation ring (concentrated on the closed point), we use the following theorem of Gabber to study the $\Qp$-fiber of the deformation ring more directly. We state Gabber's theorem in the form that we will use in our situation, but the original statement in \cite[Appendix]{Kisin:ModularityGeomGalRep} is more general.
\begin{thm}[Gabber]\label{thm:Gabber}
Let $R$ and $R'$ be complete semi-local noetherian $\fo$-algebras with residue fields finite over $\F\cong\fo/\m_\fo$. Assume that both $R$ and $R'$ are $p$-torsion free, $R$ is reduced, and $R'$ is normal.
Let $f:R'\ra R$ be a $\fo$-algebra map such that the induced map $\Hom_\fo(R,E) \ra \Hom_\fo(R',E)$ is a bijection for any finite (field) extension $E$ of $\Frac(\fo)$.
Then $f$ is an isomorphism.
\end{thm}
This theorem is certainly very delicate -- it is not even obvious how to deduce that $f$ is of finite type from the assumption. The proof uses the flattening technique of Raynaud-Gruson (and very ingenious commutative algebra). See Gabber's appendix in \cite{Kisin:ModularityGeomGalRep} for more details.

\subsection{}\label{par:ApplGabber}
We outline how to use Theorem~\ref{thm:Gabber} to prove Theorem~\ref{thm:GabberDefor} (for framed deformation rings). Let $R$ and $R_\infty$ be one of the following:
\begin{enumerate}
\item Assuming $p>2$, we set $R:=R^{\Box, [0,1]}_{\cris}$ and $R_\infty:=R^{\Box, \leqs1}_\infty$.
\item Under no assumption on $p$, we set $R:=R^{\Box, [0,1], f}_{\cris}$ and $R_\infty:=R^{\Box, \leqs1,f}_\infty$.
\end{enumerate}
In all the cases above, the restriction to $\GKinfty$ induces a natural map $\nf\res:R_\infty \ra R$.
Note that we cannot directly apply Theorem~\ref{thm:Gabber} to this map because $R$ and $R_\infty$ may not be normal (or even reduced). We fix this situation by applying normalization, as follows.  Note that both $R[\ivtd p]$ and $R_\infty[\ivtd p]$ are formally smooth over $\Frac(\fo)$ by Propositions~\ref{prop:FormalSm} and \ref{prop:DisjOrdSS}, and \cite[Theorem~3.3.8]{kisin:pstDefor}. We let $\wt R_\infty$ and$\wt R$ be the normalizations of $R_\infty/(p^\infty\text{-torsions})$ and $R/(p^\infty\text{-torsions})$, respectively. Note that the natural maps induce isomorphisms $R_\infty[\ivtd p] \riso \wt R_\infty[\ivtd p]$ and $R[\ivtd p] \riso \wt R[\ivtd p]$, and the map $\nf\res[\ivtd p]:R_\infty[\ivtd p] \ra R[\ivtd p]$ restricts to $\wt{\nf\res}:\wt R_\infty \ra \wt R$. Furthermore, $\wt R_\infty$ and $\wt R$ are finite over $R_\infty$ and $R$, respectively, since every complete local noetherian ring is excellent 
\cite[IV$_\textrm{2}$, (7.8.3)(iii)]{EGA}, so $\wt R_\infty$ and $\wt R$ satisfy the assumptions on $R'$ and $R$, respectively, in Theorem~\ref{thm:Gabber}. 

Now, we prove the following proposition.
\begin{prop}\label{prop:Gabber}
Let $R$ and $R_\infty$ be as above \S\ref{par:ApplGabber}.
Then for any finite (field) extension $E$ of $\Frac(\fo)$, the map  $\Hom_\fo(R,E) \ra \Hom_\fo(R',E)$ induced by $\nf\res:R_\infty \ra R $ is a bijection.
\end{prop}
By Gabber's theorem (Theorem~\ref{thm:Gabber}), the proposition implies that the map $\wt{\nf\res}:\wt R_\infty \ra \wt R$ induced on the normalized deformation rings is an isomorphism, so in particular, $\nf\res[\ivtd p]: R_\infty[\ivtd p] \ra R[\ivtd p]$ is an isomorphism, which proves Theorem~\ref{thm:GabberDefor} for framed deformation rings.


\subsection{}\label{par:ProofBijectivity}
We outline the proof of Proposition~\ref{prop:Gabber}. Let $R$ and $R_\infty$ be as in \S\ref{par:ApplGabber}. By Kisin's theorem (Theorem~\ref{thm:KisinIntPAdicHT}\eqref{thm:KisinIntPAdicHT:crys}) we see that $\nf\res$ induces an injective map  $\Hom_\fo(R,E) \ra \Hom_\fo(R_\infty,E)$ for any finite field extension $E$ of $\Frac(\fo)$, so it is left to show the surjectivity. Let $x\in \Hom_\fo(R_\infty,E)$ and let $V_x$ be the corresponding $E$-representation of $\GKinfty$, which is of height $\leqs 1$ by construction. . By Kisin's theorem (Theorem~\ref{thm:KisinIntPAdicHT}\eqref{thm:KisinIntPAdicHT:BT}), there exists a unique continuous $E$-linear $\GK$-action on $V_x$ which extends the $\GKinfty$-action and makes $V_x$ a crystalline $\GK$-representation.
Since the $\fo$-algebra map $x:R_\infty \ra E$ factors through $\fo_E$, we also obtain a $\GKinfty$-stable $\fo_E$-lattice $T_x\subset V_x$, such that $T_x \otimes_{\fo_E} \fo_E/\m_E \cong \bar\rho_\infty \otimes_\F \fo_E/\m_E$ as a $\GKinfty$-representation. In order to show that $x: R_\infty \ra E$ factors through the crystalline framed deformation ring $R$ (i.e., Proposition~{prop:Gabber}), we need to verify the following claims. 
\begin{description}
\item[Claim~(1)] The $\GKinfty$-stable $\fo_E$-lattice $T_x\subset V_x$ is $\GK$-stable.
\item[Claim~(2)] We have a $\GK$-equivariant isomorphism $T_x \otimes_{\fo_E} \fo_E/\m_E \cong \bar\rho\otimes_\F \fo_E/\m_E$ extending the initial such $\GKinfty$-isomorphism.
\end{description}

The following lemma proves \textbf{Claim~(2)}.
\begin{lem}\label{lem:ModpFinFlFullFaith}
Let $\bar\rho$ be a mod $p$ representation of $\GK$, and assume that $\bar\rho(I^{e^*}_K)=\set1$ as in Theorem~\ref{thm:GabberDefor}. Then we have $\bar\rho(\GKinfty) =  \bar\rho(\GK)$. In particular, the $\GK$-representation $\bar\rho$ is uniquely determined by its restriction to $\GKinfty$.
\end{lem}
\begin{proof}
To show the lemma, it is enough to show that the natural inclusion induces an isomorphism\footnote{The author learned this idea from \cite[\S8.3.2, Theorem~C]{Abrashkin:GrSchPerNot2}.}:
\begin{equation}\label{eqn:RamifEst}
\GKinfty/(I^{e^*}_{ K }\cap\GKinfty) \riso \GK/I^{e^*}_{ K }.
\end{equation}
To rephrase the isomorphism \eqref{eqn:RamifEst}, we want to show that the open subgroup $I^{e^*}_{ K }\tim\GKinfty\subset \GK$ fills up the full Galois group $\GK$; i.e., there is no non-trivial subextension of $ K _\infty/ K $ fixed by $I^{e^*}_{ K }$. This follows from the claim below, which is a nice exercise with higher ramification groups.
\begin{claimsub}
Let $ K _1:= K (\pi^{(1)})$ where $\pi^{(1)}\in K _\infty$ such that $(\pi^{(1)})^p=\pi$.
Then $I^{e^*}_{ K }$ does \emph{not} fix $ K _1$.
\end{claimsub}
Put $ K ':= K (\zeta_p)$ where $\zeta_p\in\ol K $ is a primitive $p$-th root of unity, and consider $ K _1':= K '(\pi^{(1)})$, which is a Galois closure of $ K _1/ K $.
We put $\gal:=\Gal( K _1'/ K ) \cong \Gal( K _1'/ K ')\rtimes \Gal( K _1'/ K _1)$. Here, $\Gal( K _1'/ K ')\cong \Z/p\Z$ is the wild inertia subgroup of $\gal$, and $\Gal( K _1'/ K _1)\subset (\Z/p\Z)\starr$ acts on $\Gal( K _1'/ K ')$ by Kummer theory.

Since the upper indexing is well-behaved under passing to quotients \cite[IV.~\S3, Proposotion~14]{Serre:LocalFileds}, it is enough to show that $\gal^{e^*}$ does not fix $ K _1$. Indeed, we show that  $\gal^{e^*} = \Gal( K _1'/ K ')$ by computing the higher ramification subgroups $\gal_i$ in the lower numbering and using the Herbrand function, exploiting the explicitness of the situation.

Clearly, $\gal_1 = \Gal( K _1'/ K ')$, and $\gal_i$ is a subgroup of $\Gal( K _1'/ K ')$ for all $i>0$. Let $c:=\abs{I_{ K _1'/ K _1}}$ denote the ramification index of $ K _1'/ K _1$, so the absolute ramification index of $ K _1'$ is $epc$. (Also note that $[\gal_0:\gal_1]=c$.) 
Choose a uniformizer $\pi'_1$ of  $K_1'$ so that $(\pi'_1)^p\in K'$; this is possible because there exists a uniformizer $\pi'\in K'$ such that $(\pi')^c \in [\alpha]\pi$ for some $\alpha \in \fo_{K'}/\m_{K'}$, where $[\cdot]$ denotes the Teichm\"uller lift, so we have $K'(\sqrt[p]{\pi}) = K'(\sqrt[p]{\pi'})$ in $\ol K$.

For any non-trivial $\gamma \in \gal$, we have 
\begin{equation*}
v_{ K _1'}\left(\gamma(\pi'_1)-\pi'_1\right)-1 \
= v_{ K _1'}\left(\zeta_p^{\varepsilon(\gamma)}-1\right) + v_{ K _1'}\left(\pi'_1\right) -1 \
=e^*c,
\end{equation*}
where $\varepsilon$ is the Kummer cocycle associated to $\pi'_1$.
This shows that $\gal_i=\Gal( K _1'/ K ')$ for $0< i \leq e^*c$, and $\gal_i=\set\id$ for $i>e^*c$. Since $[\gal_0:\gal_i]=c$ for $0< i \leq e^*c + c -1$, we obtain
\begin{equation*}
\begin{array}{ll}
\gal^r = \Gal( K _1'/ K '), &\textrm{for }0< r \leq e^*\\
\gal^r =  \set\id,  &\textrm{for } r > e^*.
\end{array}
\end{equation*}
In particular, $\gal^{e^*}=\Gal( K _1'/ K ')$ does not fix $ K _1$.
\end{proof}

For \textbf{Claim~(1)}, we use Breuil's theory of strongly divisible lattices (of weight $\leqs1$). We only use the fact that one can naturally associate a $\GK$-stable $\Zp$-lattice of $V_x$ to a strongly divisible ``lattice'' by purely semilinear algebra means (i.e., without relating the strongly divisible modules of weight $\leqs 1$ with Barsotti-Tate groups over $\fo_{ K }$). Since it takes a significant digression to introduce the relevant definitions, we carry out these steps in a separate section \S\ref{sec:IntHT}.

%
 
\section{Integral $p$-adic Hodge theory \`a la Breuil}\label{sec:IntHT}
In this section we give a proof of \textbf{Claim~(1)} in \S\ref{par:ProofBijectivity}, hence a proof of Theorem \ref{thm:KisinConndComp}. More precisely, we prove the following theorem.

\begin{prop}\label{prop:Claim1}
Let $V$ be a $p$-adic crystalline $\GK$-representation with Hodge-Tate weights in $[0,1]$. If $p=2$, then we additionally assume that either $V$ or $V^*(1)$ does not have any non-zero unramified quotient. Then any $\GKinfty$-stable $\Zp$-lattice $T\subset V$ is $\GK$-stable.
\end{prop}
The proof of this proposition is ``embedded'' in the literature\footnote{For example, Proposition~\ref{prop:Claim1} was used in Kisin's proof of classification of finite flat group schemes when $p>2$ and connected finite flat group schemes when $p=2$. See the proof of \cite[Theorem~2.2.7]{kisin:fcrys}, and \cite[Proposition~1.2.7]{Kisin:2adicBT}.}, but we present the proof for convenience of readers. We prove this proposition using strongly divisible modules introduced by Breuil. We recall some basic constructions used in the proof of the theorem, and direct interested readers to  \cite{Breuil:IntegralPAdicHodgeThy} for an overview of the theory, and \cite{Liu:StronglyDivLattice} for more recent developments in this subject.

\subsection{Construction}\label{par:MM}
Let $V$ be as in Proposition~\ref{prop:Claim1} and $T\subset V$ a $\GKinfty$-stable $\Zp$-lattice. By replacing $V$ with $V^*(1)$ if necessary, we assume that $V^*(1)$ does not have any non-zero unramified quotient if $p=2$. Let $D:=\nf D^*_{\cris}(V)$ be the contravariantly\footnote{In this section, contravariant functors are more intuitive and convenient to work with.} associated filtered $\vphi$-module and $\gM$ be a $\vphi$-module of height $\leqs1$ such that $T\cong \left(\nf T^{\leqs1}_\Sig(\gM)\right)^*(1)$. When $p=2$, our assumption on $V$ implies that $D$ has no non-zero admissible quotient that is pure of slope $1$,\footnote{i.e.,  $D$ is ``unipotent'' in the sense of \cite[Definition 2.1.1]{Breuil:IntegralPAdicHodgeThy}.}; or equivalently by Lemma~\ref {lem:OrdCrit}, that  $\gM$ is $\vphi$-unipotent in the sense of \S\ref{par:DefEtPhiNilp}. Under these assumptions, we will prove Proposition~\ref{prop:Claim1} by constructing  from $\gM$ another \emph{$\GK$-stable} $\Zp$-lattice $T'\subset V$ and show that $T=T'$.

Let $S$ be the $p$-adic completion of the divided power envelop of $W(k)[u]$ with respect to the ideal generated by $\PP(u)$. It can be shown that $S$ can be viewed as a subring of $ K _0[[u]]$ whose elements are precisely those of the form $\sum_{i\geq0} a_i \frac{u^i}{q(i)!}$, where $q(i):=\lfloor \frac{i}{e} \rfloor$ with $e:=\deg\PP(u)$, and $a_i\in W(k)$ converge to $0$ as $i\to\infty$. We define a differential operator $N:=-u\frac{d}{du}$ on $S$. We define $\sig:\self S$ via extending the Witt vector Frobenius on the coefficients by $\sig(u)=u^p$. We let $\Fil^w S \subset S$ denote the ideal topologically generated by $\PP(u)^i/i!$ for $i\geqs w$. Since $p|\sig(f)$ for any $f\in\Fil^1S$, we can define a $\sig$-semilinear map $\sig_1:=\frac{\sig}{p}:\Fil^1 S \ra S$. 

We set $\MM:=S\otimes_{\sig,\Sig}\gM\cong S\otimes_\Sig (\sig^*\gM)$. We have an $S$-linear map $\id\otimes\vphi_\gM:\MM\cong S\otimes_\Sig (\sig^*\gM)\ra S\otimes_\Sig\gM$. Using this, we define a $S$-submodule  $\Fil^1\MM\subset \MM$ and $\vphi_1:\Fil^1\MM \ra \MM$ as follows.
\begin{eqnarray}
\label{eqn:Filh}\Fil^1\MM:=\set{x\in\MM|\ \id\otimes\vphi_\gM(x) \in \Fil^1S\otimes_\Sig \gM \subset S\otimes_\Sig \gM } \\
\label{eqn:vphir}\vphi_1:\Fil^1\MM \xra{\id\otimes \vphi_\gM} \Fil^1S\otimes_\Sig \gM \xra{\sig_1\otimes\id} S\otimes_{\sig,\Sig}\gM = \MM
\end{eqnarray}
Let $\vphi_{\MM}:=\sig\otimes\vphi_\gM$, and note that $\vphi_1 = \ivtd p (\vphi_{\MM}|_{\Fil^1\MM})$.

We now describe $\Fil^1\MM$ explicitly. Choose a $\Sig$-basis $\set{\e_1,\cdots,\e_d}$ of $\gM$ so that $\vphi(\sig^*\gM)$ is generated by $\set{\e_1,\cdots,\e_{d_1},\PP(u)\e_{d_1+1},\cdots,\PP(u)\e_d}$ for some $1\leqs d_1 \leqs d$.  We abusively let $\e_i\in \MM$ denote the element $1\otimes\e_i$. Then we have $\Fil^1\MM \cong \left(\bigoplus_{i=1}^{d_1}(\Fil^1S)\e_i\right)\oplus \left(\bigoplus_{i=d_1+1}^{d}S\e_i\right)$. Clearly $\vphi_1(\Fil^1\MM)$ generates $\MM $.

By  \cite[Proposition~5.1.3]{Breuil:GrPDivGrFiniModFil}, one can define a differential operator $N_{\MM}:\self{\MM}$ over $N:\self S$ (i.e., a $W(k)$-linear map $N_{\MM}$ such that $N_{\MM}(sm) = N(s)m+sN_{\MM}(m)$ for any $s\in S$ and $m\in \MM$). We recall the construction below. Let $m_i:= \PP(u)\e_i$ for $i\leqs d_1$ and $m_i=\e_i$ for $i>d_1$. Since $\set{m_i}$ is an $S$-basis of $\Fil^1\MM$, it follows that $\set{\vphi_1(m_i)}$ is an $S$-basis of $\MM$.  Define $N_0:\self{\MM}$ by  $N_0(\sum_is_i\vphi_1(m_i))=N(s_i)\vphi_1(m_i)$. Clearly, the image of $N_0$ is contained in $I\MM$ where $I\subset S$ is the ideal generated by $\frac{u^i}{q(i)}$ for all $i\geqs1$.   We recursively define $N_n$ for all $n\geqs1$ so that $N_n(\vphi_1(m_i)):= (\sig\otimes\vphi_\gM)\circ N_{n-1} (m_i)$. Clearly, $N_n(m) - N_{n-1}(m) \in \langle\sig^{n-1}(I)\rangle\tim \MM$ for any $m\in \MM$, so the limit $N_{\MM}:=\lim_{n\to\infty}N_n = N_0+\sum_{n=1}^\infty (N_n-N_{n-1})$ converges $p$-adically. By construction we have $N_{\MM}\circ\vphi_1 = (\vphi_{\MM}\circ N_{\MM})|_{\Fil^1\MM}$ and $N_{\MM}\equiv 0 \bmod{I\MM}$.

Using the work of Breuil \cite[\S6]{Breuil:GriffithsTransv} and Kisin (Theorem~\ref{thm:KisinIntPAdicHT}\eqref{thm:KisinIntPAdicHT:crys}), we obtain the following:
\begin{propsub}\label{prop:QpBreuil}
The tuple $(\MM[\ivtd p],\vphi_{\MM}[\ivtd p],N_{\MM}[\ivtd p],\Fil^1\MM[\ivtd p])$ is naturally isomorphic (with the obvious notion) to the following:
\begin{enumerate} 
\item an $S[\ivtd p]$-module $S\otimes_{W(k)}D$,
\item a $\sig$-linear map $\sig_S\otimes\vphi_D:\self{S\otimes_{W(k)}D} $,
\item a derivation $N\otimes\id_D::\self{S\otimes_{W(k)}D} $ over $N:\self S$,
\item an $S[\ivtd p]$-submodule $\set{m\in S\otimes_{W(k)}D|\ m\bmod{\PP(u)} \in\Fil^1D_K}$.
\end{enumerate}
In other words, $\MM$ is a strongly divisible lattice in $S\otimes_{W(k)}D$  in the sense of \cite[Definition~2.2.1]{Breuil:IntegralPAdicHodgeThy}. 
\end{propsub}
%
In order to construct from $\MM$ a $\GK$-stable $\Zp$-lattice in $V$, we need to introduce the ring $\Asthat$. We do this in the next section.

\subsection{} \label{par:Asthat}
Put $\rep:=\invlim_{x^p\la x} \fo_{\ol K }/(p)$. It is well-known that the $\kbar$-algebra $\rep$ is complete with respect to a naturally given valuation and $\Frac(\rep)$ is algebraically closed. We fix a uniformizer $\pi\in\fo_{ K }$ such that $\PP(\pi)=0$ and we choose successive $p$-power roots $\pi^{(n)}$; i.e., $\pi^{(0)}=\pi$ and $(\pi^{(n+1)})^p = \pi^{(n)}$. The sequence $\nf\pi:=\set{\pi^{(n)}}$ is an element of $\rep$. There exists a ``canonical lift'' $\theta:W(\rep)\thra \fo_{\C_{ K }}$ of the first projection $\rep \thra \fo_{\ol K }/(p)$, which is $\GK$-equivariant for the natural actions on both sides and is a topological quotient map (for the ``product topology'' on the source and the natural $p$-adic topology on the target).  Let $\BdR^+$ be the completion of $W(\rep)[\ivtd p]$ with respect to the kernel of $\theta[\ivtd p]$, and let $\BdR:=\BdR^+[\ivtd t]$, where $t$ is Fontaine's $p$-adic analogue of $2\pi i$. See \cite{fontaine:Asterisque223ExpII} for more details.

We define $\Acris$ as the $p$-adic completion of the divided power envelop of $W(\rep)$ with respect to $\ker(\theta)$. The Witt vector Frobenius map and the $\GK$-action on $W(\rep)$ extend to $\Acris$. We let $\Fil^w\Acris$ be the ideal topologically generated by $\ivtd{i!}(\ker\theta)^i$ for $i\geqs w$. We have $\sig(\Fil^1\Acris)\Acris \subset p\Acris$, so we can define $\sig_1:=\frac{\sig}{p}:\Fil^1\Acris \ra \Acris$.

We define $\Asthat$ to be the $p$-adic completion of the divided power polynomial ring $\Acris[\frac{X^i}{i!}]_{i\geq1}$. We  define a Frobenius endomorphism $\sig:\self{\Asthat}$ using $\sig:\self{\Acris}$ on the coefficients and $\sig(1+X)= (1+X)^p$. 
We define the ideals
\begin{equation*}
\Fil^h\Asthat:=\left\{\sum_{i\geq0} a_i \frac{X^i}{i!}\in\Asthat |\ a_i\in\Fil^{i-h}\Acris,\ \lim_{i\to\infty}a_i=0\right\},
\end{equation*}
where we set $\Fil^w\Acris:=\Acris$ for $w\leq 0$.
Then the map $\sig_1:=\frac{\sig}{p}:\Fil^1\Asthat \ra \Asthat$ is well-defined. Let $N:\self{\Asthat}$ be the $\Acris$-derivation $(1+X)\frac{d}{dX}$, and for any $\gamma\in\GK$, we define $\gamma (1+X):= \epsilon(\gamma)(1+X)$ where $\epsilon(\gamma):=\frac{\gamma[\nf\pi]}{[\nf\pi]} \in\Acris$. 

We briefly discuss the relation with $\Asthat$ and the period rings. Clearly, $\Acris$ has a natural filtered $\GK$-equivariant embedding into $\BdR^+$. We extend this to $\Asthat$ by sending $X$ to  $ \frac{[\nf\pi]}{\pi}-1\in\BdR^+$. This embedding is also filtered and $\GK$-equivariant. Furthermore, $\Bst^+$, as a subring of $\BdR^+$, is exactly the image of $\Acris[\log(1+X),\ivtd p]$; i.e. the subring of $\Asthat[\ivtd p]$ consisting of elements killed by some power of $N$. 

We also define filtered $W(k)$-algebra embeddings $S\hra \Acris$ by sending $u$ to $[\nf\pi]$ and $S\hra\Asthat$ by sending $u$ to $\frac{[\nf\pi]}{1+X}$. Both embeddings respect the Frobenius structures on both sides, and the derivation $N:\self{\Asthat}$ restricts to $N:\self S$. Furthermore, the image of $S$ in $\Acris$ is fixed by $\GKinfty$, and the image in $\Asthat$ is fixed by $\GK$. (In fact, one can even show that $S$ fills up $\Acris^{\GKinfty}$ and $(\Asthat)^{\GK}$, respectively. See \cite[\S4]{Breuil:GriffithsTransv} for the proof.) We also record that the filtered map $\Asthat\ra\Acris$ defined by $\frac{X^i}{i!}\mapsto0$ for all $i$ is a map of $S$-algebras which respects the Frobenius structures and $\GKinfty$-actions on both sides. 
%
%

\subsection{Construction of Galois-stable $\Zp$-lattices}\label{par:Tst}
We come back to the setting of \S\ref{par:MM}, and define a $\Zp[\GK]$-module $\nf T^*_{\st}(\MM)$ and a $\Zp[\GKinfty]$-module $\nf T^*_{\qst}(\MM)$, as follows:
\begin{eqnarray}
\nf T^*_{\st}(\MM)&:=&\Hom_{S,\vphi_1, N, \Fil^1}(\MM,\Asthat)\\
\nf T^*_{\qst}(\MM)&:=&\Hom_{S,\vphi_1, \Fil^1}(\MM,\Acris)
\end{eqnarray}
where $\GK$ acts on $\nf T^*_{\st}(\MM)$ through $\Asthat$, and $\GKinfty$ acts on $\nf T^*_{\qst}(\MM)$ through $\Acris$, respectively. Clearly both $\nf T^*_{\st}(\MM)$ and $\nf T^*_{\qst}(\MM)$ are $p$-adically separated and complete. Breuil showed that there exists a natural $\GK$-equivariant isomorphism $\nf T^*_{\st}(\MM)[\ivtd p] \riso V$, so we may identify $\nf T^*_{\st}(\MM)$ with a $\GK$-stable $\Zp$-lattice in $V$. We can construct this isomorphism as follows. By Proposition~\ref{prop:QpBreuil}, we can view $D$ as a $\vphi$-stable submodule in $\MM[\ivtd p]$ which is killed by $N_{\MM}$. Since $\Bcris^+\subset\Asthat[\ivtd p]$ is a subring of elements killed by $N$,  we obtain a $\GK$-equivariant $\Zp$-linear map $\Hom_{S[\ivtd p],\vphi_1, N, \Fil^1}(\MM[\ivtd p],\Asthat[\ivtd p]) \ra \Hom_{K_0,\vphi,\Fil}(D,\Bcris^+)$. To show that this is an isomorphism, see  \cite[Proposition~2.2.5]{Breuil:IntegralPAdicHodgeThy}  where the sketch of the proof and references of the full proof are given.\footnote{The proof is also worked out in \cite[Theorem~12.2.1.3]{Kim:Thesis}.} 

Note that there is a $\GKinfty$-equivariant map $\nf T^*_{\st}(\MM) \ra \nf T^*_{\qst}(\MM)$ induced by $\Asthat \thra \Acris$; $\frac{X^i}{i!} \mapsto 0$ for all $i$. This map is an isomorphism by \cite[Lemma 3.4.3]{Liu:StronglyDivLattice}. So we may view $\nf T^*_{\qst}(\MM)$ as the restriction of the $\GK$-action on $\nf T^*_{\st}(\MM)$ to $\GKinfty$, and the natural $\GKinfty$-equivariant embedding \begin{equation}\label{eqn:TqstEmb}
\nf T^*_{\qst}(\MM)=\Hom_{S,\vphi_1, \Fil^1}(\MM,\Acris)\hra  \Hom_{K_0,\vphi,\Fil}(D,\Bcris^+)\cong V
\end{equation} 
is induced the $\vphi$-equivariant embedding $D\hra \MM[\ivtd p]$.

Recall that we have constructed $\MM$ (with its extra structure) from the contravariantly associated $\vphi$-module $\gM$ to a $\GKinfty$-stable $\Zp$-lattice $T\subset V$. To compare $T$ with $\nf T^*_{\st}(\MM)$, we need to recall how we construct $T$ from $\gM$ in  \cite{kisin:fcrys}. Note that
\begin{equation*}
T \cong \left( \nf T^{\leqs1}_\Sig(\gM)\right)^*(1) \cong \Hom_{\Sig,\vphi}(\gM,\wh\fo_{\Eps^{\ur}}) \liso \Hom_{\Sig,\vphi}(\gM,\wh\Sig^{\ur}),
\end{equation*}
where $\wh\Sig^{\ur}$ is the topological closure of the integral closure of $\Sig$ in $ \wh\fo_{\Eps^{\ur}}$. (The endomorphism $\sig$ and the $\GKinfty$-action on $\wh\fo_{\Eps^{\ur}}$ stabilize $\wh\Sig^{\ur}$.) The arrow above is induced by the natural inclusion $\wh\Sig^{\ur}\hra \wh\fo_{\Eps^{\ur}}$, and it is  an  isomorphism by  \cite[\S{B.}~Proposition~1.8.3]{fontaine:grothfest}.  Note that the embedding $\Sig\hra\Acris$ by $u\ra[\nf\pi]$ extends to $\wh\Sig^{\ur}\hra\Acris$ uniquely in a way that respects the endomorphisms $\sig$ and $\GKinfty$-actions on both sides.

Now let us recall the construction of $\GKinfty$-equivariant embedding of $T$ into $V\cong \nf V^*_{\cris}(D)$. Let $\OO_\disk \subset K_0[[u]]$ be the subring which converges on the ``open unit disk''; i.e. $f(u)\in\OO_\disk$ is a formal power series over $K_0$ such that $f(x)$ converges for any $x\in\C_K$ with $|x|<1$. One can define an endomorphism $\sig:\self{\OO_\disk}$ so that it is the Witt vector Frobenius on the coefficients $K_0$ and $\sig(u)=u^p$. Then we have  $\Sig[\ivtd p] \subset \OO_\disk \subset S[\ivtd p]$ as subrings of $K_0[[u]]$, which are stable under the endomorphisms $\sig$ defined on each of them. In particular, one can embed $\OO_\disk$ into $\Bcris^+$ (since $S[\ivtd p]$ embeds into $\Bcris$).

Put $\M:=\OO_\disk \otimes_\Sig\gM$ equipped with a $\OO_\disk$-linear map $\vphi_\M:=\id_{\OO_\disk}\otimes\vphi_\gM:\sig^*\M\ra\M$ where $\sig^*\M:=\OO_\disk\otimes_{\sig,\OO_\disk}\M$. Let us define a filtration 
\begin{equation*}\Fil^w(\sig^*\M):=\left\{ m\in\sig^*\M|\ \vphi_\M(m) \in \PP(u)^w\M\right\}.\end{equation*}
Finally, by \cite[Lemma~1.2.6]{kisin:fcrys} there exists a $\vphi$-equivariant embedding $\xi:D \hra \M$ where $\xi(D)\subset \vphi_\M(\sig^*\M)$. Now we obtain $\GKinfty$-equivariant injective maps
\begin{equation}\label{eqn:TSigEmb}
 \Hom_{\Sig,\vphi}(\gM,\wh\Sig^{\ur}) \hra  \Hom_{\OO_\disk,\vphi,\Fil}(\sig^*\M,\Bcris^+) \hra  \Hom_{K_0,\vphi,\Fil}(D,\Bcris^+),
\end{equation}
where the first map is induced by $\vphi_\M$ on the first argument and the natural inclusion $\wh\Sig^{\ur} \hra\Bcris^+$ on the second argument, and the second map is induced by $\xi:D\hra\vphi_\M( \sig^*\M)$, using \cite[\S1.2.7, Proposition~1.2.8]{kisin:fcrys}. Now, \cite[Proposition~2.1.5]{kisin:fcrys} asserts that the composition of the arrows above, hence each arrow, is an isomorphism, so we obtain a $\GKinfty$-equivariant embedding $T\hra V$.

Now consider the following natural $\GKinfty$-equivariant morphism: 
\begin{equation}\label{eqn:TSigTqstLatt}
\mathfrak{T}:T \cong \Hom_{\Sig,\vphi}(\gM,\wh\Sig^{\ur}) \ra \Hom_{S,\vphi_1,\Fil^1}(\MM,\Acris) = \nf T^*_{\qst}(\MM)
\end{equation}
where the arrow is defined as follows: for a $\Sig$-linear map $f:\gM\ra\wh\Sig^{\ur}$, we consider $\wt f:\MM = S\otimes_{\sig,\Sig}\gM \ra \Acris$ obtained by $S$-linearly extending $\gM \xra f \wh\Sig^{\ur}\xra\sig \Acris$, where we view $S$ as a $\Sig$-algebra via $\sig:\Sig\ra S$. One can check that if $f$ respects $\vphi$, then $\wt f$ respects $\vphi_1$ and takes $\Fil^1\MM$ to $\Fil^1\Acris$. The arrow in \eqref{eqn:TSigTqstLatt} is defined by $f\mapsto \wt f$. One can show that this map $\mathfrak{T}$ is furthermore $\GKinfty$-equivariant and respects the natural embeddings of the source and the target into $V$ defined respectively in \eqref{eqn:TSigEmb} and \eqref{eqn:TqstEmb}.

Proposition~\ref{prop:Claim1} now directly follows from the lemma below:
\begin{lem}\label{lem:TongLemma}
Let $T$, $\gM$, and $\MM$ be as in \S\ref{par:MM}, and if $p=2$ then additionally assume that $\gM$ is $\vphi$-unipotent in the sense of \S\ref{par:DefEtPhiNilp}. 
Then the natural $\GK$-equivariant map  $\mathfrak{T}:T \ra T^*_{\qst}(\MM)$ in \eqref{eqn:TSigTqstLatt} is an isomorphism.
\end{lem}
Under the assumption of the lemma,  $T$ and $\nf T^*_{\qst}(\MM)$ define the same $\Zp$-lattice in $V$. Since we observed that $\nf T^*_{\qst}(\MM)$ and $\nf T^*_{\st}(\MM)$ define the same  the same $\Zp$-lattice in $V$ and the latter is $\GK$-stable, it follows that $T$ is $\GK$-stable. This proves Proposition~\ref{prop:Claim1}.
\subsection{Proof of Lemma \ref{lem:TongLemma}}\label{par:TongLemmaPf}
Since the injectivity of $\mathfrak{T}$ is clear, it is enough to show the surjectivity of $\mathfrak{T}$. By Nakayama's lemma, it is enough to show the surjectivity of  $\mathfrak{T} \otimes_{\Zp}\Fp:T\otimes_{\Zp}\Fp \ra \nf T^*_{\qst}(\MM)\otimes_{\Zp}\Fp$.

Put $\ol\gM:=\gM/p\gM $. Using the exactness assertion of Theorem~\ref{thm:FontaineEtPhiMod} and \cite[\S{B.}~Proposition~1.8.3]{fontaine:grothfest} we obtain: 
\begin{equation*}
T\otimes_{\Zp}\Fp \cong \Hom_{\Sig/p\Sig,\vphi}(\ol\gM,\fo\comp{\Eps^{\ur}}/p\fo\comp{\Eps^{\ur}}) \liso\Hom_{\Sig/p\Sig,\vphi}(\ol\gM,\wh\Sig^{\ur}/p\wh\Sig^{\ur})
\end{equation*}
where the arrow on the right is induced from the natural inclusion $\wh\Sig^{\ur}/p\wh\Sig^{\ur}\hra \fo\comp{\Eps^{\ur}}/p\fo\comp{\Eps^{\ur}}$.

Put $\ol\MM:=\MM/p\MM$. Let $\Fil^1\ol\MM\subset\ol\MM$ denote the image of $\Fil^1\MM$, and let $\ol{\vphi_1}:=\vphi_1\bmod{p\MM}$.  From the natural projection $\Acris \thra \Acris/p\Acris$, we obtain the following natural injective map:
\begin{multline}\label{eqn:RednTqst}
\nf T^*_{\qst}(\MM)\otimes_{\Zp}\Fp = \Hom_{S,\vphi_1,\Fil^1}(\MM,\Acris)\otimes_{\Zp}\Fp \\ \ra \Hom_{S/pS,\vphi_1,\Fil^1}(\ol\MM,\Acris/p\Acris)=:\nf T^*_{\qst}(\ol\MM).
\end{multline}
Therefore we obtain the natural map 
\begin{equation}\label{eqn:TSigTqstLattModp}
\ol{\mathfrak{T}}:T\otimes_{\Zp}\Fp \xlra{\mathfrak{T}_\gM \otimes_{\Zp}\Fp} \nf T^*_{\qst}(\MM)\otimes_{\Zp}\Fp \xlra{\text{\eqref{eqn:RednTqst}}} \nf T^*_{\qst}(\ol\MM).
\end{equation}
Alternatively, one can directly construct $\ol{\mathfrak{T}}$ in a similar manner to \eqref{eqn:TSigTqstLatt}; namely,
\begin{equation*}
\ol{\mathfrak{T}}:T\otimes_{\Zp}\Fp\cong\Hom_{\Sig/p\Sig,\vphi}(\ol\gM,\wh\Sig^{\ur}/p\wh\Sig^{\ur})\ra \Hom_{S/pS,\vphi_1,\Fil^1}(\ol\MM,\Acris/p\Acris)\end{equation*}
induced by $S/pS$-linear extension on the first arguments and by the map $\sig:\wh\Sig^{\ur}/p\wh\Sig^{\ur} \ra \Acris/p\Acris$ on the second arguments. The following lemma is the main step of the proof.

\begin{lemsub}\label{lem:TongLemmaModP}
Assume that $\gM$ is $\vphi$-unipotent in the sense of \S\ref{par:DefEtPhiNilp} if $p=2$. 
Then, the natural map $\ol{\mathfrak{T}}:T\otimes_{\Zp}\Fp \ra \nf T^*_{\qst}(\ol\MM)$ is injective.
\end{lemsub}

Let us explain how to deduce Lemma \ref{lem:TongLemma} from Lemma \ref{lem:TongLemmaModP}.
Granting the injectivity of $\ol{\mathfrak{T}}:T\otimes_{\Zp}\Fp\ra \nf T^*_{\qst}(\ol\MM)$, it follows from the construction of $\ol{\mathfrak{T}}$ that $\mathfrak{T} \otimes_{\Zp}\Fp:T\otimes_{\Zp}\Fp \ra \nf T^*_{\qst}(\MM)\otimes_{\Zp}\Fp$ is injective. But since the source and the target have the same $\Fp$-dimension, it has to be an isomorphism. 

We prove Lemma~\ref{lem:TongLemmaModP} for the rest of the section.
The proof can also be found in \cite[Proposition~1.2.7]{Kisin:2adicBT}. 
\parag[Proof of Lemma \ref{lem:TongLemmaModP}: the case $p>2$]\label{par:TongLemmaModPEasyCase}
The case $p>2$ essentially follows from the proof of \cite[Proposition 4.2.1]{Breuil:ApplicationNormFields}\footnote{Although \cite[Proposition 4.2.1]{Breuil:ApplicationNormFields} is only worked out when $e=1$, the argument works with little modification for any $e$.}. 
Using the same notation as in \S\ref{par:Asthat}, we let $\Fil^1\rep\subset\rep$ denote the kernel of the first projection $\rep\thra\ol K/(p)$.  Note that $\Acris/p\Acris$ is naturally isomorphic to the divided power envelop of $\rep$ with respect to $\Fil^1\rep$. (See \cite[Remark 3.20(8)]{Berthelot-Ogus} for the proof.)

Observe that the kernel of $\sig:\wh\Sig^{\ur}/p\wh\Sig^{\ur} \ra \Acris/p\Acris$ is principally generated by $u^e$, where $e$ is the ramification index of $ K $. So if $f\in\Hom_{\Sig/p\Sig,\vphi}(\ol\gM,\wh\Sig^{\ur}/p\wh\Sig^{\ur})$ is in the kernel of $\ol{\mathfrak{T}}$, then we have $f(x)\in u^e(\wh\Sig^{\ur}/p\wh\Sig^{\ur})$ for any $x\in\ol\gM$. Suppose $f(x)\in u^{e'}(\wh\Sig^{\ur}/p\wh\Sig^{\ur})$ for some $e'\geq e$. Since $\ol\gM$ is of $\PP$-height $\leqs 1$, there exists $y\in\ol\gM$ such that $\vphi_{\ol\gM}(\sig^*y) = u^{e}x$. Since $\PP(u)\bmod{p}$ is (a unit multiple of) $u^e$, we have
\begin{equation*}
f(x) = u^{-e} f\left(\vphi_{\ol\gM}(\sig^*y)\right) = u^{-e} \sig\left(f(y)\right) \in u^{e'p-e}(\wh\Sig^{\ur}/p\wh\Sig^{\ur}) \subset u^{2e'}(\wh\Sig^{\ur}/p\wh\Sig^{\ur}),
\end{equation*}
since we assumed that $p>2$. By iterating this process, we conclude that $f = 0$.

\parag[Non-example: the case $p=2$]
Before we present the proof of the case $p=2$, we give an example of non-$\vphi$-unipotent $\gM$ where the lemma fails to hold. Let $\gM\cong \Sig\tim\e$ with $\vphi_{\gM}(\sig^*\e) = \PP(u)\e$. Let $\ol\gM:=\gM/p\gM$ and $\ol\MM:=S/pS\otimes_{\sig,\Sig}\gM$. We show that $\ol{\mathfrak{T}}:T\otimes_{\Zp}\Fp\ra \nf T^*_{\qst}(\ol\MM)$ is the zero map, which in turn implies that $\mathfrak{T}\otimes_{\Zp}\Fp:T\otimes_{\Zp}\Fp \ra \nf T^*_{\qst}(\MM)\otimes_{\Zp}\Fp$ is the zero map. In particular, $\ol{\mathfrak{T}}$ cannot be injective (so $\mathfrak{T}$ cannot be an isomorphism).

Let $f\in\Hom_{\Sig/p\Sig,\vphi}(\ol\gM,\wh\Sig^{\ur}/p\wh\Sig^{\ur})$ be any element. Then we have
\begin{displaymath}
(f(\e))^p=(f(\e))^2 = \sig\left(f(\e)\right) = f\left(\vphi_{\ol\gM}(\sig^*\e)\right) = \alpha u^{e}\tim f(\e),
\end{displaymath}
where $\alpha u^e = \PP(u) \bmod{p}$ with $\alpha \in k\starr$.
If $f$ is non-zero then we have $f(\e)= \alpha\tim u^e$. On the other hand, $\sig:\wh\Sig^{\ur}/p\wh\Sig^{\ur} \ra \Acris/p\Acris$ maps any multiple of $u^e$ to $0$. This proves that $\ol{\mathfrak{T}}$ is the zero map.

\parag[Proof of Lemma \ref{lem:TongLemmaModP}: the case $p=2$]
Now, we handle the remaining case.\footnote{The author thanks Tong Liu for his advice.} Assume that $p=2$ and $\ol\gM$ is $\vphi$-unipotent  in the sense of \S\ref{par:DefEtPhiNilp}. Let $f\in\Hom_{\Sig/p\Sig,\vphi}(\ol\gM,\wh\Sig^{\ur}/p\wh\Sig^{\ur})$ be in the kernel of $\ol{\mathfrak{T}}$.
We set $\ol\gN:=f(\ol\gM) \subset \wh\Sig^{\ur}/p\wh\Sig^{\ur}$, which is a $\Sig/p\Sig$-submodule stable under the $p$th power map $\sig:\self{\wh\Sig^{\ur}/p\wh\Sig^{\ur}}$. This makes $\ol\gN$ into a $(\vphi,\Sig/p\Sig)$-module. Since we have the $\vphi$-compatible surjection $f:\ol\gM\thra \ol\gN$, it follows that $\ol\gN$ is of height $\leqs 1$; i.e., $u^{e}\ol\gN \subset \vphi_{\ol\gN}(\sig^*\ol\gN)$.

Since $f$ is in the kernel of $\mathfrak{T}_{\ol\gM}$, the same argument as \S\ref{par:TongLemmaModPEasyCase} implies that $\ol\gN \subset u^e(\wh\Sig^{\ur}/p\wh\Sig^{\ur})$. Using that $\vphi_{\ol\gN}$ is induced from the $p$th power map $\sig:\self{\wh\Sig^{\ur}/p\wh\Sig^{\ur}}$, we have $\vphi_{\ol\gN}(\sig^*\ol\gN) \subset u^{2e}(\wh\Sig^{\ur}/p\wh\Sig^{\ur})$, so $u^{e}\ol\gN \supset \vphi_{\ol\gN}(\sig^*\ol\gN)$. Since $\ol\gN$ is of $\PP$-height $\leqs1$, we obtain $\vphi_{\ol\gN}(\sig^*\ol\gN) = u^{e}\ol\gN$; i.e., $\ol\gN$ is of Lubin-Tate type  in the sense of \S\ref{par:DefEtPhiNilp}. But by the definition of $\vphi$-unipotent-ness, $\ol\gM$ does not admit any non-zero quotient of Lubin-Tate type. Therefore $\ol\gN=0$, so $f=0$.

\section{Positive characteristic analogue of crystalline deformation rings}\label{sec:GenLaffHartl}
In this section, we introduce a class of $\Gal(k\llpar u\rrpar^{\sep}/k\llpar u\rrpar)$-representation over some equi-characteristic local field which could be thought of as an analogue of crystalline representations, and develop deformation theory for them. Such representations are (implicitly) introduced by Genestier-Lafforgue \cite{Genestier-Lafforgue:FontaineEqChar}, and its torsion version also appeared in Abrashkin \cite{Abrashkin:EquiCharTorsionCrystalline}. A useful observation is that the linear algebra objects that give rise to such Galois representations have very similar structure to various $(\vphi,\Sig)$-modules we saw in Kisin theory. Considering the  norm field isomorphism $\GKinfty \cong \Gal(k\llpar u\rrpar^{\sep}/k\llpar u\rrpar)$  \cite{wintenberger:NormFiels}, it is not too surprising that the $\GKinfty$-deformation theory has an analogue in positive characteristic.

\subsection{Notations/Definitions}\label{subsec:Notations}
Let $\fo_0:=\Fq[[\pi_0]]$ be a complete discrete valuation ring of characteristic $p$. For this section, let $K:=k\llpar u \rrpar$ and $\fo_K:=k[[u]]$ where $k$ is a finite extension of $\Fq$. (So $K$ is no more a finite extension of $\Qp$.) We fix a finite map $\iota:\fo_0 \ra \fo_K$ over $\Fq$. Roughly speaking, $\fo_0$ will play the role of $\Zp$, and $\pi_0\in\fo_0$ will play the role of $p$.

Put $\GK:=\Gal(K^{\sep}/K)$. We would like to study a certain class of $\GK$-representations over $\fo_0$, $\Frac(\fo_0)$, or finite algebras thereof. 
It is  defined in terms of linear-algebraic objects called \emph{(effective) local shtukas} over $\fo_K$, which we introduce below. Local shtukas have many analogous features to $(\vphi,\Sig)$-modules of finite height in Kisin theory, so we use similar notations to Kisin theory to emphasize the analogy.

Let  $\Sig := \fo_K[[\pi_0]] $ and $\fo_\Eps:= K[[\pi_0]]$. We define a partial $q$-Frobenius endomorphism $\sig$ for each of these rings  so that it acts as the $q$th power map on $K$ and $\sig(\pi_0)=\pi_0$. This $\sig$ lifts the $q$th power map modulo $\pi_0$, and fixes  $\fo_0$. We also set  $\Eps := K\llpar \pi_0 \rrpar$ and extend $\sig$ on $\Eps$. Then $\sig$ fixes $\Frac(\fo_0)$

Let $u_0:=\iota(\pi_0)\ne0$ where $\iota:\fo_0\ra\fo_K$ is the map we fixed earlier.  Put $\PP(u):=\pi_0-u_0\in\Sig$ and let $e:=\ord_u(u_0)$. Clearly we have  $\Sig/(\PP(u)) \cong \fo_K$, which is a totally ramified ring extension of $k[[\pi_0]]$. This shows  that $\PP(u)$ is a $\Sig\starr$-multiple of some Eisenstein polynomial in  $k[[\pi_0]][u]$ with degree $e$. 

\parag
An \emph{\'etale $(\vphi,\fo_\Eps)$-module} (or simply, an  \emph{\'etale $\vphi$-module}) is a finitely generated  $\fo_\Eps$-module $M$ equipped with an $\fo_\Eps$-linear isomorphism $\vphi_M:\etfstr M$. We say an \'etale $\vphi$-module $M$ is \emph{free} (respectively, \emph{torsion}) if the underlying $\fo_\Eps$-module is free (respectively, $\pi_0^\infty$-torsion). We let $\etphimod{\fo_\Eps}$ denote the category of \'etale $(\vphi,\fo_\Eps)$-modules with the obvious notion of morphisms, and let $\fretphimod{\fo_\Eps}$ and $\toretphimod{\fo_\Eps}$ respectively denote the full subcategories of free and $\pi_0^\infty$-torsion \'etale $\vphi$-modules. We let $\etphimod{\fo_\Eps}[\ivtd{\pi_0}]$ denote the ``isogeny category'' for  $\fretphimod{\fo_\Eps}$; i.e. the categories defined by formally inverting the multiplication by $\pi_0$. 

There exist natural notions of  subquotient, direct sum, $\otimes$-product, and internal hom for \'etale $\vphi$-modules. We respectively define the duals for free and torsion \'etale $\vphi$-modules by the $\fo_\Eps$-linear duals and the Pontryagin duals with the induced $\vphi$-structures.
 
The main motivation for considering \'etale $\vphi$-modules is that we have a natural analogue of Theorem~\ref{thm:FontaineEtPhiMod}. Let $\wh\fo_{\Eps^{\ur}}:=K^{\sep}[[\pi_0]]$, and we let $\GK$ act on it through the coefficients, and define the partial $q$-Frobenius endomorphism $\sig$ so that it acts as the $q$th power map on $K^{\sep}$ and $\sig(\pi_0)=\pi_0$.  Then the following construction
 \begin{equation}\label{eqn:EtPhiModTEpsEqChar}
 \nf T_\Eps (M):=(M\otimes_{\fo_\Eps}\wh\fo_{\Eps^{\ur}})^{\vphi=1} \qquad \textrm{for }M\in\etphimod{\fo_\Eps},\\
 \end{equation}
defines an equivalence of categories between $\etphimod{\fo_\Eps}$ and the category of finitely generated $\fo_0$-module with continuous $\GK$-action. One can define the quasi-inverse $ \nf D_\Eps$ in a similar fashion to \eqref{eqn:EtPhiModDEps}. Furthermore, they respect all the natural operations such as $\otimes$-product, duality, etc, and they preserve rank and length whenever applicable. The proof is identical to the proof of Theorem~\ref{thm:FontaineEtPhiMod}. (See \cite[\S5.1]{Kim:Thesis} for the full proof.)

\begin{defnsub}\label{def:LTchar}
Consider the following \'etale $\vphi$-module  $M_\LT:=\fo_\Eps\tim\e$ equipped with $\vphi_{M_\LT}(\sig^*\e) = \PP(u)\iv\e$.
Let $\chi_\LT:\GK\ra\fo_0\starr$ denote the character that defines the $\GK$-action on $\nf T_\Eps(M_\LT)$. For any $\fo_0[\GK]$-module $V$, we let $V(n)$ the $\fo_0[\GK]$-module whose $\GK$-action is twisted by $\chi_\LT^n$. 
\end{defnsub}
This character $\chi_\LT$ is equivalent to the character obtained from the $\pi_0$-adic Tate module of the Lubin-Tate formal $\fo_0$-module over $\fo_K$. See \cite[\S7.3.7]{Kim:Thesis} for the proof. Note that when $K$ is a finite extension of $\Qp$, we can obtain $\chi_{\cyc}|_{\GKinfty}$ from the \'etale $\vphi$-module defined analogously as above. Compare with Example~\ref{exa:repfinht}\ref{exa:repfinht:cycchar}.

\parag
For a non-negative integer $h$, an \emph{effective local shtuka (over $\fo_K$) of height $\leqs h$} is a finite free $\Sig$-module $\gM$ equipped with an $\Sig$-linear morphism $\vphi_\gM:\fstr \gM$ such that $\coker (\vphi_\gM)$ is killed by $\PP(u)^h$. We let $\phimodh\Sig$ denote the category of effective local shtukas of height $\leqs h$ with the obvious notion of morphisms. We let $\phimodh\Sig[\ivtd{\pi_0}]$ denote the ``isogeny category'' for  $\phimodh\Sig$; i.e. the categories defined by formally inverting the multiplication by $\pi_0$.

The original definition of effective local shtuka (over $\fo_K$) requires  $\coker (\vphi_\gM)$ to be flat over $\fo_K$, but this is automatic by the same argument as the proof of Lemma~\ref{lem:FlatCokerVphi}\eqref{lem:FlatCokerVphi:Flat}. Note that effective local shtukas can be defined over any $\fo_0$-scheme (not just over $\Spf\fo_K$), and there are more general objects called \emph{local shtukas} which are defined by allowing $\vphi_\gM$ to have a pole at $\PP(u)$. See \cite[Definition~0.1]{Genestier-Lafforgue:FontaineEqChar} or \cite[Definition~2.1.1]{hartl:period} for more general definition.

Since $\PP(u)$ is a unit in $\fo_\Eps$, the scalar extension $\gM\otimes_\Sig\fo_\Eps$ is naturally an \'etale $\vphi$-module. So we can associate a $\GK$-representation to an effective local shtuka $\gM$ of height $\leqs h$ as follows:
\begin{equation}\label{eqn:TSigLatticeEqChar}
\nf T^{\leqs h}_\Sig(\gM):=\nf T_\Eps(\gM\otimes_\Sig \fo_\Eps)(h)
\end{equation}

We state the following fundamental and non-trivial result on this functor $\nf T^{\leqs h}_\Sig$. Compare with Theorem~\ref{thm:main}.
\begin{propsub}\label{prop:main}\hfill
\begin{enumerate}
\item \label{prop:main:FFaith}
The functor $\nf T^{\leqs h}_\Sig$ from the category of effective local shtukas of height $\leqs h$ to the category of $\fo_0$-representations of $\GK$ is fully faithful. 
\item \label{prop:main:ClassifLatt}
Let  $V:=\nf T^{\leqs h}_\Sig(\gM)[\ivtd{\pi_0}]$, then for any $\GK$-stable $\fo_0$-lattice $T'\subset V$ there exists an effective local shtuka $\gM'$ of height $\leqs h$ such that $T'\cong\nf T^{\leqs h}_\Sig(\gM')$.
\end{enumerate}
\end{propsub}
\begin{proof}
The proof of \eqref{prop:main:FFaith} is very similar to the proof of its $p$-adic analogue \cite[Proposition~2.1.12]{kisin:fcrys}, except that one needs to work with ``weakly admissible isocrystals with Hodge-Pink structure'' instead of filtered $\vphi$-modules, and apply \cite[Th\'eor\`eme~7.3]{Genestier-Lafforgue:FontaineEqChar} instead of  \cite[Lemma~1.3.13]{kisin:fcrys}. The detail is worked out in \cite[Theorem~5.2.3]{Kim:Thesis}.

The claim \eqref{prop:main:ClassifLatt} easily follows from \cite[Lemme~2.3]{Genestier-Lafforgue:FontaineEqChar} and the equivalence of category $\nf T_\Eps$. (See \cite[Lemma~2.1.15]{kisin:fcrys} for the proof of its $p$-adic analogue.)
\end{proof}

By \emph{$\fo_0$-lattice $\GK$-representation}, we mean a finite free $\fo_0$-module equipped with continuous $\GK$-action. By \emph{$\pi_0^\infty$-torsion $\GK$-representation}, we mean a finitely generated $\pi_0^\infty$-torsion $\fo_0$-module equipped with continuous $\GK$-action. 
\begin{defnsub}\label{def:EqCharCrys}
Let $h$ be a non-negative integer.
An $\fo_0$-lattice $\GK$-action $T$ is called \emph{of height $\leqs h$} if there exists an effective local shtuka $\gM$ of height $\leqs h$ such that $T\cong \nf T^{\leqs h}_\Sig(\gM)$. A continuous $\GK$-representation $V$ over $\Frac(\fo_0)$ is called \emph{of height $\leqs h$} if it admits a $\GK$-stable $\fo_0$-lattice $T\subset V$ which is of height $\leqs h$; or equivalently by Proposition~\ref{prop:main}\eqref{prop:main:ClassifLatt}, any $\GK$-stable $\fo_0$-lattice $T\subset V$ is of height $\leqs h$. A $\pi_0^\infty$-torsion $\GK$-representation $\ol T$  is called \emph{of height $\leqs h$} if there exist $\fo_0$-lattice $\GK$-representations $T'\subset T $ of height $\leqs h$ such that $\ol T\cong T/T'$.
\end{defnsub}

Note that $\chi_\LT^r$ for $0\leqs r\leqs h$ is of height $\leqs h$, because by definition of $\chi_\LT$ (Definition~\ref{def:LTchar}) we have $\chi_\LT^r\sim \nf T_\Sig^{\leqs h}(\Sig(h-r))$ where $\Sig(h-r):=\Sig\tim\e$ with $\vphi_{\Sig(h-r)}(\sig^*\e) = \PP(u)^{h-r}\e$. (\emph{cf.} Example~\ref{exa:repfinht}\eqref{exa:repfinht:cycchar}.) It is not difficult to show that any unramified $\GK$-representation is of height $\leqs 0$ (hence, of height $\leqs h$ for any non-negative $h$). See, for example,  \cite[Proposition~5.2.10]{Kim:Thesis} for the proof.

Proposition~\ref{prop:main} suggests that $\GK$-representations of height $\leqs h$ should enjoy similar properties to those enjoyed by $\GKinfty$-representation of height $\leqs h$ in the setting of Kisin theory. On the other hand, $\GK$-representations of height $\leqs h$ can also be regarded as a positive characteristic analogue of crystalline representations with Hodge-Tate weights in $[0,h]$, for the following reasons.\footnote{We remark that in positive characteristic $K_\infty:=K(\sqrt[q^\infty]{u})$ is a purely inseparable field extension of $K$, so the gap between $\GK$ and $\GKinfty$ collapses.} Effective local shtukas arise naturally by completing global objects at ``places of good reduction'' such as $t$-motives, elliptic sheaves, and  Drinfeld shtukas. (See \cite[Example~2.1.2]{hartl:period} for more details.) It has been known for experts that there exists a natural anti-equivalence of categories between the category of effective local shtukas of height $\leqs 1$ and the category of strict $\pi_0$-divisible groups (using the terminology of \cite{fal02}), and if $\gM$ is the effective local shtuka of height $\leqs1$ which corresponds to a strict $\pi_0$-divisible group $G$ then $\left(\nf T^{\leqs1}_\Sig(\gM)\right)^*(1)$ is naturally isomorphic to the $\pi_0$-adic Tate module of $G$. This is generalized by Hartl \cite[\S3]{hartl:dictionary}  to any effective local shtukas.\footnote{Note that not all the $\pi_0$-divisible groups come from effective local shtukas -- the $\pi_0$-divisible groups that come from effective local shtukas are called \emph{divisible Anderson modules} in \cite[\S3]{hartl:dictionary}.}  See \cite[\S7.3]{Kim:Thesis} for the proof.

\parag
For a non-negative integer $h$, a  \emph{torsion shtuka of height $\leqs h$} is a finitely generated $\pi_0^\infty$-torsion $u$-torsionfree $\Sig$-module $\gM$ equipped with an $\Sig$-linear morphism $\vphi_\gM:\fstr \gM$  such that $\coker (\vphi_\gM)$ is killed by $\PP(u)^h$. We let $(\Mod/\Sig)^{\leqs h}$ denote the category of torsion shtukas of height $\leqs h$ with the obvious notion of morphisms. There exist natural notions of  subquotient, direct sum, $\otimes$-product for torsion shtukas. We can also define duality in the same way as in \S\ref{par:CarDual}. 

Let  $\ol\gM$  be a torsion shtuka of height $\leqs h$. Since $\ol\gM\otimes_\Sig\fo_\Eps$ is a $\pi_0^\infty$-torsion \'etale $\vphi$-module, one can associate a $\pi_0^\infty$-torsion $\GK$-representation $\nf T^{\leqs h}_\Sig (\ol\gM):=\nf T_\Eps(\ol\gM\otimes_\Sig\fo_\Eps)(h)$.

The same proof as \cite[Lemma~2.3.4]{kisin:fcrys} shows that any torsion shtuka of height $\leqs h$ can be obtained as the cokernel of an isogeny of effective local shtukas of height $\leqs h$. Using the ``exactness'' of $\nf T^{\leqs h}_\Sig$ which can be easily verified, we obtain the following lemma.
\begin{lemsub}
Let $\ol T$ be a finite $\pi_0^\infty$-torsion $\GK$-representation. Then $\ol T$ is of height $\leqs h$ in the sense of Definition~\ref{def:EqCharCrys} if and only if there exists a torsion shtuka $\ol\gM$ of height $\leqs h$ such that $\ol T\cong \nf T^{\leqs h}_\Sig(\ol\gM)$.
\end{lemsub}

\parag
\label{par:LimThmEqChar}
We finally remark that the analogue of the ``limit theorem'' (Proposition~\ref{prop:LimitThm}) holds; i.e., an $\fo_0$-lattice $\GK$-representation obtained as a limit of $\pi_0^\infty$-torsion $\GK$-representation of height $\leqs h$ is again of height $\leqs h$ (as  an $\fo_0$-lattice $\GK$-representation). The proof is ``identical'' to the proof of Proposition~\ref{prop:LimitThm}.

\subsection{Deformation theory}
Let $\F$ be a finite extension of $\Fq$ (which is the residue field of $\fo_0$), and $\bar\rho:\GK\ra\GL_d(\F)$ a representation. Let $\fo$ be a $p$-adic discrete valuation ring with residue field $\F$, and put $F:=\Frac(\fo)$. Let $\art{\fo}$ be the category of artin local $\fo$-algebras $A$ whose residue field is $\F$ (via the natural map), and similarly let $\arhat{\fo}$ be the category of complete local noetherian $\fo$-algebras with residue field  $\F$.

Let $D,D^\Box:\arhat\fo \ra \Sets$ be the deformation functor and framed deformation functor for $\bar\rho$. Since the tangent spaces of these functors are infinite-dimensional (as explained in \S\ref{subsec:RepbilitySetup}), they cannot be represented by complete local noetherian $\fo$-algebras.

We say that a deformation $\rho_A$ over $A\in\art\fo$ is \emph{of height $\leqs h$} if it is a $\pi_0^\infty$-torsion $\GK$-representation of height $\leqs h$ as a $\pi_0^\infty$-torsion $\GK$-representation; or equivalently, if there exists $\ol\gM\in\Modh$ and an isomorphism $\Th_\Sig(\gM)\cong\rho_A$ as $\fo_0[\GK]$-modules. For $A\in\arhat\fo$, we say that $\rho_A$ is \emph{of height $\leqs h$} if $\rho_A\otimes A/\m_A^n$ is a deformation of height $\leqs h$ for each $n$. When $A\in\art\fo$, both definitions are compatible because the condition of being height $\leqs h$ is closed under subquotient. When $A$ is finite flat over $\fo_0$, a deformation $\rho_A$ over $A$ is of height $\leqs h$ if and only if $\rho_A$ is of height $\leqs h$ as a $\fo_0$-lattice $\GK$-representation, as remarked in \S\ref{par:LimThmEqChar}.

Let $D^{\leqs h}\subset D$ and $D^{\Box,\leqs h}\subset D^\Box$ respectively denote subfunctors of deformations and framed deformations of height $\leqs h$. In this setting, we have the analogue of Theorem~\ref{thm:Representability}:
\begin{thmsub}\label{thm:RepresentabilityEqChar}
The functor $D^{\leqs h}$ always has a hull. If $\End_{\GK}(\bar\rho) \cong \F$ then $D^{\leqs h}$ is representable (by $R^{\leqs h}\in\arhat\fo$). The functor $D^{\Box,\leqs h}$ is representable (by $R^{\Box,\leqs h}\in\arhat\fo$) with no assumption on $\bar\rho$.
Furthermore, the natural inclusions $D^{\leqs h}\hra D$ and $D^{\Box,\leqs h}\hra D^\Box$ of functors are relatively representable by surjective maps in $\arhat\fo$.
\end{thmsub}
We call $R^{\Box,\leqs h}$ the \emph{universal framed deformation ring of height $\leqs h$} and $R^{\leqs h}$ the  \emph{universal deformation ring of height $\leqs h$} if it exists. 

The proof of Theorem~\ref{thm:Representability} can easily be adapted. The main difficulty is to show the finiteness of the tangent space, but the same proof of Proposition~\ref{prop:RepresentabilityH3} works, if we replace $\Sig$, $\fo_\Eps$, $\Modh$ by their positive characteristic  analogues as introduced in \S\ref{subsec:Notations} and the $p$th power map is replaced by the $q$th power map in suitable places.

\parag[Torsion shtukas with coefficients]
\label{par:ModFIEqChar}
Let $A$ be a $\pi_0$-adically separated and complete topological $\fo_0$-algebra, (for example, finite $\fo_0$-algebras or any $\fo_0$-algebra $A$ with $\pi_0^N\cdot A = 0$ for some $N$). Set $\Sig_A:=\Sig\wh\otimes_{\fo_0}A:=\varprojlim_\alpha \Sig\wh\otimes_{\fo_0}A/I_\alpha$ where $I_\alpha$ form a basis of open ideals in $A$. We define a ring endomorphism $\sig:\self{\Sig_A}$ by $A$-linearly extending the Frobenius endomorphism $\sig$. We also  put $\fo_{\Eps,A}:=\fo_\Eps\wh\otimes_{\fo_0}A :=\varprojlim_\alpha\fo_\Eps\wh\otimes_{\fo_0}A/I_\alpha$and similarly define $\sig:\self{\fo_{\Eps,A}}$.

Let $\ModFIh A$ be the category of finite free $\Sig_A$-modules $\gM_A$ equipped with a $\Sig_A$-linear map $\vphi_{\gM_A}:\sig^*(\gM_A)\ra\gM_A$ such that $\PP(u)^h$ annihilates $\coker(\vphi_{\gM_A})$. Similarly, let $\ModFIet A$ be the category of finite free $\fo_{\Eps,A}$-modules $M_A$ equipped with a $\fo_{\Eps,A}$-linear isomorphism $\vphi_{M_A}:\sig^*(M_A)\riso M_A$. If $A$ is finite artinean, then one can associate to these objects $\GK$-representations over $A$,  and show that $\nf T_\Eps$, defined in \eqref{eqn:EtPhiModTEpsEqChar}, induces an equivalences of categories between $\ModFIet A$ and the category of $\GK$-representations over $A$.\footnote{The relevant freeness follows from length consideration.}

\subsection{Moduli of torsion shtukas of height $\leqs h$}\label{subsec:DefnOfResolEqChar}
Let $h$ be a positive integer. Consider a deformation $\rho_R
$ of $\bar\rho_\infty$ over $R\in\arhat\fo$ which is of height $\leqs h$ (i.e. $\rho_R\otimes_R R/\m_R^n$ is of height $\leqs h$ for each $n$). The main examples to keep in mind are universal framed deformation of height $\leqs h$. 

Put $M_R:=\varprojlim M_n$ where $M_n\in\ModFIet{R/\m_R^n}$ is such that $\nf T_\Eps(M_n)(h)\cong \rho_R\otimes_R R/\m_R^n$ for each $n$. For any $R$-algebra $A$, we view $M_R\otimes_R A$ as an \'etale $\vphi$-module by $A$-linearly extending $\vphi_{M_R}$. 

For a complete local noetherian ring $R$, let $\aug R$ be the category of pairs $(A,I)$ where $A$ is an $R$-algebra and $I\subset A$  is an ideal with $I^N=0$ for some $N$ which contains $\m_R\tim A$. Note that an artin local $R$-algebra $A$ can be viewed as an element in $\aug\fo$ by setting $I:=\m_A$. A morphism $(A,I)\ra(B,J)$ in $\aug R$ is an $R$-morphism $A\ra B$ which takes $I$ into $J$. 
We define a functor $D^{\leqs h}_{\Sig,\rho_R}:\aug R \ra \Sets$ by putting $D^{\leqs h}_{\Sig,\rho_R}(A,I)$ the set of $\vphi$-stable $\Sig_A$-lattices  in $M_R\otimes_RA$ which are torsion shtukas of height $\leqs h$. 

With this setting, we have an analogue of Proposition~\ref{prop:ResolDefor}.
\begin{propsub}\label{prop:ResolDeforEqChar}
The functor $D^{\leqs h}_{\Sig,\rho_R}$ can be represented by a projective $R$-scheme $\GRh_{\rho_R}$ and a $\Sig\otimes_{\Zp}\OO_{\GRh_{\rho_R}}$-lattice $\nf\gM^{\leqs h}_{\rho_R} \subset M_R \otimes_{R} \OO_{\GRh_{\rho_R}}$. (We call  $\nf\gM^{\leqs h}_{\rho_R}$ a \emph{universal $\Sig$-lattice of height $\leqs h$} for $\rho_R$.) Moreover, the formation of $\GRh_{\rho_R}$ and  $\nf\gM^{\leqs h}_{\rho_R}$ commute with scalar extension $R\ra R'$.
\end{propsub}
Indeed, the proof of its $p$-adic analogue (Proposition~\ref{prop:ResolDefor}) works verbatim in the positive characteristic  setting. The proof is also worked out in Proposition~11.1.9, Corollary~11.1.11of \cite{Kim:Thesis} for the positive characteristic setting.

The discussions in \S\ref{sec:GenFibers} also applies to this positive characteristic setting. For example, the structure morphism $\GRh_{\rho_R}\ra \Spec R$ becomes an isomorphism after inverting $\pi_0$ (Proposition~\ref{prop:GenIsom}, also using Proposition~\ref{prop:main}\eqref{prop:main:FFaith}); $R^{\Box,\leqs h}[\ivtd{\pi_0}]$ is formally smooth (Proposition~\ref{prop:FormalSm}); and the condition of having ``$\vphi$-nilpotent'' (respectively, ``$\vphi$-unipotent'') local shtuka model defines a union of connected components in $R^{\Box,\leqs h}[\ivtd{\pi_0}]$  (Proposition~\ref{prop:DisjOrdSS}).

When $\bar\rho$ is $2$-dimensional and $h=1$, one can define the quotient $R^{\Box,\bfv}$ of $R^{\Box,\leqs 1}$ in the similar fashion to \S\ref{subsec:HodgeType}.\footnote{In the $p$-adic case $\Spec R^{\Box,\bfv}_\infty[\ivtd p]$ is a union of connected components of $\Spec R_\infty^{\Box,\leqs 1}[\ivtd p]$. But in the positive characteristic setting, the author could only prove this when $K$ is separable over $k\llpar u_0\rrpar$. See \cite[Proposition~11.3.7]{Kim:Thesis}.}  As in Remark~\ref{rmk:Rv}, for any $\Frac(\fo_0)$-algebra $A$, an $A$-point of $R^{\Box,\leqs 1}$ factors through $R^{\Box,\bfv}$ if and only if $I_K$ acts via $\chi_\LT$ on the determinant of the corresponding  framed deformations. Then the direct analogue of the connected component result (Proposition~\ref{prop:ConndOrdNonOrd}) holds for the positive characteristic deformation ring $R^{\Box,\bfv}[\ivtd{\pi_0}]$. Furthermore, the argument in \cite[\S3]{kisin:pstDefor} can be adapted to show that $R^{\Box,\bfv}[\ivtd{\pi_0}]$ is equi-dimensional of dimension $4+[K:\Fq\llpar u_0\rrpar]$, which is strongly analogous to the $p$-adic case. (Compare with \cite[Theorem~3.3.8]{kisin:pstDefor} and \cite[\S11.3.17]{Kim:Thesis}.)
\bibliography{bib}
\bibliographystyle{amsalpha}

\end{document}